\documentclass[aihp]{imsart}

\RequirePackage{amsthm,amsmath,amsfonts,amssymb}
\RequirePackage[numbers]{natbib}
\RequirePackage[colorlinks,citecolor=blue,urlcolor=blue]{hyperref}
\RequirePackage{graphicx}

\usepackage{pgf,tikz,pgfplots}
\usepackage{caption,subcaption}
\pgfplotsset{compat=1.15}
\usetikzlibrary{arrows}

\usepackage{mathrsfs}

\startlocaldefs
\theoremstyle{plain}
\newtheorem{theorem}{Theorem}[section]
\newtheorem{lemma}[theorem]{Lemma}
\newtheorem{corollary}[theorem]{Corollary}
\newtheorem{proposition}[theorem]{Proposition}
\theoremstyle{remark}
\newtheorem{definition}[theorem]{Definition}
\newtheorem{remark}[theorem]{Remark}

\def\ed{\stackrel{d}{=}}
\def\convd{\stackrel{d}{\longrightarrow}}

\def\E{ {\mathbb E }}
\def\P{ {\mathbb P }}

\DeclareMathOperator{\Ima}{Im}

\newcommand{\diam}{\operatorname{diam}}

\newcommand{\rG}{\mathrm{G}}

\newcommand{\ftau}{\boldsymbol{\tau}}

\newcommand{\dd}{{\rm d}}
\newcommand{\un}{{\bf 1}}
\newcommand{\I}[1]{\boldsymbol{1}_{\{#1\}}}
\newcommand{\cq}{$\hfill \square$}

\newcommand{\dl}[1][]{d^{\mathsf{L}}_{#1}}
\newcommand{\dlt}[1][]{\widetilde{d}^{\mathsf{L}}_{#1}}
\newcommand{\dtree}[1][]{d^{\mathsf{T}}_{#1}}
\newcommand{\dv}[1][]{d^{\mathsf{V}}_{#1}}
\newcommand{\dvt}[1][]{\widetilde{d}^{\mathsf{V}}_{#1}}
\newcommand{\tree}[1][]{\mathscr{T}_{#1}}
\newcommand{\lop}[1][]{\mathscr{L}_{#1}}
\newcommand{\vern}[1][]{\mathscr{V}_{#1}}
\newcommand{\lopt}[1][]{\widetilde{\mathscr{L}}_{#1}}
\newcommand{\vernt}[1][]{\widetilde{\mathscr{V}}_{#1}}

\endlocaldefs

\begin{document}


\definecolor{qqqqff}{rgb}{0,0,1} 
\definecolor{qqffqq}{rgb}{0,0.6,0}
\definecolor{qqffff}{rgb}{0,1,1}
\definecolor{ffdxqq}{rgb}{1,0.8431372549019608,0}
\definecolor{ffqqqq}{rgb}{1,0,0}
\definecolor{ffqqff}{rgb}{1,0,1}
\definecolor{xdxdff}{rgb}{0.49019607843137253,0.49019607843137253,1}
\definecolor{ududff}{rgb}{0.30196078431372547,0.30196078431372547,1}
\definecolor{ffudud}{rgb}{1,0.30196078431372547,0.30196078431372547}
\definecolor{mystery}{rgb}
{0.7,0.7,0}

\begin{frontmatter}

\title{Convergences of looptrees coded by excursions}
\runtitle{Convergences of looptrees coded by excursions}

\begin{aug}
\orcid{0009-0007-6561-2939}               
\author[A]{\inits{RK}\fnms{Robin}~\snm{Khanfir}\ead[label=e1]{robin.khanfir@mcgill.ca}}
\address[A]{Department of Mathematics and Statistics, McGill University, 845 Rue Sherbrooke O, Montreal, Quebec, Canada\printead[presep={,\ }]{e1}}

\end{aug}

\begin{abstract}
In order to study convergences of looptrees, we construct continuum trees and looptrees from real-valued càdlàg functions without negative jumps called excursions. We then provide a toolbox to manipulate the two resulting codings of metric spaces by excursions and we formalize the principle that jumps correspond to loops and that continuous growths correspond to branches. Combining these codings creates new metric spaces from excursions that we call vernation trees. They consist of a collection of loops and trees glued along a tree structure so that they unify trees and looptrees. We also propose a topological definition for vernation trees, which yields what we argue to be the right space to study convergences of looptrees. However, those first codings lack some functional continuity, so we adjust them. We thus obtain several limit theorems. Finally, we present some probabilistic applications, such as proving an invariance principle for random discrete looptrees.
\end{abstract}

\begin{abstract}[language=french]
Afin d'étudier les convergences d'arbres à boucles, on construit des arbres et des arbres à boucles continus à partir d'excursions, c'est-à-dire des fonctions à valeurs réelles, càdlàg, et dont les sauts sont positifs. On fournit plusieurs outils pour manipuler au mieux les deux codages résultants d'espaces métriques par des excursions. Ceux-ci sont régis par le principe que les boucles correspondent aux sauts de l'excursion et que les branches correspondent aux portions de croissance continue. En combinant les deux codages, on définit de nouveaux espaces métriques, généralisant à la fois arbres et arbres à boucles car formés par une famille de boucles et de branches greffées entre elles selon une structure arborescente, que l'on nomme arbres à vernation. On propose aussi une caractérisation topologique des limites des arbres à vernation. Cependant, nos codages se comportent mal avec la convergence fonctionnelle des excursions codantes, donc on les ajuste pour pouvoir répondre à notre objectif initial. Enfin, on donne quelques applications probabilistes, comme la preuve d'un principe d'invariance pour des arbres à boucles discrets aléatoires.
\end{abstract}



\end{frontmatter}


\section{Introduction and presentation of the results}
\label{intro}

\subsection{Background and motivations}

Informally, a looptree consists of a collection of loops tangently glued together along some genealogical structure. Perhaps more clearly, one can associate a looptree with a finite ordered tree by replacing each vertex with a circle of length proportional to its degree and each edge with a unique point of tangency between the corresponding circles, thereby preserving the ordered tree structure. An example will be given in Figure~\ref{discrete_looptree_associated} in Section~\ref{subsection_discrete_loop}. Ever since their introduction by Curien \& Kortchemski~\cite{curien2014}, looptrees have generated a growing interest thanks to their natural appearances in some random geometric models and because they can provide useful tools or insights to study other objects. Let us mention some of them. To begin with, when a sequence of trees admits a scaling limit and if the maximum degree becomes negligible against the height, then the scaling limit may be a continuum tree whose all branch points are of infinite degree. In this situation, it is then difficult to recover the limiting joint distribution of the scaled degrees from that object. However, a scaling limit of the associated looptrees would not only give such information (merely via the lengths of the loops) but also the whole asymptotic degree structure. This is exactly what happens for large critical Galton--Watson trees whose offspring distribution belongs to the domain of attraction of an $\alpha$-stable law, with $\alpha\in(1,2)$. While Duquesne~\cite{duquesne_contour_stable} showed that these random trees converge, after suitable scaling, towards the $\alpha$-stable tree introduced in \cite{legall1998,levytree_DLG}, Curien \& Kortchemski proved that the rescaled associated looptrees converge towards a random compact metric space they called the $\alpha$-stable looptree. There are even more striking cases where the associated looptrees admit a non-degenerate scaling limit but the trees themselves do not, so they are an interesting means to give sense to scaling limits of highly dense trees. An instance of such a situation is studied by Curien, Duquesne, Kortchemski \& Manolescu~\cite{loopLPAM} with the model of random trees built by linear preferential attachment, which was introduced and popularized in \cite{Szymanski1987,barabasi,bollobas}. The scaling limit for this model is called the Brownian looptree and is distinct from the stable looptrees.

In addition to their help in studying trees, looptrees spontaneously appear as scaling limits of other random structures. For example, random Boltzmann dissections of a regular polygon, that were introduced in \cite{lamination}, have a similar geometry as looptrees in some regimes. Indeed, Curien \& Kortchemski~\cite{curien2014} proved that their scaling limit as the number of sides of the polygon tends to infinity is the $\alpha$-stable looptree, depending on the regime. Let us mention links between looptrees and dissections were also detected in \cite{haas}. Several instances of looptrees can be found within random planar maps too. The goal of this active area of research is to understand the universal large-scale properties of graphs or their proper embeddings in the two-dimensional sphere. We refer to \cite{abraham_surveymaps,miermont_surveymaps} for surveys of this field. Curien \& Kortchemski~\cite{perco_curien} considered boundaries of percolation clusters on the uniform infinite planar triangulation introduced in \cite{UIPT}. These boundaries are almost looptrees and the authors showed that in the critical case of the percolation, their scaling limit is the $3/2$-stable looptree. Similarly, Kortchemski \& Richier~\cite{KR2020,richier18} discussed asymptotics for the boundaries of Boltzmann planar maps conditioned on having a large perimeter. Once again, stable looptrees arise as scaling limits in some regimes. 

Nevertheless, the most important application of looptrees may be their substantial connections with the scaling limits of random planar maps. The bijections of Bouttier, Di Francesco \& Guitter \cite{BDFG} and of Janson \& Stef\'{a}nsson \cite{stefansson} yield a one-to-one correspondence between planar maps and some labeled trees (where each vertex is equipped with an integer), which can be reformulated into a bijection between planar maps and some labeled looptrees. Looptrees seem to be the right objects to consider here because the scaling limits of random planar maps with large faces introduced by Le Gall \& Miermont~\cite{LGM11} are implicitly constructed from a Gaussian field on stable looptrees. Furthermore, Marzouk~\cite{marzouk19} obtained limit theorems for planar maps without a strong control on the corresponding trees, but by observing their Lukasiewicz walks instead. As we will see in Section~\ref{subsection_discrete_loop}, the Lukasiewicz walk directly codes the geometry of the associated looptree. For more detailed discussions on the relations between maps and looptrees, we refer to Marzouk~\cite{marzouk_similaire}.
\medskip

The purpose of the present work is to build a framework and to provide a toolbox to construct, manipulate, and demonstrate convergences involving looptrees under the most general setting possible. The first issue is how to define compact continuum fractal looptrees that may appear as scaling limits. Although compact continuum fractal trees have been well-known since the introduction of the Brownian Continuum Random Tree (CRT) by Aldous~\cite{aldousI}, mimicking the discrete setting to define continuum looptrees associated with them is not so easy. On the one hand, it could be not clear how to choose the lengths of the loops while ensuring the compactness of the resulting space. On the other hand, when the branch points are dense in the tree, it might be possible that two different cycles in the associated looptree would never share a common point. It is indeed what happens for the $\alpha$-stable looptree $\mathscr{L}_\alpha$ of Curien \& Kortchemski~\cite{curien2014}, with any $\alpha\in(1,2)$. They avoided the above difficulties by building directly $\mathscr{L}_\alpha$ from the excursion $X^{\mathsf{exc},(\alpha)}$ of an $\alpha$-stable spectrally positive Lévy process. This process also codes the $\alpha$-stable tree, and Curien \& Kortchemski justified that $\lop[\alpha]$ can be interpreted as its associated looptree.

The idea to encode metric spaces by real-valued functions is in line with the founding work of Le Gall~\cite{legall_trees} which studies and constructs compact real trees from continuous excursions. Namely, if $f:[0,1]\longrightarrow [0,\infty)$ is continuous with $f(0)=f(1)=0$ then setting
\begin{equation}
\label{classical_tree_distance}
\forall 0\leq s\leq t\leq 1,\quad d_f(t,s)=d_f(s,t)=f(s)+f(t)-2\inf_{[s,t]}f
\end{equation}
defines a pseudo-distance on $[0,1]$, which in turn induces a compact real tree by quotient. See \cite{legall_trees,duquesne_coding} for an extensive study of this coding. The definition of real trees will be reminded in Section~\ref{topo} and the quotient metric space induced by a pseudo-distance is defined in Section~\ref{space}. Similarly, the $\alpha$-stable looptree is induced by another pseudo-distance inherited from $X^{\mathsf{exc},(\alpha)}$. This construction of Curien \& Kortchemski~\cite{curien2014} can be straightforwardly adapted to build a pseudo-distance $\dl[f]$ from a càdlàg function $f$ without negative jumps, called an excursion. We denote by $\lop[f]$ and we call the looptree coded by $f$ the quotient metric space induced by $\dl[f]$. Without immediately giving the definition of $\dl[f]$, let us already state that each jump of $f$ of height $\Delta$ corresponds to a circle of length $\Delta$ in $\lop[f]$, that those circles are dense in $\lop[f]$, and that if $g$ is a 'sub-excursion' of $f$ then $\lop[g]$ is a subset of $\lop[f]$. This method has many benefits. While it allows an automatic construction of a large diversity of complex looptrees, it is also fairly simple to find a function coding a given discrete looptree. This kind of behavior is especially useful for tackling limit problems.

However, $\dl$ alone is not enough to understand all convergences of looptrees. Indeed, a looptree consisting of a chain of $n$ loops of lengths $1/n$ put back to back is asymptotically close to a segment that has no loops. It is not an unusual situation, as \cite{haas}, \cite{curien2014}, \cite{perco_curien}, and \cite{KR2020} provide as many examples of looptrees that converge towards loopless compact real trees. In particular, it ensures that $\dl$ lacks the key property of being continuous with respect to the coding function. Furthermore, the coding brought by $\dl$ does not fully exploit the diversity of excursions because it only cares about the jumps: if $f$ is continuous, then $\dl[f]=0$ and the looptree coded by $f$ is reduced to a single point. In fact, $\dl$ disregards continuous growth that would classically code a real tree via (\ref{classical_tree_distance}). In contrast, applying (\ref{classical_tree_distance}) to a discontinuous excursion would still give a tree but without distinguishing between jumps and continuous growth. These observations motivate us to define a new pseudo-distance $\dtree[f]$ that induces a real tree but that only harnesses the continuous growth of $f$. Informally, $\dtree[f]$ is obtained by stripping $f$ of its jumps and then by using (\ref{classical_tree_distance}).

A linear combination of $\dtree[f]$ and $\dl[f]$ gives a pseudo-distance that induces a new hybrid metric space, denoted by $\vern[f]$, composed of loops and trees tangently glued together along some genealogical structure. More precisely, we choose $\dv[f]=\dtree[f]+2\dl[f]$ because the diameter of a metric circle is half of its length. We call $\vern[f]$ the \emph{vernation tree} coded by $f$. The name vernation tree was chosen to distinguish this new object from the looptree $\lop[f]$. It refers to the arrangement of bud scales or young leaves in a leaf bud before it opens. Figure~\ref{collection} presents some examples and non-examples of vernation trees. Let us already mention that independently, Blanc-Renaudie~\cite{blanc-renaudie} and Marzouk~\cite{marzouk_similaire} construct the same kind of spaces as part of their studies of scaling limits of models with prescribed degrees. While Blanc-Renaudie took the point of view of stick-breaking constructions and managed to mimic the association of a looptree with a discrete tree, Marzouk found the same coding by excursions as us and also identified the decomposition into a pure jump part and a continuous part. However, our methods to study convergences are completely different and we truly believe both should be useful according to context. Let us now describe our codings more precisely.

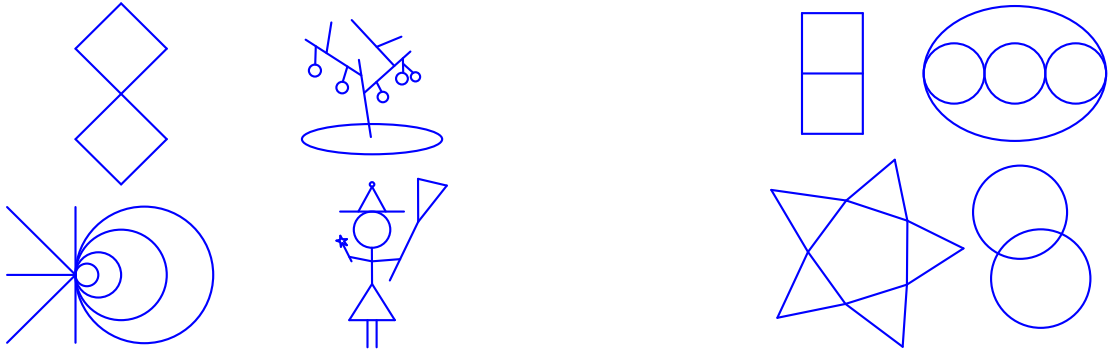
\begin{figure}
\begin{center}
\begin{subfigure}{0.4\textwidth}
\begin{tikzpicture}[line cap=round,line join=round,>=triangle 45,x=1cm,y=1cm,scale=0.6]
\clip(-2.089510074745934,-2.002876489918182) rectangle (8.15112747699881,6.292448935223836);
\draw [line width=0.8pt,color=qqqqff] (0.5,0) circle (0.5cm);
\draw [line width=0.8pt,color=qqqqff] (1,0) circle (1cm);
\draw [line width=0.8pt,color=qqqqff] (1.5091945987813442,0) circle (1.5091945987813442cm);
\draw [line width=0.8pt,color=qqqqff] (0.25,0) circle (0.25cm);
\draw [line width=0.8pt,color=qqqqff] (0,1.5)-- (0,-1.5);
\draw [line width=0.8pt,color=qqqqff] (0,0)-- (-1.5,0);
\draw [line width=0.8pt,color=qqqqff] (-1.5,1.5)-- (0,0);
\draw [line width=0.8pt,color=qqqqff] (0,0)-- (-1.5,-1.5);
\draw [rotate around={0:(6.5,3)},line width=0.8pt,color=qqqqff] (6.5,3) ellipse (1.536800161822768cm and 0.3342973786592923cm);
\draw [line width=0.8pt,color=qqqqff] (6.474546526743044,3.0468898303959677)-- (6.212063998583023,4.753026263436108);
\draw [line width=0.8pt,color=qqqqff] (6.324946658092304,4.0192889766257816)-- (7.340738869671115,4.9367640331481235);
\draw [line width=0.8pt,color=qqqqff] (6.98891195716452,4.618989971919239)-- (6.05457448168701,5.63234273277218);
\draw [line width=0.8pt,color=qqqqff] (6.265167400233894,4.407854152705443)-- (5.044016748270928,5.199246561308144);
\draw [line width=0.8pt,color=qqqqff] (5.502010324575042,4.902434174651965)-- (5.608354183814974,5.579846227140176);
\draw [line width=0.8pt,color=qqqqff] (6.601158329416467,5.039535245465522)-- (7.143876973551099,5.264867193348151);
\draw [line width=0.8pt,color=qqqqff] (5.267023825205505,5.054722126532627)-- (5.254002770798945,4.648033252172099);
\draw [line width=0.8pt,color=qqqqff] (7.175667075462864,4.78766931476466)-- (7.183249352775102,4.464295482460084);
\draw [line width=0.8pt,color=qqqqff] (5.96000390087916,4.605621787995783)-- (5.857712585566994,4.267433586340068);
\draw [line width=0.8pt,color=qqqqff] (6.597216229191291,4.265205943979864)-- (6.710780802087063,4.044323437404049);
\draw [line width=0.8pt,color=qqqqff] (5.248009637127754,4.5159147275574245) circle (0.1322543844171494cm);
\draw [line width=0.8pt,color=qqqqff] (5.8460129238969305,4.140309357355298) circle (0.12766147295941013cm);
\draw [line width=0.8pt,color=qqqqff] (7.178450487302816,4.668960579050869)-- (7.4027983398662744,4.466492968320303);
\draw [line width=0.8pt,color=qqqqff] (7.155689543680665,4.337996394303786) circle (0.12927104372764345cm);
\draw [line width=0.8pt,color=qqqqff] (7.452220099103397,4.377533801693483) circle (0.10176563080446378cm);
\draw [line width=0.8pt,color=qqqqff] (6.737590831207744,3.9317268275128194) circle (0.11574443494375693cm);
\draw [line width=0.8pt,color=qqqqff] (6.5,2) circle (0.05cm);
\draw [line width=0.8pt,color=qqqqff] (6.5,1) circle (0.4cm);
\draw [line width=0.8pt,color=qqqqff] (6.2,1.4)-- (6.5,1.95);
\draw [line width=0.8pt,color=qqqqff] (6.8,1.4)-- (6.5,1.95);
\draw [line width=0.8pt,color=qqqqff] (5.8,1.4)-- (7.2,1.4);
\draw [line width=0.8pt,color=qqqqff] (6.496197401372517,0.600018075103789)-- (6.5,-0.2);
\draw [line width=0.8pt,color=qqqqff] (6,-1)-- (6.5,-0.2);
\draw [line width=0.8pt,color=qqqqff] (6.5,-0.2)-- (7,-1);
\draw [line width=0.8pt,color=qqqqff] (6,-1)-- (7,-1);
\draw [line width=0.8pt,color=qqqqff] (6.4,-1)-- (6.4,-1.6);
\draw [line width=0.8pt,color=qqqqff] (6.6,-1)-- (6.6,-1.6);
\draw [line width=0.8pt,color=qqqqff] (6.4976234832448645,0.29998870408100486)-- (6,0.4);
\draw [line width=0.8pt,color=qqqqff] (6.05,0.3)-- (5.85,0.7);
\draw [line width=0.8pt,color=qqqqff] (6.4976234832448645,0.29998870408100486)-- (7.115614979076993,0.3485656359143744);
\draw [line width=0.8pt,color=qqqqff] (6.88931860611223,-0.1233811005228917)-- (7.115614979076993,0.3485656359143744);
\draw [line width=0.8pt,color=qqqqff] (7.115614979076993,0.3485656359143744)-- (7.51116499477031,1.1734950705246188);
\draw [line width=0.8pt,color=qqqqff] (5.85,0.7)-- (5.817221688752327,0.6219348017540037);
\draw [line width=0.8pt,color=qqqqff] (5.85,0.7)-- (5.948235085801955,0.6562478343146206);
\draw [line width=0.8pt,color=qqqqff] (5.948235085801955,0.6562478343146206)-- (5.88740743717177,0.7435900990143727);
\draw [line width=0.8pt,color=qqqqff] (5.817221688752327,0.6219348017540037)-- (5.801624855770229,0.7248738994358545);
\draw [line width=0.8pt,color=qqqqff] (5.801624855770229,0.7248738994358545)-- (5.718961640965107,0.7576272486982615);
\draw [line width=0.8pt,color=qqqqff] (5.718961640965107,0.7576272486982615)-- (5.799309147272558,0.7766976444268715);
\draw [line width=0.8pt,color=qqqqff] (5.799309147272558,0.7766976444268715)-- (5.804212928141515,0.8659464562418926);
\draw [line width=0.8pt,color=qqqqff] (5.804212928141515,0.8659464562418926)-- (5.8493277121359215,0.7904282308599516);
\draw [line width=0.8pt,color=qqqqff] (5.8493277121359215,0.7904282308599516)-- (5.949364841862647,0.7904282308599516);
\draw [line width=0.8pt,color=qqqqff] (5.949364841862647,0.7904282308599516)-- (5.88740743717177,0.7435900990143727);
\draw [line width=0.8pt,color=qqqqff] (1,4)-- (0,3);
\draw [line width=0.8pt,color=qqqqff] (0,3)-- (1,2);
\draw [line width=0.8pt,color=qqqqff] (1,4)-- (2,3);
\draw [line width=0.8pt,color=qqqqff] (1,4)-- (0,5);
\draw [line width=0.8pt,color=qqqqff] (0,5)-- (1,6);
\draw [line width=0.8pt,color=qqqqff] (1,6)-- (2,5);
\draw [line width=0.8pt,color=qqqqff] (2,5)-- (1,4);
\draw [line width=0.8pt,color=qqqqff] (2,3)-- (1,2);
\draw [line width=0.8pt,color=qqqqff] (7.51116499477031,1.1734950705246188)-- (7.508161941815038,2.1233465036423635);
\draw [line width=0.8pt,color=qqqqff] (7.51116499477031,1.1734950705246188)-- (8.138707295804616,1.9741049997395035);
\draw [line width=0.8pt,color=qqqqff] (7.508161941815038,2.1233465036423635)-- (8.138707295804616,1.9741049997395035);
\end{tikzpicture}
\end{subfigure}
\hfill
\begin{subfigure}{0.4\textwidth}
\begin{tikzpicture}[line cap=round,line join=round,>=triangle 45,x=1cm,y=1cm,scale=0.4]
\clip(-2.990171685705195,-7.443217522374598) rectangle (10.532308911918086,4.5165045313244345);
\draw [line width=0.8pt,color=qqqqff] (0,0)-- (2,0);
\draw [line width=0.8pt,color=qqqqff] (2,0)-- (2,4);
\draw [line width=0.8pt,color=qqqqff] (2,4)-- (0,4);
\draw [line width=0.8pt,color=qqqqff] (0,4)-- (0,0);
\draw [line width=0.8pt,color=qqqqff] (0,2)-- (2,2);
\draw [line width=0.8pt,color=qqqqff] (7.17,-2.6) circle (1.5455743269089322cm);
\draw [line width=0.8pt,color=qqqqff] (7.85,-4.8) circle (1.6330646037435268cm);
\draw [line width=0.8pt,color=qqqqff] (0.1971095009337085,-3.9166367134837556)-- (1.43,-5.64);
\draw [line width=0.8pt,color=qqqqff] (1.43,-5.64)-- (3.45,-5);
\draw [line width=0.8pt,color=qqqqff] (3.45,-5)-- (3.4655381582084956,-2.881094960683824);
\draw [line width=0.8pt,color=qqqqff] (3.4655381582084956,-2.881094960683824)-- (1.4551412681039195,-2.2115396274529946);
\draw [line width=0.8pt,color=qqqqff] (1.4551412681039195,-2.2115396274529946)-- (0.1971095009337085,-3.9166367134837556);
\draw [line width=0.8pt,color=qqqqff] (0.1971095009337085,-3.9166367134837556)-- (-1.01,-1.86);
\draw [line width=0.8pt,color=qqqqff] (-1.01,-1.86)-- (1.4551412681039195,-2.2115396274529946);
\draw [line width=0.8pt,color=qqqqff] (1.4551412681039195,-2.2115396274529946)-- (3.05,-0.86);
\draw [line width=0.8pt,color=qqqqff] (3.05,-0.86)-- (3.4655381582084956,-2.881094960683824);
\draw [line width=0.8pt,color=qqqqff] (3.4655381582084956,-2.881094960683824)-- (5.31,-3.8);
\draw [line width=0.8pt,color=qqqqff] (5.31,-3.8)-- (3.45,-5);
\draw [line width=0.8pt,color=qqqqff] (3.45,-5)-- (3.31,-7.06);
\draw [line width=0.8pt,color=qqqqff] (3.31,-7.06)-- (1.43,-5.64);
\draw [line width=0.8pt,color=qqqqff] (1.43,-5.64)-- (-0.81,-6.1);
\draw [line width=0.8pt,color=qqqqff] (-0.81,-6.1)-- (0.1971095009337085,-3.9166367134837556);
\draw [line width=0.8pt,color=qqqqff] (5,2) circle (1cm);
\draw [line width=0.8pt,color=qqqqff] (7,2) circle (1cm);
\draw [line width=0.8pt,color=qqqqff] (9,2) circle (1cm);
\draw [rotate around={0:(7,2)},line width=0.8,color=qqqqff] (7,2) ellipse (3cm and 2.23606797749979cm);
\end{tikzpicture}
\end{subfigure}
\end{center}
\caption{All eight figures have to be understood as intrinsic one-dimensional spaces. \textit{Left:} Four examples of vernation trees. \textit{Right:} Four non-examples of vernation trees.}
\label{collection}
\end{figure}

\subsection{Description of vernation trees coded by excursions}
\label{intro_desc}

A \emph{càdlàg} function is a right-continuous function that has left limits everywhere. We denote by $\mathbb{D}([-1,1])$ the space of all càdlàg function from $[-1,1]$ to $\mathbb{R}$. The space $\mathbb{D}([-1,1])$ is equipped with the Skorokhod distance $\rho$ defined by
\begin{equation}
\label{sko_dist}
\forall f_1,f_2\in \mathbb{D}([-1,1]),\quad \rho(f_1,f_2)=\inf_{\lambda\in\Lambda}\max\left(||\lambda-\mathsf{id}||_\infty,||f_1-f_2\circ\lambda||_\infty\right),
\end{equation}
where $\Lambda$ is the set of increasing bijections from $[-1,1]$ into itself. The distance $\rho$ is not complete but induces a \emph{Polish} (namely separable and completely metrizable) topology on $\mathbb{D}([-1,1])$ called the Skorokhod (J1) topology. We refer to Billingsley~\cite[Chapter 3]{billingsley2013convergence} and Jacod \& Shiryanev~\cite[Chapter VI]{jacod} for background. Let $f\in\mathbb{D}([-1,1])$. Throughout this work, we adopt the following notations for all $s,t\in[-1,1]$:
\begin{equation}
\label{delta_x}
f(t-)=\lim_{s\to t-}f(s),\quad \Delta_t(f)=f(t)-f(t-),\quad\text{ and }\quad x_s^t(f)=\I{s\leq t}\max\Big(\inf_{[s,t]}f-f(s-)\, ,\, 0\Big).
\end{equation}
We may simply write $\Delta_t$ or $x_s^t$ when the function is obvious according to context. We now define our coding functions.

\begin{definition}
\label{excursion}
An \emph{excursion} is a nonnegative càdlàg function $f:[-1,1]\longrightarrow[0,\infty)$ that satisfies the following:
\begin{enumerate}
\item[$(a)$] $f(s)=f(1)=0$ for all $s\in[-1,0)$;
\item[$(b)$] $\Delta_t(f)\geq 0$ for all $t\in[0,1]$.
\end{enumerate}
We denote by $\mathbb{H}$ the space of all excursions. This is a closed subset of $\mathbb{D}([-1,1])$, so it is Polish.\cq
\end{definition}

\noindent
The reason why we work with the Skorokhod topology on $[-1,1]$ rather than $[0,1]$ is to allow convergences $f_n\to f$ where $f(0)\neq 0$ but $f_n(0)=0$ for all $n\geq 1$: for example, when $f(t)=\un_{[0,1]}(t)(1-t)$ and $f_n(t)=\un_{[1/n,1]}(t)(1-t)$. We fix the value $f(1)=0$ to make the time $1$ correspond to a distinguished point of our coded metric spaces. The absence of negative jumps allows us to interpret them as lengths of some loops. We denote by $(\mathcal{C},\delta)$ the metric circle of perimeter $1$, seen as $[0,1]$ endowed with the pseudo-distance $\delta$ given by 
\begin{equation}
\label{circle_dist_cano}
\forall a,b\in[0,1],\quad \delta(a,b)=\min\big(|a-b|,1-|a-b|\big),
\end{equation}
so that $0$ and $1$ are identified. For all $s\in[0,1]$, we rescale $\delta$ with $\Delta_s$ by setting $\delta_s=0$ when $\Delta_s=0$, and
\begin{equation}
\label{circle_dist_s}
\forall x,y\in[0,\Delta_s],\quad \delta_s(x,y)=\Delta_s\delta\left(\frac{x}{\Delta_s},\frac{y}{\Delta_s}\right)=\min\left(|x-y|,\Delta_s-|x-y|\right)
\end{equation}
when $\Delta_s>0$. If needed, we would precise the dependence on $f$ by writing $\delta_s^f$. Equipped with the pseudo-distance $\delta_s$, the segment $[0,\Delta_s]$ is isometric to a circle of length $\Delta_s$, viewed as the loop associated with the time $s$. Then for any $t\in[0,1]$, the quantity $x_s^t\in[0,\Delta_s]$ will represent the position on this loop where the sub-looptree containing the point corresponding to $t$ is attached. In other words, the shortest path on the looptree from $t$ to the loop arrives at the position $x_s^t$. In particular, when $t=1$ corresponds to the root of the looptree, the position $x_s^1=0$ is the origin of the loop. Therefore, any path from the root to some $t$ with $x_s^t\neq 0$ must visit the loop associated with $s$. This motivates the following expression of Curien \& Kortchemski~\cite{curien2014} for the underlying genealogy of the looptree coded by $f$: for $s,t\in[0,1]$, we write
\begin{equation}
\label{genealogy_def}
s\preceq_f t\quad \text{ when both }\quad s\leq t\ \text{ and }\ f(s-)\leq \inf_{[s,t]}f.
\end{equation}
We justify that this indeed defines a genealogy in Section~\ref{genealogy_sec}. We also provide some tools to study it.
\smallskip

We are now ready to define our codings. First, the \emph{looptree pseudo-distance} coded by an excursion $f$ is given by
\begin{equation}
\label{dl_easy}
\forall s,t\in[-1,1],\quad \dl[f](s,t)=\sum_{r\in[0,1]}\delta_r(x_r^s,x_r^t).
\end{equation}
This expression gives the same quantity constructed by Curien \& Kortchemski~\cite{curien2014} for the stable looptrees in terms of the genealogy (\ref{genealogy_def}). See Section~\ref{pseudo_dist_coded_sec} for that other expression. The term $\delta_r(x_r^s,x_r^t)$ corresponds to the minimal distance to run on the loop associated with $r$ by a path linking the points corresponding to $s$ and $t$. Next, we modify the classic expression (\ref{classical_tree_distance}) to define the \emph{tree pseudo-distance} coded by $f$ as
\begin{equation}
\label{dtree_easy}
\forall -1\leq s\leq t\leq 1,\quad \dtree[f](s,t)=\dtree[f](t,s)=f(s)+f(t)-2\inf_{[s,t]}f-\sum_{r\in[0,1]}|x_r^s-x_r^t|.
\end{equation}
Here, we have subtracted the terms $|x_r^s-x_r^t|$ to compensate for the contribution of the jumps of the excursion, since this information was already used to construct the looptree pseudo-distance. Nevertheless, if the excursion is continuous on $[-1,1]$ then we recover the formula (\ref{classical_tree_distance}). To combine the distances covered on loops and on branches, we set
\begin{equation}
\label{dv_easy}
\dv[f]=\dtree[f]+2\dl[f],
\end{equation}
that we call the \emph{vernation-tree pseudo-distance} coded by $f$. The constant $2$ before $\dl$ could be replaced with another one, but this choice will be justified by the fact that it ensures the greatest possible continuity to the coding $f\longmapsto\dv[f]$. However, Marzouk~\cite{marzouk_similaire} showed that metrics $\dtree+a\dl$ with $a\neq 2$ can naturally appear as scaling limits of uniformly random discrete looptrees with prescribed degrees.
\smallskip

We show that $\dl[f],\dtree[f],\dv[f]$ are continuous pseudo-distances on $[0,1]$ in Section~\ref{section_coding}. Hence, they each induce a compact metric space after the quotient by the equivalent relation generated by their vanishing set. We refer to Section~\ref{space} for details. We respectively call and denote these metric spaces as follows:
\begin{enumerate}
\item[$\bullet$] the \emph{looptree} $(\lop[f],\dl[f])$ coded by $f$,
\item[$\bullet$] the \emph{tree} $(\tree[f],\dtree[f])$ coded by $f$,
\item[$\bullet$] the \emph{vernation tree} $(\vern[f],\dv[f])$ coded by $f$.
\end{enumerate}
See Figure~\ref{example_codings} for an example. Moreover, we equip all these spaces with the canonical projection of $1$, and with the pushforward measure of the Lebesgue measure on $[0,1]$. Hence, they are pointed weighted metric spaces, i.e.~metric spaces endowed with a distinguished point and a Borel probability measure. If the excursion $f$ is obvious from the context, we may drop the subscript $_f$ from the notation.
\smallskip

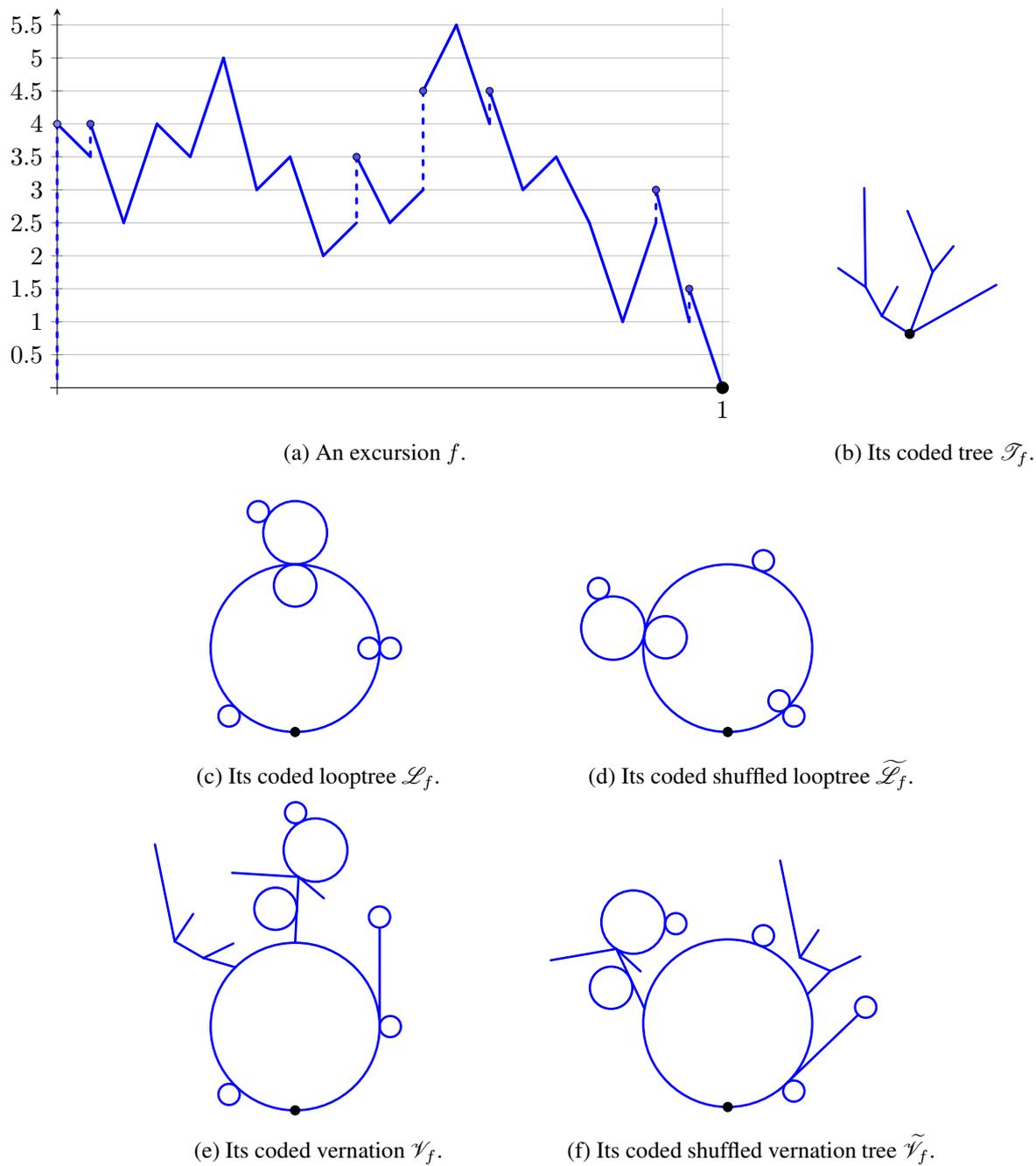
\begin{figure}
\begin{center}
\begin{subfigure}[b]{0.7\textwidth}
\begin{tikzpicture}[line cap=round,line join=round,>=triangle 45,x=1cm,y=1cm,scale=1]
\begin{axis}[
x=10cm,y=1cm,
axis lines=middle,
ymajorgrids=true,
xmajorgrids=true,
xmin=-0.01,
xmax=1.01,
ymin=-0.1,
ymax=5.75,
xtick={-1,0,1},
ytick={0,0.5,...,5.5},]
\clip(-0.2,-0.2) rectangle (1.1,6);
\draw [line width=1.2pt,color=qqqqff] (0,4)-- (0.05,3.5);
\draw [line width=1.2pt,color=qqqqff] (0.05,4)-- (0.1,2.5);
\draw [line width=1.2pt,color=qqqqff] (0.15,4)-- (0.1,2.5);
\draw [line width=1.2pt,color=qqqqff] (0.15,4)-- (0.2,3.5);
\draw [line width=1.2pt,color=qqqqff] (0.25,5)-- (0.2,3.5);
\draw [line width=1.2pt,color=qqqqff] (0.25,5)-- (0.3,3);
\draw [line width=1.2pt,color=qqqqff] (0.35,3.5)-- (0.3,3);
\draw [line width=1.2pt,color=qqqqff] (0.35,3.5)-- (0.4,2);
\draw [line width=1.2pt,color=qqqqff] (0.4,2)-- (0.45,2.5);
\draw [line width=1.2pt,color=qqqqff] (0.45,3.5)-- (0.5,2.5);
\draw [line width=1.2pt,color=qqqqff] (0.5,2.5)-- (0.55,3);
\draw [line width=1.2pt,color=qqqqff] (0.55,4.5)-- (0.6,5.5);
\draw [line width=1.2pt,color=qqqqff] (0.6,5.5)-- (0.65,4);
\draw [line width=1.2pt,color=qqqqff] (0.65,4.5)-- (0.7,3);
\draw [line width=1.2pt,color=qqqqff] (0.7,3)-- (0.75,3.5);
\draw [line width=1.2pt,color=qqqqff] (0.75,3.5)-- (0.8,2.5);
\draw [line width=1.2pt,color=qqqqff] (0.95,1.5)-- (1,0);
\draw [line width=1.2pt,color=qqqqff] (0.8,2.5)-- (0.85,1);
\draw [line width=1.2pt,color=qqqqff] (0.9,2.5)-- (0.85,1);
\draw [line width=1.2pt,color=qqqqff] (0.9,3)-- (0.95,1);
\draw [line width=1.2pt,dash pattern=on 2pt off 4pt,color=qqqqff] (0.05,4)-- (0.05,3.5);
\draw [line width=1.2pt,dash pattern=on 2pt off 4pt,color=qqqqff] (0.45,3.5)-- (0.45,2.5);
\draw [line width=1.2pt,dash pattern=on 2pt off 4pt,color=qqqqff] (0.55,4.5)-- (0.55,3);
\draw [line width=1.2pt,dash pattern=on 2pt off 4pt,color=qqqqff] (0.65,4.5)-- (0.65,4);
\draw [line width=1.2pt,dash pattern=on 2pt off 4pt,color=qqqqff] (0.9,2.5)-- (0.9,3);
\draw [line width=1.2pt,dash pattern=on 2pt off 4pt,color=qqqqff] (0.95,1)-- (0.95,1.5);
\draw [line width=1.2pt,dash pattern=on 2pt off 4pt,color=qqqqff] (0,4)-- (0,0);
\begin{scriptsize}
\draw [fill=xdxdff] (0,4) circle [radius=1.5pt];
\draw [fill=ududff] (0.05,4) circle [radius=1.5pt];
\draw [fill=ududff] (0.45,3.5) circle [radius=1.5pt];
\draw [fill=ududff] (0.55,4.5) circle [radius=1.5pt];
\draw [fill=ududff] (0.65,4.5) circle [radius=1.5pt];
\draw [fill=ududff] (0.9,3) circle [radius=1.5pt];
\draw [fill=ududff] (0.95,1.5) circle [radius=1.5pt];
\draw [fill=black] (1,0) circle [radius=2.5pt];
\end{scriptsize}
\end{axis}
\end{tikzpicture}
\caption{An excursion $f$.}
\end{subfigure}
\hfill
\begin{subfigure}[b]{0.25\textwidth}
\begin{tikzpicture}[line cap=round,line join=round,>=triangle 45,x=1cm,y=1cm,scale=1]
\clip(-1.9294436651515683,-1.1254056570576816) rectangle (1.2,2.5450839334064272);
\draw [line width=1pt,color=qqqqff] (-0.3009698028201784,0.27736368789891436)-- (-0.7201916109704967,0.5498578631966218);
\draw [line width=1pt,color=qqqqff] (-0.7201916109704967,0.5498578631966218)-- (-0.48501028864838136,0.9910944727734126);
\draw [line width=1pt,color=qqqqff] (-0.7201916109704967,0.5498578631966218)-- (-0.9628298500909859,0.987038243493483);
\draw [line width=1pt,color=qqqqff] (-0.9628298500909859,0.987038243493483)-- (-0.9839808473738328,2.4868891145183127);
\draw [line width=1pt,color=qqqqff] (-0.9628298500909859,0.987038243493483)-- (-1.3730667513695263,1.2728803062298497);
\draw [line width=1pt,color=qqqqff] (-0.3009698028201784,0.27736368789891436)-- (1.0023659843192807,1.01987013632988);
\draw [line width=1pt,color=qqqqff] (-0.3009698028201784,0.27736368789891436)-- (0.04509190898878812,1.2155754399143364);
\draw [line width=1pt,color=qqqqff] (0.04509190898878812,1.2155754399143364)-- (-0.3333576725163929,2.1411973466659826);
\draw [line width=1pt,color=qqqqff] (0.04509190898878812,1.2155754399143364)-- (0.35768641119605926,1.6058121337166855);
\begin{scriptsize}
\draw [fill=black] (-0.3009698028201784,0.27736368789891436) circle (2pt);
\end{scriptsize}
\end{tikzpicture}
\caption{Its coded tree $\tree[f]$.}
\end{subfigure}
\begin{subfigure}[b]{0.4\textwidth}
\begin{tikzpicture}[line cap=round,line join=round,>=triangle 45,x=1cm,y=1cm,scale=0.5]
\clip(-5.853370904701763,-1.5504041341823112) rectangle (3,2.694532449652895);
\draw [line width=1pt,color=qqqqff] (0,0) circle (1.27cm);
\draw [line width=1pt,color=qqqqff] (-0.9933614697563665,-1.0286559144842706) circle (0.16cm);
\draw [line width=1pt,color=qqqqff] (1.43,0) circle (0.16cm);
\draw [line width=1pt,color=qqqqff] (0,0.95) circle (0.32cm);
\draw [line width=1pt,color=qqqqff] (0,1.75) circle (0.48cm);
\draw [line width=1pt,color=qqqqff] (-0.5542562584220398,2.07) circle (0.16cm);
\draw [line width=1pt,color=qqqqff] (1.11,0) circle (0.16cm);
\begin{scriptsize}
\draw [fill=black] (0,-1.27) circle (2pt);
\end{scriptsize}
\end{tikzpicture}
\caption{Its coded looptree $\lop[f]$.}
\end{subfigure}
\begin{subfigure}[b]{0.4\textwidth}
\begin{tikzpicture}[line cap=round,line join=round,>=triangle 45,x=1cm,y=1cm,scale=0.5]
\clip(-5.853370904701763,-1.4864723210673187) rectangle (3,1.7501599002246733);
\draw [line width=1pt,color=qqqqff] (0,0) circle (1.27cm);
\draw [line width=1pt,color=qqqqff] (0.5356874285847545,1.325872912030506) circle (0.16cm);
\draw [line width=1pt,color=qqqqff] (0.9933614697563667,-1.0286559144842717) circle (0.16cm);
\draw [line width=1pt,color=qqqqff] (-0.9355673653615976,0.1649657687835838) circle (0.32cm);
\draw [line width=1pt,color=qqqqff] (-1.723413567771364,0.30388431091712803) circle (0.48cm);
\draw [line width=1pt,color=qqqqff] (-1.9423064594997914,0.9052875882201087) circle (0.16cm);
\draw [line width=1pt,color=qqqqff] (0.7710707912094866,-0.7984671783759024) circle (0.16cm);
\begin{scriptsize}
\draw [fill=black] (0,-1.27) circle (2pt);
\end{scriptsize}
\end{tikzpicture}
\caption{Its coded shuffled looptree $\lopt[f]$.}
\end{subfigure}
\begin{subfigure}[b]{0.4\textwidth}
\begin{tikzpicture}[line cap=round,line join=round,>=triangle 45,x=1cm,y=1cm]
\clip(-2.853370904701763,-1.4798132100189718) rectangle (3,3.5176475422690736);
\draw [line width=1pt,color=qqqqff] (0,0) circle (1.27cm);
\draw [line width=1pt,color=qqqqff] (-0.9933614697563665,-1.0286559144842706) circle (0.16cm);
\draw [line width=1pt,color=qqqqff] (-0.8980256121069153,0.8980256121069156)-- (-1.3778062697376139,1.0387754631315416);
\draw [line width=1pt,color=qqqqff] (-1.3778062697376139,1.0387754631315416)-- (-0.9289692571342836,1.2591055248964262);
\draw [line width=1pt,color=qqqqff] (-1.3778062697376139,1.0387754631315416)-- (-1.8099324381923567,1.2903047348607446);
\draw [line width=1pt,color=qqqqff] (-1.8099324381923567,1.2903047348607446)-- (-2.107481896102372,2.7604966679420877);
\draw [line width=1pt,color=qqqqff] (-1.8099324381923567,1.2903047348607446)-- (-1.5326749492089442,1.7063916065625717);
\draw [line width=1pt,color=qqqqff] (0,1.27)-- (0.050888389428292505,2.2687043465517682);
\draw [line width=1pt,color=qqqqff] (-0.2941411961824201,1.7856364578929376) circle (0.32cm);
\draw [line width=1pt,color=qqqqff] (0.30620323747696576,2.6751700011474338) circle (0.48cm);
\draw [line width=1pt,color=qqqqff] (1.43,0) circle (0.16cm);
\draw [line width=1pt,color=qqqqff] (1.27,0)-- (1.27,1.5);
\draw [line width=1pt,color=qqqqff] (1.2682824259940673,1.6599907807954388) circle (0.16cm);
\draw [line width=1pt,color=qqqqff] (0.007067025981794206,3.2409592967092262) circle (0.16cm);
\draw [line width=1pt,color=qqqqff] (0.050888389428292505,2.2687043465517682)-- (-0.9472039569811677,2.3304430543359294);
\draw [line width=1pt,color=qqqqff] (0.050888389428292505,2.2687043465517682)-- (0.4318648632963328,1.9448881543082759);
\begin{scriptsize}
\draw [fill=black] (0,-1.27) circle (2pt);
\end{scriptsize}
\end{tikzpicture}
\caption{Its coded vernation $\vern[f]$.}
\end{subfigure}
\begin{subfigure}[b]{0.4\textwidth}
\begin{tikzpicture}[line cap=round,line join=round,>=triangle 45,x=1cm,y=1cm]
\clip(-2.853370904701763,-1.47735716833149693) rectangle (3,2.58995545091422112);
\draw [line width=1pt,color=qqqqff] (0,0) circle (1.27cm);
\draw [line width=1pt,color=qqqqff] (0.5356874285847545,1.325872912030506) circle (0.16cm);
\draw [line width=1pt,color=qqqqff] (1.1934096283981037,0.43436558202359926)-- (1.5431134361292944,0.7917270892656445);
\draw [line width=1pt,color=qqqqff] (1.5431134361292944,0.7917270892656445)-- (1.9919504487326247,1.0120571510305292);
\draw [line width=1pt,color=qqqqff] (1.5431134361292944,0.7917270892656445)-- (1.0861779549856567,0.9947294687394631);
\draw [line width=1pt,color=qqqqff] (1.0861779549856567,0.9947294687394631)-- (0.7886284970756419,2.4649214018208063);
\draw [line width=1pt,color=qqqqff] (1.0861779549856567,0.9947294687394631)-- (1.3634354439690692,1.4108163404412903);
\draw [line width=1pt,color=qqqqff] (-1.2507058463255043,0.22053318563700153)-- (-1.675308821535275,1.1259128371882228);
\draw [line width=1pt,color=qqqqff] (-1.7527288224267799,0.537350059345486) circle (0.32cm);
\draw [line width=1pt,color=qqqqff] (-1.4199939734866016,1.5323784917838883) circle (0.48cm);
\draw [line width=1pt,color=qqqqff] (0.9933614697563667,-1.0286559144842717) circle (0.16cm);
\draw [line width=1pt,color=qqqqff] (0.8822161304829266,-0.913561546430087)-- (1.9612258309909028,0.1284260092584084);
\draw [line width=1pt,color=qqqqff] (2.0734007386014333,0.24251715748020936) circle (0.16cm);
\draw [line width=1pt,color=qqqqff] (-0.7804379645931981,1.5085434023496456) circle (0.16cm);
\draw [line width=1pt,color=qqqqff] (-1.675308821535275,1.1259128371882228)-- (-2.6603561300253475,0.95362876933381);
\draw [line width=1pt,color=qqqqff] (-1.675308821535275,1.1259128371882228)-- (-1.3093114975452367,0.7852576466624788);
\begin{scriptsize}
\draw [fill=black] (0,-1.27) circle (2pt);
\end{scriptsize}
\end{tikzpicture}
\caption{Its coded shuffled vernation tree $\vernt[f]$.}
\end{subfigure}
\end{center}
\caption{An example of an excursion and the quotient metric spaces that it codes. We stress that the representations of the shuffled looptree and the shuffled vernation tree are only likely examples because they depend on the chosen shuffle $\Phi$. The representations given are isometric with respect to the scale of the $y$-axis of the graph of the excursion.}
\label{example_codings}
\end{figure}

Our construction of $\dv$ as the combination of $\dl$ and $\dtree$ follows the intuitive principle that jumps code for loops and continuous growths code for branches in the associated vernation tree. More accurately, it is possible to decompose an excursion $f$ into a pure jump part and a continuous part, that respectively solely induces $\dl[f]$ and $\dtree[f]$, by setting
\begin{equation}
\label{J_def}
Jf:t\in[-1,1]\longmapsto \sum_{r\leq t}x_r^t(f).
\end{equation}
The operator $J:f\mapsto Jf$ should extract the pure jump part of the excursion, which motivates the following definition.
\begin{definition}
\label{type_exc}
Let $f$ be an excursion.
\begin{enumerate}
\item[$(i)$] We say that $f$ is a \emph{continuous excursion} when $Jf=0$, which happens if and only if $f$ is continuous on $[-1,1]$.
\item[$(ii)$] We say that $f$ is a \emph{pure jump growth (PJG) excursion} when $Jf=f$.\cq
\end{enumerate}
\end{definition}

\noindent
The fact that an excursion $f$ is continuous on $[-1,1]$ if and only if $Jf=0$ is immediate via the elementary inequality $0\leq x_r^t\leq\Delta_r\leq Jf(r)$ for all $r,t\in[-1,1]$. Finally, the following result confirms the above decomposition heuristic and entails that the set of vernation trees coded by excursions contains all looptrees and trees coded by excursions.

\begin{theorem}
\label{intro_uni}
For all $f\in\mathbb{H}$, the function $Jf$ is well-defined and càdlàg on $[-1,1]$. Furthermore, the following holds.
\begin{enumerate}
\item[$(i)$] $Jf$ is a PJG excursion such that $\dl[f]=\dl[Jf]$.

\item[$(ii)$] $f-Jf$ is a continuous excursion such that $\dtree[f]=\dtree[f-Jf]$.
\end{enumerate}
Furthermore, it holds that $2\dl[f]=\dv[Jf]$ and $\dtree[f]=\dv[f-Jf]$.
\end{theorem}

\subsection{Shuffling and functional continuity of the coding}
\label{intro_shuffling}

The unification of trees and looptrees via vernation trees reveals an interesting observation. Curien \& Kortchemski~\cite{curien2014} have proved that the $\alpha$-stable looptree converges in distribution to the Brownian CRT (up to a constant multiplicative factor) as $\alpha$ tends to $2$, while the excursion $X^{\mathsf{exc},(\alpha)}$ converges in distribution to a multiple of the Brownian excursion. Plus, the multiplicative constants match. Since we will see in Section~\ref{stable} that the excursions $X^{\mathsf{exc},(\alpha)}$ are PJG and the Brownian excursion is continuous, we deduce that the vernation trees coded by the $X^{\mathsf{exc},(\alpha)}$ converge to the vernation tree coded by the limiting excursion. Furthermore, \cite{curien2014} and \cite{KR2020} exhibit several instances where rescaled looptrees associated with discrete trees converge to the looptree or the tree coded by the scaling limit of their Lukasiewicz walks. Thereby, one can hope that the coding $f\in\mathbb{H}\longmapsto\vern[f]$ enjoys some kind of functional continuity. Such a continuity, already enjoyed by the classic coding (\ref{classical_tree_distance}), would reduce the problem of understanding convergences of looptrees to the study of convergences of excursions, which is a much simpler task thanks to a lot of tools and methods that were developed over the years.

Sadly, the coding of vernation trees by excursions is not continuous. Indeed, \cite[Theorem 13]{haas} and \cite[Theorem 2]{KR2020} provide counterexamples where the limit differs by a strange multiplicative factor from what one could expect. The reason behind this failure is that a succession of small jumps may generate a large variation for the excursion but a small distance with respect to $\dl$ at the same time, since $\delta_t(0,\Delta_t)=0$, whereas a large variation of a continuous excursion is always correlated with a large distance for $\dtree$. For example, the two sequences of PJG excursions defined by
\[\left(f_n(t),g_n(t)\right)=\begin{cases}
					\left(k/2n+1/n-(t-k/2n),(k+1)/2n\right) &\text{ if }\frac{k}{2n}\leq t<\frac{k+1}{2n}\text{ with }0\leq k\leq n-1,\\
					\left(1/2-(t-1/2),1/2-(t-1/2)\right) &\text{ if }1/2\leq t\leq 1,
				\end{cases}\]
uniformly converge to the same continuous excursion $f(t)=1/2-|t-1/2|, t\in[0,1]$ but the vernation trees they code have very different asymptotic behavior. On the one hand, $\vern[f_n]$ consists of a chain of $n$ loops of length $2/n$ put back to back and so converges to a segment of length $1$. On the other hand, $\vern[g_n]$ consists of a bouquet of $n$ loops of length $1/n$ all glued at the same point. This sequence converges to a single point because the sequence of the diameters tends to $0$. Nevertheless, our next result asserts that the former situation, where many small jumps merge into a continuous part at the limit, is the only thing that prevents the codings from being continuous.

\begin{theorem}
\label{gh-continu_particular}
Let $f$ be an excursion and let $(f_n)$ be a sequence of excursions such that $f_n\longrightarrow f$ for the Skorokhod topology on $[-1,1]$. We assume that at least one of the following assumptions holds true.
\begin{enumerate}
\item[$(a)$] There is $N\geq 1$ such that all the $f_n$ have at most $N$ jumps.
\item[$(b)$] The excursion $f$ is PJG in the sense of Definition~\ref{type_exc}.
\end{enumerate}
Then, the three convergences $\lop[f_n]\longrightarrow\lop[f]$, $\tree[f_n]\longrightarrow\tree[f]$, and $\vern[f_n]\longrightarrow\vern[f]$ hold for the pointed GHP topology.
\end{theorem}

\noindent
The \emph{pointed Gromov--Hausdorff--Prokhorov (GHP) topology} formalizes convergences of compact metric spaces endowed with a distinguished point and a Borel probability measure. We will recall its definition in Section~\ref{def_GH}. 
\smallskip

To solve the issue discussed above, we need to match the jump variations of an excursion with those of $\dl$. To this end, we are going to replace every position $x\in[0,\Delta_t]$ on the loop associated with any $t\in[0,1]$ with some $\phi(x)$ so as to have $2\delta_t(0,\phi(x))\simeq x$ when $\Delta_t$ is small. This roughly means that we are going to shuffle the positions on which the components of the vernation tree are glued onto every loop, while keeping the same loops and branches arranged in the same genealogical structure. See Figure~\ref{example_codings} above for an example. The formal construction of this new coding depends on an object $\Phi$ called \emph{shuffle} that contains the shuffling instruction and that must satisfy some assumptions.

Recall from (\ref{circle_dist_cano}) that $(\mathcal{C},\delta)=([0,1],\delta)$ stands for the metric circle of perimeter $1$. We say that a function is c\`agl\`ad when it is left-continuous with right limits (\textit{continue \`a gauche, limite \`a droite} in French). We denote by $\overleftarrow{\mathbb{D}}([0,1],\mathcal{C})$ the set of c\`agl\`ad functions from $[0,1]$ to $\mathcal{C}$ and we endow it with the Skorokhod distance $\overleftarrow{\rho}$ given by 
\begin{equation}
\label{sko_dist_circle}
\overleftarrow{\rho}(\phi_1,\phi_2)=\inf_{\lambda\in\Lambda}\max\Big(||\lambda-\mathsf{id}||_\infty,\sup_{x\in[0,1]}\delta\big(\phi_1(x),\phi_2\circ\lambda(x)\big)\Big)
\end{equation}
for all $\phi_1,\phi_2\in \overleftarrow{\mathbb{D}}([0,1],\mathcal{C})$, where $\Lambda$ is the set of increasing bijections from $[0,1]$ into itself.

\begin{definition}
\label{shuffle}
A \emph{shuffle} is a Borel application $\Phi:\Delta\in (0,\infty)\mapsto\phi_\Delta\in\overleftarrow{\mathbb{D}}([0,1],\mathcal{C})$ such that for all $\Delta>0$, $\phi_\Delta$ is surjective,
\begin{equation}
\label{11}
\sup_{u\in(0,\Delta]}\sup_{x\in(0,1]}\left|\frac{2}{x}\delta\left(\phi_u(0),\phi_u(x)\right)-1\right|<\infty,\quad\text{ and }\quad\sup_{x\in(0,1]}\left|\frac{2}{x}\delta\left(\phi_\Delta(0),\phi_\Delta(x)\right)-1\right|\xrightarrow[\Delta\rightarrow0^+]{}0.
\end{equation}
\cq
\end{definition}
\noindent
Let us discuss this definition. For all $\Delta>0$, $\Delta\phi_\Delta(\cdot/\Delta)$ describes shuffling on any loop of length $\Delta$, so (\ref{11}) formalizes our desired approximation. We ask $\phi_\Delta$ to be \emph{càglàd} for the functions $t\mapsto \phi_{\Delta_s}(x_s^t/\Delta_s)$ to be \emph{càdlàg} on $[0,1]$. Indeed, recall from (\ref{delta_x}) that $x_s^t$, which represents a gluing position on the loop associated with $s$, is non-increasing with $t$. Moreover, we need $\phi_\Delta$ to be surjective to be able to encode a loop of length $\Delta$ without gaps. However, the surjectivity and the convergence in (\ref{11}) lead to the fact that $\phi_\Delta$ has to oscillate more and more quickly near $x=0$ when $\Delta$ is small enough. Moreover, when $\Delta$ gets smaller, $\phi_\Delta$ needs to jump an infinite number of times near $0$. As a result, there is no simple or canonical choice for $\Phi$. Nevertheless, it is not difficult to construct various examples of shuffles. For instance, setting
\begin{equation}
\label{remark_shuffle}
\phi_{1-\Delta}(x)=\begin{cases}
										x-\frac{\Delta^k}{2}&\text{ if }x\in\left(\frac{\Delta^k+\Delta^{k+1}}{2},\Delta^k\right]\text{ with an integer }k,\\
										1-x+\frac{\Delta^{k+1}}{2}&\text{ if }x\in\left(\Delta^{k+1},\frac{\Delta^k+\Delta^{k+1}}{2}\right]\text{ with an integer }k,\\
										0&\text{ if }x=0,
									\end{cases}
\end{equation}
when $\Delta\in(0,1)$ and $\phi_\Delta(x)=x$ when $\Delta\geq 1$, for all $x\in[0,1]$, defines a shuffle.
\smallskip

Now, we fix a shuffle $\Phi$ and let the dependence on it be implicit. Let $f$ be an excursion as in Definition~\ref{excursion}. Recall $\Delta_s$ and $x_s^t$ from (\ref{delta_x}). The shuffled loop associated with $s\in[0,1]$ is the segment $[0,\Delta_s]$ equipped with the pseudo-distance
\begin{equation}
\label{circle_shuffled_dist_s}
\forall x,y\in[0,\Delta_s],\quad\tilde{\delta}_s(x,y)=\Delta_s\delta\left(\phi_{\Delta_s}\left(\frac{x}{\Delta_s}\right),\phi_{\Delta_s}\left(\frac{y}{\Delta_s}\right)\right)
\end{equation}
when $\Delta_s>0$, and with $\tilde{\delta}_s=0$ when $\Delta_s=0$. If needed, we would precise the dependence on $f$ by writing $\tilde{\delta}_s^f$. Recall from (\ref{dtree_easy}) that $\dtree[f]$ stands for the tree pseudo-distance coded by $f$. Then, we define
\begin{equation}
\label{dlt_easy}
\forall s,t\in[-1,1],\quad \dlt[f](s,t)=\sum_{r\in[0,1]}\tilde{\delta}_r(x_r^s,x_r^t)\quad\text{ and }\quad \dvt[f](s,t)=\dtree[f](s,t)+2\dlt[f](s,t),
\end{equation}
and we call $\dlt[f]$ (resp.~$\dvt[f]$) the \emph{shuffled looptree} (resp.~\emph{vernation-tree}) \emph{pseudo-distance} coded by $f$. In Section~\ref{section_coding}, we prove that $\dlt[f]$ and $\dv[f]$ are indeed pseudo-distances on $[0,1]$. We call the quotient metric spaces they induce (see Section~\ref{space})
\begin{enumerate}
\item[$\bullet$] the \emph{shuffled looptree} $(\lopt[f],\dlt[f])$ coded by $f$,
\item[$\bullet$] the \emph{shuffled vernation tree} $(\vernt[f],\dvt[f])$ coded by $f$.
\end{enumerate}
Note that $\dlt[f]$ and $\dvt[f]$ may not be continuous on $[0,1]$, because of the discontinuities of the $\phi_{\Delta_s}$, but we prove in Section~\ref{space} that $\lopt[f]$ and $\vernt[f]$ are compact nonetheless. Moreover, we equip these spaces with the canonical projection of $1$ as their roots, and with the pushforward measure of the Lebesgue measure on $[0,1]$ as their Borel probability measures. If the excursion $f$ is obvious from the context, we may drop the subscript $_f$ from the notation.
\smallskip

Before stating our limit theorem for shuffled vernation trees, we discuss why this coding cannot be continuous at any excursion despite our efforts. Recall that, because of the properties required by Definition~\ref{shuffle}, $\phi_\Delta$ has an infinite number of jumps when $\Delta$ is small. In particular, the function $\hat{\Phi}:(\Delta,x)\in(0,\infty)\times [0,1]\mapsto \phi_\Delta(x)\in\mathcal{C}$ is not continuous everywhere. Hence, the position $\Delta_s\phi_{\Delta_s}(x_s^t/\Delta_s)$ where some component of $\vernt[f]$ is attached to the loop associated with $s$ may greatly vary after even a small perturbation of $f$. To avoid this situation, we need to assume that no gluing position $(\Delta_s,x_s^t/\Delta_s)$ on the unshuffled vernation tree $\vern[f]$ is a discontinuity point of $\hat{\Phi}$. We set $\mathrm{B}(\Phi)=\mathrm{B}^1(\Phi)\cup\mathrm{B}^2(\Phi)$ where
\begin{align*}
\mathrm{B}^1(\Phi)&=\left\{\Delta>0\ :\ \Phi\text{ discontinuous at }x\text{ for the Skorokhod topology}\right\},\\
\mathrm{B}^2(\Phi)&=\left\{(\Delta,x)\in(0,\infty)\times[0,1]\ :\ \phi_\Delta\text{ discontinuous at }x\right\},
\end{align*}
Remark that Definition~\ref{shuffle} implies $\mathrm{B}(\Phi)\cap\left((0,\infty)\!\times\!\{0,1\}\right)=\emptyset$. For any excursion $f$, we define $B_f=B_f^1\cup B_f^2$ where
\begin{align*}
B_f^1&=\left\{\Delta_t\ :\ t\in[0,1]\text{ such that }\Delta_t>0\right\},\\
B_f^2&=\left\{\left(\Delta_r,\frac{x_r^t}{\Delta_r}\right)\ :\ r,s,t\in[0,1]\text{ such that }s\neq t,\Delta_r>0,\text{ and }x_r^t=x_r^s>0\right\}.
\end{align*}

\begin{theorem}
\label{gh-continu}
Let $f$ be an excursion and let $(f_n)$ be a sequence of excursions such that $f_n\longrightarrow f$ for the Skorokhod topology on $[-1,1]$. If $B_f\cap\mathrm{B}(\Phi)=\emptyset$, then the convergence $\vernt[f_n]\longrightarrow\vernt[f]$ holds for the pointed GHP topology.
\end{theorem}

\noindent 
The price for the functional continuity of $f\mapsto\vernt[f]$ is that there is no simple or canonical choice for the shuffle $\Phi$. That makes $\vernt[f]$ inconvenient to use with explicit or discrete examples. Thus, we think it is best to first define a model or to express a given vernation tree by using (unshuffled) $\vern[f]$, and then to study its asymptotic behavior by comparing $\vern[f]$ with $\vernt[f]$. While this technical task will not be automatic, a nice choice of shuffle may ease it. For example, if $X$ is a random excursion, it may be possible to select $\Phi$ such that $\vernt[X]$ and $\vern[X]$ have the same law (see Section~\ref{application_proba}). Another advantage of the flexibility of the choice for the shuffle is that it is easy to ensure that the condition $B_f\cap\mathrm{B}(\Phi)=\emptyset$ holds.

\subsection{Topological aspects of vernation trees}

Theorem~\ref{gh-continu} suggests that vernation trees yield the right framework to study convergences of looptrees. Indeed, Theorem~\ref{intro_uni} justifies that this notion unifies the classic encoding (\ref{classical_tree_distance}) of real trees and the encoding of looptrees introduced by Curien \& Kortchemski~\cite{curien2014}. Moreover, vernation trees can naturally appear as limits of looptrees. Indeed, looptrees may converge to trees or other looptrees, and gluing two looptrees together still gives a looptree. Conversely, the limits of vernation trees are also vernation trees in some sense. To make this point rigorous, we introduce an intrinsic and topological definition for vernation trees that does not involve encoding. 

\begin{definition}
\label{vernation_topo_def}
A metric space $(X,d)$ is a \emph{vernation tree} when it satisfies the following properties:
\begin{enumerate}
\item[$(a)$] for all $x,y\in X$, there is an isometry $g:[0,d(x,y)]\longrightarrow X$ such that $g(0)=x$ and $g(d(x,y))=y$;
\item[$(b)$] for all $L_1,L_2\subset X$ homeomorphic to the topological circle, if $L_1\neq L_2$ then $L_1\cap L_2$ has at most one point.\cq
\end{enumerate}
\end{definition}

\noindent
We let the reader compare this definition with Figure~\ref{collection} above. We then obtain the following characterization.

\begin{theorem}
\label{topo_conclusion}
A compact metric space is a vernation tree (in the sense of Definition~\ref{vernation_topo_def}) if and only if it is the limit of a sequence of vernation trees coded by excursions (as induced by (\ref{dv_easy})) for the GH topology. This holds if and only if it is the limit of a sequence of shuffled vernation trees coded by excursions (as induced by (\ref{dlt_easy})) for the GH topology.
\end{theorem}

\noindent
The definition of the \emph{Gromov--Hausdorff (GH) topology}, which formalizes convergences of (isometry classes of) compact metric spaces, is recalled in Section~\ref{def_GH}.

\subsection{Outline and probabilistic applications}

In Section~\ref{section_coding}, we construct the mentioned pseudo-distances and the metric spaces they induce by quotient. Section~\ref{unification} gives some tools to better understand the relations between trees, looptrees, vernation trees, and excursions. It eventually presents the proof of Theorem~\ref{intro_uni}. We prove Theorems~\ref{gh-continu} and \ref{gh-continu_particular} in Section~\ref{section_cv} by first stating a functional continuity of $f\longmapsto\dvt[f]$. In Section~\ref{topo}, we study the topological notion of vernation trees introduced with Definition~\ref{vernation_topo_def}. We then show Theorem~\ref{topo_conclusion}. Section~\ref{application_proba} is devoted to three probabilistic applications of our work.

The first one, in Section~\ref{random_mapping_sec}, reformulates in metric terms the asymptotics for uniform random mappings found by Aldous, Miermont \& Pitman~\cite{p-mapping} in terms of processes. Indeed, such a mapping $M$ can be represented by a graph on $\{1,2,\ldots,n\}$ with edges $(i,M(i))$, whose each connected component is a vernation tree as a collection of trees rooted on the same cycle. We give scaling limits for these components. The second application, in Section~\ref{stable}, retrieves and completes \cite[Theorem 1.2]{curien2014} of Curien \& Kortchemski about the convergences of $\alpha$-stable looptrees when $\alpha$ tends to $1$ or $2$ by showing that their whole family is continuous in distribution with respect to $\alpha$. The strategy of the proof is to justify that the coding excursions $X^{\mathsf{exc},(\alpha)}$ are PJG and to find a shuffle such that $\vernt[X^{\mathsf{exc},(\alpha)}]$ has the same law as $\vern[X^{\mathsf{exc},(\alpha)}]$. The last application, in Section~\ref{subsection_discrete_loop}, provides two invariance principles for discrete looptrees associated with finite plane trees. We first retrieve \cite[Theorem 4.1]{curien2014} of Curien \& Kortchemski under the PJG case. In the general case, we prove that if a sequence of \emph{exchangeable} random plane trees admits $X$ as the scaling limit of their Lukasiewicz walks and if their heights become negligible against the scaling, then the scaling limit of their associated looptrees is $1/2\cdot\vern[X]$, where $c\cdot E$ stands for the metric space obtained from $E$ by multiplying all distances by $c>0$. Of course, we will clarify what it means for a random plane tree to be exchangeable, but informally, it means its law is invariant by plane reordering.

\section{Tree, looptree, and vernation tree coded by an excursion}
\label{section_coding}

\subsection{Genealogy associated with an excursion}
\label{genealogy_sec}

Let $f$ be an excursion as in Definition~\ref{excursion}. Here, we study the genealogical structure of the metric spaces coded by $f$. Recall from (\ref{genealogy_def}) that for all $s,t\in[0,1]$, we write $s\preceq t$ when $s\leq t$ and $f(s-)\leq\inf_{[s,t]}f$. We also write $s\prec t$ when $s\preceq t$ and $s<t$. This definition from Curien \& Kortchemski~\cite{curien2014} is motivated by the following principle. If $r\in[s,1]$ is the last time such that $f(r)=f(s-)$, then the function $u\in[0,1]\mapsto f(s+u(r-s))-f(s-)$ extends into an excursion that should code the subspace genealogically above the loop associated with $s$. The proposition below justifies that $\preceq$ indeed describes a genealogy. Moreover, the following lemma provides an expression for its most recent common ancestors.

\begin{proposition}
\label{genealogy_prop}
The relation $\preceq$ enjoys the following properties.
\begin{enumerate}
\item[$(i)$] The relation $\preceq$ is a partial order on $[0,1]$.
\item[$(ii)$] The genealogy admits $0$ as its root, namely $0\preceq t$ for all $t\in[0,1]$.
\item[$(iii)$] Two points have a most recent common ancestor: for any $s,t\in[0,1]$, there is a unique $u\in[0,1]$ such that for all $r\in[0,1]$, $r\preceq u\Longleftrightarrow r\preceq s,t$. We denote it by $s\wedge_f t$, or just by $s\wedge t$ if the excursion $f$ is obvious within context.
\item[$(iv)$] The ancestral lineages are totally ordered: for any $t\in[0,1]$,  $\preceq$ induces a total order on $\{s\in[0,1] : s\preceq t\}$.
\end{enumerate}
\end{proposition}

\begin{proof}
Let $t,s,r\in[0,1]$. Obviously, we have $t\leq t$ and $f(t-)\leq f(t)=\inf_{[t,t]}f$ because $f$ is an excursion. So $t\preceq t$ and the relation $\preceq$ is reflexive. If $s\preceq t$ and $t\preceq s$, then by definition $s\leq t\leq s$, so $t=s$, and $\preceq$ is antisymmetric. To prove the transitivity, we suppose that $s\preceq t$ and $t\preceq r$ and we can assume $s< t< r$. We find $f(s-)\leq \inf_{[s,t]}f\leq f(t-)\leq \inf_{[t,r]}f$, which entails $s\preceq r$ and $(i)$. The point $(ii)$ is simply deduced from the fact that $f$ is an excursion. Let us prove $(iii)$. We set $u=\sup\{r\in[0,1]\ :\ r\preceq s,t\}$. By definition, $u\leq s$ and $u\leq t$. We give ourselves a sequence $(r_n)_{n\geq1}$ that converges towards $u$ and such that $0\leq r_n\leq r_{n+1}\leq u$ and $r_n\preceq s,t$ for all $n\geq1$. We have $\displaystyle{f(r_n-)\longrightarrow f(u-)}$ and thanks to \[f(r_n-)\leq\inf_{[r_n,t]} f\leq\inf_{[u,t]} f,\] we obtain $u\preceq t$. Likewise, $u\preceq s$. Conversely, if $r\preceq t$ and $r\preceq s$, then $r\leq u$ by definition and $f(r-)\leq\inf_{[r,t]}f\leq\inf_{[r,u]}f$, so $r\preceq u$. The uniqueness of $s\wedge t$ immediately follows from the antisymmetry of $\preceq$. Let us prove $(iv)$. We suppose $s,r\preceq t$ and without loss of generality, we assume $r\leq s$. Then, $f(r-)\leq\inf_{[r,t]}f\leq\inf_{[r,s]}f$, so $r\preceq s$.
\end{proof}

\begin{lemma}
\label{autre_utile}
Let $s,t\in[0,1]$. If $s\leq t$, then it holds that $\inf_{[s\wedge t,t]}f=\inf_{[s,t]}f$ and \[s\wedge t=\sup\left\{r\leq s\ :\ r\preceq t\right\}=\sup\left\{r\leq s\ :\ f(r-)\leq \inf_{[s,t]}f\right\}.\]
\end{lemma}

\begin{proof}
If $r\leq s$ and $r\preceq t$, then $f(r-)\leq\inf_{[r,t]}f\leq\inf_{[r,s]}f$, and so $r\preceq s$. Thus, $r\preceq s\wedge t$. Conversely, if $r\preceq s\wedge t$, then $r\leq s$ and $r\preceq t$. So, we have shown $s\wedge t=\sup\{r\leq s\ :\ r\preceq t\}$. In particular, if $s\wedge t<r\leq s$, then we cannot have $r\preceq t$. It follows that for all $r\in(s\wedge t,s]$, it holds $f(r)\geq f(r-)>\inf_{[r,t]}f.$ Since $f$ is c\`adl\`ag, this leads to $\inf_{[r,s]}f>\inf_{[r,t]}f$. Hence, $\inf_{[r,t]}f=\inf_{[s,t]}f$ for all $r\in(s\wedge t,s]$ and we end the proof by making $r$ tend towards $s\wedge t$.
\end{proof}

For $s,t\in[0,1]$, recall from (\ref{delta_x}) the quantity $x_s^t$ that represents, when $s\preceq t$, the relative position of the ancestor of $t$ on the loop associated with $s$. This interpretation is justified by the following lemma, linking the genealogy and the $x_s^t$.

\begin{lemma}
\label{tri_utile}
Let $u,s,t\in[0,1]$. If $u< t\wedge s$, then $x_u^t=x_u^s=x_u^{t\wedge s}$.
\end{lemma}
\begin{proof}
We merely write $\inf_{[u,t\wedge s]}f\leq f(t\wedge s-)\leq\min\left(\inf_{[t\wedge s,t]}f,\inf_{[t\wedge s,s]}f\right)$.
\end{proof}

Finally, we provide a bound for sums of $x_s^t$ that will be useful to control the pseudo-distances coded by excursions.
\begin{equation}
\label{9}
\textit{If }\quad 0\leq s<t\leq 1\quad\textit{ then }\quad\sum_{r\in(s,t)}x_r^t=\sum_{\substack{s<r\\r\prec t}}x_r^t\leq f(t-)-\inf_{[s,t]}f.
\end{equation}

\begin{proof}[Proof of (\ref{9})]
Note that if $x_r^t>0$ then $r\preceq t$, which yields the first equality. Let $r_i,1\leq i\leq n,$ be some elements of $\{r\in[0,1]\ :\ r\prec t,s<r\}$ such that $r_i<r_{i+1}$ for any $1\leq i\leq n-1$. Let us also set $r_{n+1}=t$. With these notations, we have $r_{i+1}\in(r_i,t]$ for any $1\leq i\leq n$, so we can write \[\sum_{i=1}^{n} x_{r_i}^t=\sum_{i=1}^{n} \left(\inf_{[r_i,t]}f-f(r_i-)\right)\leq \sum_{i=1}^{n} f(r_{i+1}-)-f(r_i-)=f(t-)-f(r_1-).\]
Eventually, $f(r_1-)\geq\inf_{[s,t]}f$ because $r_1\in (s,t]$. The lemma follows.
\end{proof}

\subsection{Pseudo-distances coded by an excursion}
\label{pseudo_dist_coded_sec}

Let us fix an excursion $f$, as in Definition~\ref{excursion}, and a shuffle $\Phi$, as in Definition~\ref{shuffle}. Here, we prove the following result.

\begin{proposition}
\label{thm_distance_d}
The functions $\dl$, $\dlt$, and $\dtree$ respectively defined by (\ref{dl_easy}), (\ref{dlt_easy}), and (\ref{dtree_easy}) are pseudo-distances on $[-1,1]$. As sums of pseudo-distances, the functions $\dv$ and $\dvt$ given by (\ref{dv_easy}) and (\ref{dlt_easy}) are pseudo-distances on $[-1,1]$.
\end{proposition}

Observe that for any $s\in[-1,0)$, $\dl(t,s)=\dl(t,1)$ and $\dtree(t,s)=\dtree(t,1)$ for all $t\in[-1,1]$. Thus, we will focus on proving that $\dl$ and $\dtree$ are pseudo-distances on $[0,1]$. But first, we give alternative expressions to relate them to the genealogy $\preceq$ given by (\ref{genealogy_def}) and studied in Section~\ref{genealogy_sec}. For all $r\in[0,1]$, recall from (\ref{circle_dist_s}) and (\ref{circle_shuffled_dist_s}) the pseudo-distances $\delta_r$ and $\tilde{\delta}_r$ on $[0,\Delta_r]$ that induce the loop and the shuffled loop (of length $\Delta_r$) associated with $r$. Let $s,t\in[0,1]$, note from (\ref{delta_x}) that if $r\preceq s$ does not hold then $x_r^s=0$. Plus, Lemma~\ref{tri_utile} ensures that if $r\prec t\wedge s$ then $x_r^s=x_r^t$. Therefore, Proposition~\ref{genealogy_prop} $(iii)$ yields the two identities
\begin{align*}
\delta_r(x_r^s,x_r^t)&=\delta_r(0,x_r^s)\I{s\wedge t\prec r\preceq s}+\delta_r(0,x_r^t)\I{s\wedge t\prec r\preceq t}+\delta_{s\wedge t}(x_{s\wedge t}^s,x_{s\wedge t}^t)\I{s\wedge t=r},\\
\tilde{\delta}_r(x_r^s,x_r^t)&=\tilde{\delta}_r(0,x_r^s)\I{s\wedge t\prec r\preceq s}+\tilde{\delta}_r(0,x_r^t)\I{s\wedge t\prec r\preceq t}+\tilde{\delta}_{s\wedge t}(x_{s\wedge t}^s,x_{s\wedge t}^t)\I{s\wedge t=r}.
\end{align*}
They allow us to recover the expression of $\dl$ proposed by Curien \& Kortchemski~\cite{curien2014}: for all $s,t\in[0,1]$, it holds that
\begin{align}
\label{dl_hard}
\dl(s,t)&=\delta_{s\wedge t}(x_{s\wedge t}^s,x_{s\wedge t}^t)+\sum_{s\wedge t\prec r\preceq s}\delta_r(0,x_r^s)+\sum_{s\wedge t\prec r\preceq t}\delta_r(0,x_r^t),\\
\label{dlt_hard}
\dlt(s,t)&=\tilde{\delta}_{s\wedge t}(x_{s\wedge t}^s,x_{s\wedge t}^t)+\sum_{s\wedge t\prec r\preceq s}\tilde{\delta}_r(0,x_r^s)+\sum_{s\wedge t\prec r\preceq t}\tilde{\delta}_r(0,x_r^t).
\end{align}
Basically, a minimal path on the looptree linking $s$ and $t$ visits the loop associated with $s\wedge t$ but does not go below. Then, each term of the sum corresponds to the length of such a path getting through one of the loops between $s$ and $t$.

Similarly, we can write $|x_r^s-x_r^t|=x_r^s\I{s\wedge t\prec r\preceq s}+x_r^t\I{s\wedge t\prec r\preceq t}+|x_{s\wedge t}^s-x_{s\wedge t}^t|\I{s\wedge t=r}$. Thus, if $s\leq t$ then
\begin{equation}
\label{dtree_hard}
\dtree(s,t)=f(s)+f(t)-2\inf_{[s,t]}f-(x_{s\wedge t}^s-x_{s\wedge t}^t)-\sum_{s\wedge t\prec r\preceq s}x_r^s-\sum_{s\wedge t\prec r\preceq t}x_r^t.
\end{equation}
In particular, we get a simpler formula when $s=s\wedge t$, i.e.~when $s$ is an ancestor of $t$, thanks to Lemma~\ref{autre_utile}:
\begin{equation}
\label{dtree_particular}
\textit{if }\quad s\preceq t\quad\textit{ then }\quad \dtree(s,t)=f(t)-\inf_{[s,t]}f-\sum_{s\prec r\preceq t}x_r^t.
\end{equation}

Let $t\in[0,1]$, the formulas (\ref{dl_easy}), (\ref{dtree_easy}), and (\ref{dlt_easy}) readily imply that $\dl(t,t)=\dtree(t,t)=\dlt(t,t)=0$. They also entail the symmetries of $\dl,\dtree,\dlt$. We now show our quantities are nonnegative and finite. The functions $\dl$ and $\dlt$ are obviously nonnegative because they are series of nonnegative terms. Keeping only the term for $r=t$ in the series in (\ref{dtree_particular}) yields that if $s\prec t$ then $\dtree(s,t)\leq f(t-)-\inf_{[s,t]}f$ . Moreover, we apply (\ref{9}) to get a lower bound. Thus, for all $s,t\in[0,1]$,
\begin{equation}
\label{maj_dtreeO}
\textit{if }\quad s\prec t\quad\textit{ then }\quad0\leq\dtree(s,t)\leq f(t-)-\inf_{[s,t]}f.
\end{equation}
If we only have $s\preceq t$, the upper bound given by $f(t)-\inf_{[s,t]}f$ stays true. Now, we only assume $s<t$. Note $t\wedge s\leq s<t$ and $\delta_{t\wedge s}(x_{s\wedge t}^s,x_{s\wedge t}^t)\leq x_{t\wedge s}^s-x_{t\wedge s}^t$. Plus, we have $\delta_r(0,x_r^r)=\delta_r(0,\Delta_r)=0$ and $\delta_r(0,x_r^t)\leq x_r^t$ for all $t,r\in[0,1]$. Also recall from Lemma~\ref{autre_utile} that $\inf_{[s,t]}f=\inf_{[s\wedge  t,t]}f$. Eventually, if $s<t$ then (\ref{9}), (\ref{dl_hard}), (\ref{dtree_hard}) and (\ref{maj_dtreeO}) yield
\begin{align}
\label{maj_dl}
0\leq \dl(s,t)&\leq f(s)+f(t-)-2\inf_{[s,t]}f,\\
\label{maj_dtree}
0\leq \dtree(s,t)&\leq f(s)+f(t-)-2\inf_{[s,t]}f.
\end{align}
Remark we cannot find these inequalities for $\dlt$ because $\tilde{\delta}_r(0,\Delta_r)\neq 0$ and $\tilde{\delta}_r(a,b)\nleq |a-b|$. Nethertheless, the bound in (\ref{11}) ensures that there is a constant $K\in(0,\infty)$ which only depends on $||f||_\infty$ and $\Phi$ such that $\tilde{\delta}_r(0,x)\leq Kx$ for any $t\in[0,1]$ and $x\in[0,\Delta_r]$. Plus, we still have $\inf_{[s,t]}f=\inf_{[s\wedge  t,t]}f\leq \inf_{[s\wedge t,s]}f$. Thus, if $s\leq t$ then (\ref{9}) and (\ref{dlt_hard}) imply
\begin{equation}
\label{maj_dlt}
0\leq \dlt(s,t)\leq \tilde{\delta}_{s\wedge t}(x_{s\wedge t}^s,x_{s\wedge t}^t)+K\left(f(s)+f(t)-2\inf_{[s,t]}f\right).
\end{equation}
The next result asserts that $\dtree$ behaves like a metric of a tree-like space, and will be needed to get the triangular inequality.

\begin{proposition}
\label{geo_tree}
For all $r,s,t\in[0,1]$, if $r\preceq s\preceq t$ then $\dtree(r,t)=\dtree(r,s)+\dtree(s,t)$. Moreover, it always holds that 
\begin{equation}
\label{four_points_equality}
\dtree(s,t)=\dtree(s\wedge t,s)+\dtree(s\wedge t,t)=\dtree(0,s)+\dtree(0,t)-2\dtree(0,s\wedge t).
\end{equation}
\end{proposition}

\begin{proof}
We first treat the case $r=0$ and we may assume that $s>0$ because we have $\dtree(0,0)=0$. This allows us to write 
\[\dtree(0,t)=\dtree(s,t)+\inf_{[s,t]}f-\inf_{[0,t]}f-x_s^t-\sum_{0\prec u\prec s}x_u^t,\]
by rearranging (\ref{dtree_particular}). Next, we use the equality $x_0^t=\inf_{[0,t]}f$ and we apply Lemma~\ref{tri_utile} with $s=s\wedge t$ to get
\[\dtree(0,t)=\dtree(s,t)+\inf_{[s,t]}f-x_0^t-x_s^t-\sum_{0\prec u\prec s}x_u^s=\dtree(s,t)+f(s-)-x_0^t+\dtree(0,s)+x_s^s-f(s)+x_0^s.\]
Lemma~\ref{tri_utile} entails that $x_0^t=x_0^s$ because $0<s$. The desired result follows from recognizing $x_s^s=f(s)-f(s-)$. Then,
\[\dtree(r,t)=\dtree(0,t)-\dtree(0,r)=\dtree(0,s)+\dtree(s,t)-\dtree(0,r)=\dtree(r,s)+\dtree(s,t),\]
which is the general result. Now, we assume that $s\leq t$ so that Lemma~\ref{autre_utile} gives $\inf_{[s\wedge t,t]}f=\inf_{[s,t]}f$. Combining (\ref{dtree_hard}) and (\ref{dtree_particular}) yields the first equality in (\ref{four_points_equality}). The second one follows the first assertion of the proposition as $0\preceq s\wedge t\preceq s,t$.
\end{proof}

\begin{proof}[Proof of Proposition~\ref{thm_distance_d}]
The only thing left to show is the triangular inequality. The functions $\dl$ and $\dlt$ are sums of the pseudo-distances $\delta_u$ and $\tilde{\delta}_u$, for $u\in[0,1]$, so they readily enjoy the triangular inequality. We now focus on $\dtree$. Let $t,s,r\in[0,1]$. We know that either $s\wedge r\preceq t\wedge s$ or $t\wedge s\preceq s\wedge r$, because the ancestral lineage of $s$ is totally ordered. By symmetry, we can assume $s\wedge r\preceq t\wedge s$ without loss of generality. Then, $s\wedge r\preceq t$ which yields that $s\wedge r\preceq t\wedge r$. Proposition~\ref{geo_tree} and the fact that $\dtree$ is nonnegative, known from (\ref{maj_dtree}), allow us to write
\[\dtree(t,r)=\dtree(t\wedge r,t)+\dtree(t\wedge r,r)\leq \dtree(s\wedge r,t)+\dtree(s\wedge r,r)=\dtree(s\wedge r,t\wedge s)+\dtree(t\wedge s,t)+\dtree(s\wedge r,r).\]
The same argument gives that $\dtree(s\wedge r,t\wedge s)\leq \dtree(s\wedge r,s)=\dtree(s,r)-\dtree(s\wedge r,r)$ and $\dtree(t\wedge s,t)\leq \dtree(t,s)$.
\end{proof}

\begin{remark}
The triangular inequality for $\dtree$ will become even clearer later. Indeed, we will prove with Theorem~\ref{J_prop} that there exists a continuous excursion $g$ with $g(0)=0$ such that $\dtree[f]=d_g$, where $d_g$ is the classic tree pseudo-distance defined by (\ref{classical_tree_distance}). In fact, Marzouk~\cite{marzouk_similaire} defines the same pseudo-distance $\dtree$ as such.  \cq 
\end{remark}

\subsection{Regularity of the pseudo-distances, and compactness of the induced quotient metric spaces}
\label{space}

Let $d$ be a pseudo-distance on a set $X$. Writing $x\sim_d y$ when $d(x,y)=0$ defines an equivalence relation on $X$. Then $d$ induces a genuine distance (still denoted by $d$ with a slight abuse of notation) on the quotient space $X/\!\sim_d$. More precisely, if we denote by $\mathsf{p}:X\to X/\! \sim_d$ the canonical projection, then $d(\mathsf{p}(x),\mathsf{p}(y))=d(x,y)$ for all $x,y\in X$. We call the metric space $(X/\!\sim_d,d)$ the \emph{quotient metric space of $X$ induced by $d$}. Moreover, if $X$ is a compact metric space and if $d$ is continuous on $X^2$, then $\mathsf{p}$ is continuous and $X/\!\sim_d=\mathsf{p}(X)$ is compact. Let $f$ be an excursion. Recall from Sections~\ref{intro_desc} and \ref{intro_shuffling} that $\lop[f],\tree[f],\vern[f],\lopt[f],\vernt[f]$ are the quotient metric spaces of $[0,1]$ respectively induced by $\dl,\dtree,\dv,\dlt,\dvt$.

\begin{proposition}
\label{d_continu}
The functions $\dl,\dtree,\dv$ are continuous on $[-1,1]^2$. The metric spaces $\lop[f],\tree[f],\vern[f]$ are thus compact.
\end{proposition}

\begin{proof}
Note that if $s\in[-1,0]$ then $\dl(t,s)=\dl(t,1)$ and $\dtree(t,s)=\dtree(t,1)$ for all $t\in[-1,1]$. As $\dl$ is a pseudo-distance, we only need to show that $\dl(t,t_n)\to 0$ when $t_n\uparrow t$ or $t_n\downarrow t$ on $[0,1]$. Thanks to (\ref{maj_dl}), if $t_n<t<t_m$ then 
\begin{align*}
\dl(t,t_n)\leq f(t-)+f(t_n)-2\inf_{[t_n,t]}f&\underset{t_n\rightarrow t-}{\longrightarrow} f(t-)+f(t-)-2f(t-)=0,\\
\dl(t,t_m)\leq f(t)+f(t_m-)-2\inf_{[t,t_m]}f&\underset{t_m\rightarrow t+}{\longrightarrow} f(t)+f(t)-2f(t)=0
\end{align*}
by Definition~\ref{excursion}. We prove the continuity of $\dtree$ in the same way with (\ref{maj_dtree}). The continuity of $\dv=\dtree+2\dl$ follows.
\end{proof}

The pseudo-distance $\dlt$ (as well as $\dvt$) is not continuous in general because of the possible discontinuities of the $\phi_\Delta$ given by the shuffle $\Phi$ (recall Definition~\ref{shuffle}). However, it still enjoys some regularity in the following sense.
\begin{definition}
\label{bivariate}
We denote by $\mathbb{D}([-1,1]^2)$ the set of functions $\psi:[-1,1]^2\longrightarrow\mathbb{R}$ such that for any monotonous sequences $(s_n)$ and $(t_n)$ of elements of $[-1,1]$, the sequence $(\psi(s_n,t_n))$ converges, and such that its limit is $\psi(\lim s_n,\lim t_n)$ when $(s_n)$ and $(t_n)$ are non-increasing. We provide a distance $\rho_2$ to $\mathbb{D}([-1,1]^2)$ by setting for all $\psi_1,\psi_2\in\mathbb{D}([-1,1]^2)$,
\begin{equation}
\label{sko_dist_biv}
\rho_2(\psi_1,\psi_2)=\inf_{\lambda,\mu\in\Lambda}\max\left(||\psi_1-\psi_2\circ(\lambda,\mu)||_{\infty},||\lambda-\mathsf{id}||_\infty,||\mu-\mathsf{id}||_\infty\right),
\end{equation}
 where $\Lambda$ is the set of increasing bijections from $[-1,1]$ into itself. The distance $\rho_2$ is not complete but induces a separable and completely metrizable topology on $\mathbb{D}([-1,1]^2)$, that we call the Skorokhod topology on $[-1,1]^2$.\cq
\end{definition}

\noindent
\begin{remark}
\label{bivariate_unif}
This is a generalization of the Skorokhod space for bivariate functions, and is more precisely presented by Straf~\cite{straf1972}. Let us mention that if $(\psi_n)$ converges to $\psi$ for this topology and if $\psi$ is continuous on $[-1,1]^2$, then the convergence also happens uniformly on $[-1,1]^2$. Indeed, this follows from the uniform continuity of $\psi$ on $[-1,1]^2$. \cq
\end{remark}

\begin{proposition}
\label{dt_continu}
The canonical projection maps from $[0,1]$ to $\lopt[f]$ or $\vernt[f]$ are c\`adl\`ag. The metric spaces $\lopt[f]$ and $\vernt[f]$ are compact, and the functions $\dlt$ and $\dvt$ are in $\mathbb{D}([-1,1]^2)$.
\end{proposition}

\begin{proof}
Note that if $s\in[-1,0)$ then $\dlt(t,s)=\dlt(t,1)$ for all $t\in[-1,1]$. By triangular inequality, it is enough to show that if $s_n\uparrow s$ and $t_n\downarrow t$ on $[0,1]$, then $\dlt(t_n,t)\to 0$ and there is $\bar{s}\in[0,1]$ such that $\dlt(s_n,\bar{s})\to 0$ and $\dtree(s,\bar{s})=0$. The compactness of $\lopt[f]$ and $\vernt[f]$ would then follow from that of $[0,1]$. We begin by proving $\dlt(t,t_n)\to 0$. Since $f$ is c\`adl\`ag, we get $f(t)+f(t_n)-2\inf_{[t,t_n]}f\longrightarrow 0$, so the inequality (\ref{maj_dlt}) entails that we only need to show that $\tilde{\delta}_{t\wedge t_n}(x_{t\wedge t_n}^t,x_{t\wedge t_n}^{t_n})$ tends to $0$. Thanks to Lemma~\ref{autre_utile}, we observe that the sequence $(t\wedge t_n)$ is non-decreasing. Moreover, the set $\{r\in[0,1]\ :\ \Delta_r\geq\varepsilon\}$ is finite for all $\varepsilon>0$, so only two cases are possible: either $\Delta_{t\wedge t_n}\longrightarrow0$, or there exists $\tau\in[0,1]$ with $\Delta_\tau>0$ such that $t\wedge t_n=\tau$ for all large enough $n$. In the first case, we have $\tilde{\delta}_{t\wedge t_n}(x_{t\wedge t_n}^t,x_{t\wedge t_n}^{t_n})\leq\Delta_{t\wedge t_n}\longrightarrow0$. In the second case, when $n$ is large enough, we have $\tilde{\delta}_{t\wedge t_n}(x_{t\wedge t_n}^t,x_{t\wedge t_n}^{t_n})=\tilde{\delta}_\tau(x_\tau^t,x_\tau^{t_n})$. But, $x_\tau^{t_n}\uparrow x_\tau^t$ because $f$ is c\`adl\`ag. By Definition~\ref{shuffle}, the function $\phi_{\Delta_{\tau}}$ is c\`agl\`ad, so $\tilde{\delta}_\tau(x_\tau^t,x_\tau^{t_n})\longrightarrow \hat{\delta}_\tau(x_\tau^t,x_\tau^{t})=0$. Thus, $\dlt(t,t_n)\to 0$.
\smallskip

We set $\displaystyle{s'=\inf\{r\geq s\ :\ f(r)\leq f(s-)\}}$, so that $\inf_{[s,s']}f=f(s-)=f(s')=f(s'-)$ because $f$ is an excursion. Thus, $s\preceq s'$ and $\dtree(s,s')=0$ by the inequality (\ref{maj_dtreeO}). The sequence $(s_n\wedge s')$ is non-decreasing according to Lemma~\ref{autre_utile}, so as before, either $\Delta_{s_n\wedge s'}\longrightarrow0$, or there exists $\sigma\in[0,1]$ with $\Delta_\sigma>0$ such that $s_n\wedge s'=\sigma$ for $n$ large enough.
\begin{itemize}
\item In the first case: we set $\bar{s}=s'$ and get $\dtree(s,\bar{s})=0$. As $f$ is càdlàg, it holds $f(\bar{s})+f(s_n)-2\inf_{[s_n,\bar{s}]}f\to 0$ by definition of $s'=\bar{s}$. Plus, we have $\tilde{\delta}_{\bar{s}\wedge s_n}(x_{\bar{s}\wedge s_n}^{s_n},x_{\bar{s}\wedge s_n}^{\bar{s}})\leq\Delta_{s_n\wedge s'}\longrightarrow0$ here. Then (\ref{maj_dlt}) gives $\dlt(s_n,\bar{s})\longrightarrow0$.

\item In the second case: when $n$ is large enough, it holds $\sigma= s_n\wedge s'\preceq s'$ so $\inf_{[\sigma,s']}f=\inf_{[s_n,s']}f$ according to Lemma~\ref{autre_utile}. But, we see $\inf_{[s_n,s']}f\longrightarrow f(s')$ by definition of $s'$, so the inequality (\ref{maj_dtreeO}) yields $\dtree(\sigma,s')=0$. We already know $\dtree(s,s')=0$ so $\dtree(\sigma,s)=0$. Let us set \[\bar{s}=\inf\left\{r\geq\sigma\ :\ \phi_{\Delta_\sigma}\left(\frac{f(r)-f(\sigma-)}{\Delta_\sigma}\right)=\lim_{n\to\infty}\phi_{\Delta_\sigma}\left(\frac{x_\sigma^{s_n}}{\Delta_\sigma}\right)\right\},\]which is well-defined because $\phi_{\Delta_\sigma}$ is surjective and càglàd, by Definition~\ref{shuffle}, and because $f$ is an excursion. Using Definition~\ref{excursion} of excursions, we find $\sigma\preceq \bar{s}$ and $f(\bar{s})=\inf_{[\sigma,\bar{s}]}f$. As a result, $\sigma\preceq s_n\wedge\bar{s}$ when $n$ is large enough and if $r\in(\sigma,\bar{s}]$ then $x_r^{\bar{s}}=0$. It also follows from (\ref{maj_dtreeO}) that $\dtree(\sigma,\bar{s})=0=\dtree(s,\bar{s})$. Then, (\ref{dlt_hard}) entails that
\[\dlt(s_n,\bar{s})\leq \dlt(s_n,s')-\tilde{\delta}_{s_n\wedge s'}(x_{s_n\wedge s'}^{s_n},x_{s_n\wedge s'}^{s'})+\tilde{\delta}_\sigma(x_\sigma^{s_n},x_\sigma^{\bar{s}})\]
when $n$ is large enough. We check $\phi_{\Delta_\sigma}(x_\sigma^{\bar{s}}/\Delta_\sigma)=\lim\phi_{\Delta_\sigma}(x_\sigma^{s_n}/\Delta_\sigma)$ using the left-continuity of $\phi_{\Delta_\sigma}$. Thus, $\tilde{\delta}_\sigma(x_\sigma^{s_n},x_\sigma^{\bar{s}})\longrightarrow 0$. We complete the proof with (\ref{maj_dlt}) because $f(s')+f(s_n)-2\inf_{[s_n,s']}f\longrightarrow 0$, as seen above.
\end{itemize}
\end{proof}

\subsection{First properties and an example}
\label{fst_properties_sec}

A \emph{pointed metric space} is a triple $(X,d,a)$ where $(X,d)$ is a metric space equipped with a distinguished point $a\in X$, also called a root. We say two pointed metric spaces $(X_1,d_1,a_1)$ and $(X_2,d_2,a_2)$ are \emph{pointed-isometric} when there exists a bijective isometry $\lambda$ from $X_1$ to $X_2$ such that $\lambda(a_1)=a_2$. A \emph{pointed weighted metric space} is a quadruple $(X,d,a,\mu)$ where $(X,d,a)$ is a pointed metric space also equipped with a Borel probability measure $\mu$ on $(X,d)$. We say two pointed weighted metric spaces $(X_1,d_1,a_1,\mu_1)$ and $(X_2,d_2,a_2,\mu_2)$ are \emph{GHP-isometric} when there exists a bijective isometry $\lambda$ from $X_1$ to $X_2$ such that $\lambda(a_1)=a_2$ and $\lambda_*\mu_1=\mu_2$. When no confusion is possible, we shall denote a metric space (possibly endowed with a root and/or a Borel probability measure) by its underlying set. Finally, when $c>0$, we write $c\cdot (X,d,a,\mu)=(X,cd,a,\mu)$ for the pointed weighted metric space obtained after multiplying all distances by $c$. As same, we write $c\cdot (X,d,a)=(X,cd,a)$ and $c\cdot (X,d)=(X,cd)$.

Recall from Definition~\ref{excursion} that $\mathbb{H}$ stands for the space of excursions. We remind from Section~\ref{intro_desc} and \ref{intro_shuffling} that for $f\in\mathbb{H}$, $\lop[f],\tree[f],\vern[f],\lopt[f],\vernt[f]$ are the quotient metric spaces of $[0,1]$ resp.~induced by $\dl,\dtree,\dv,\dlt,\dvt$. Each is equipped with the canonical projection of $1$ and with the pushforward measure of the Lebesgue measure on $[0,1]$. Thus, they are all pointed weighted metric spaces. Here, we provide some properties that make these codings easy to use in practice.

\begin{proposition}
\label{time-change}
The maps $f\in\mathbb{H}\mapsto \dl[f],\dtree[f],\dv[f]$ are homogeneous, and the maps $f\in\mathbb{H}\mapsto \dl[f],\dtree[f],\dlt[f],\dv[f],\dvt[f]$ have the time-change property. These properties are respectively defined below for $\dl[f]$.
\begin{enumerate}
\item[(Homogeneity)] For any $\alpha>0$, it holds that $\dl[\alpha f]=\alpha \dl[f]$. Thus, $\lop[\alpha f]$ and $\alpha\cdot\lop[f]$ are GHP-isometric.

\item[(Time-change)] For any increasing bijection $\lambda:[-1,1]\to [-1,1]$ with $\lambda(0)\leq 0$, it holds that $\dl[f\circ \lambda](s,t)=\dl[f](\lambda(s),\lambda(t))$ for all $s,t\in[-1,1]$. Thus, $\lop[f\circ \lambda]$ and $\lop[f]$ are pointed-isometric.
\end{enumerate}
\end{proposition}

\begin{proof} The first thing to see is that $\alpha f$ and $f\circ \lambda$ are indeed excursions. Then, it is obvious that $x_s^t(\alpha f)=\alpha x_s^t(f)$ and $x_s^t(f\circ \lambda)=x_{\lambda(s)}^{\lambda(t)}(f)$ for all $s,t\in[-1,1]$. In particular, $\Delta_s(\alpha f)=\alpha\Delta_s(f)$ and $\Delta_s(f\circ\lambda)=\Delta_{\lambda(s)}(f)$. By definition (\ref{circle_dist_s}) and (\ref{circle_shuffled_dist_s}) of $\delta_s$ and $\tilde{\delta}_s$, the proposition follows from (\ref{dl_easy}), (\ref{dtree_easy}), (\ref{dv_easy}) and (\ref{dlt_easy}). Let us point out that we have used $\delta_s(\alpha a,\alpha b)=\alpha\delta_s(a,b)$ which does not hold for $\tilde{\delta}_s$, which explains why $f\longmapsto\dlt[f]$ and $f\longmapsto\dvt[f]$ are not homogeneous.
\end{proof}

\noindent
Another important property of the codings is the branching property. It asserts that a 'sub-excursion' codes a subspace of the looptree or the vernation tree. It allows observing the structure of the coded space directly on the graph of the excursion. We have already presented this idea when we defined the genealogy $\preceq$ in Section~\ref{genealogy_sec}.

\begin{definition}
\label{gluing}
Let $(X_0,d_0,a_0)$ and $(X_1,d_1,a_1)$ be two pointed metric spaces and let $a\in X_0$. We write $a_0^*=a$ and $a_1^*=a_1$. The \emph{gluing} of $X_1$ on $X_0$ at $a$ is the pointed metric space denoted by $(X_0\vee_a X_1,d,a_0)$ that is obtained from the quotient of the disjoint union $X_0\sqcup X_1$ by the identification $a\sim a_1$, endowed with the distance $d$ defined by setting $d(x,y)=d_i(x,y)$ if $x,y\in X_i$ with $i\in\{0,1\}$, and by setting $d(x,y)=d_i(x,a_i^*)+d_{1-i}(a_{1-i}^*,y)$ if $x\in X_i$ and $y\in X_{1-i}$ with $i\in\{0,1\}$. We view $X_0$ and $X_1$ as closed subsets of $X_0\vee_a X_1$, so that $X_0\cup X_1=X_0\vee_a X_1$, $X_0\cap X_1=\{a\}=\{a_1\}$, and the distinguished points of $X_0$ and $X_0\vee_a X_1$ are the same.\cq
\end{definition}

\begin{proposition}
\label{branching_property}
The map $f\in\mathbb{H}\longmapsto \dl[f]$ (respectively $\dtree[f]$, $\dlt[f]$, $\dv[f]$, $\dvt[f]$) enjoys the branching property. Namely, let $u,v\in[0,1]$ be such that $u\prec v$ and $f(u-)=f(v-)$, and let us set 
\[g:t\in[-1,1]\longmapsto \begin{cases}
						f(t)&\text{ if }t\notin[u,v),\\
						f(u-)&\text{ if }t\in[u,v),
						\end{cases}\quad\text{ and }\quad h:t\in[-1,1]\longmapsto \begin{cases}
                        f(u+t(v-u))-f(u-)&\text{ if }t\in[0,1),\\
                        0&\text{ if }t\notin[0,1).
                        \end{cases}\]
If $s,t\in[0,1)$ and $a,b\in[0,1]\backslash[u,v)$, then $\dl[f](a,b)=\dl[g](a,b)$ and $\dl[f](u+s(v-u),u+t(v-u))=\dl[h](s,t)$, and
\[\dl[f](u+s(v-u),a)=\dl[g](a,u)+\dl[h](s,1).\]
Thus, the pointed metric space $\lop[f]$ is pointed-isometric to the gluing of $\lop[h]$ on $\lop[g]$ at the canonical projection of $u$.
\end{proposition}

\begin{proof}
First, we observe that $g$ and $h$ are indeed excursions. We only show the result for $\dl$ and $\dtree$, as the same proof holds for $\dlt$ and because the branching property is preserved by linear combination. Let $a,b\in[0,1]\backslash[u,v)$ and $s,t\in[0,1)$. We assume $a\leq b$ and $s\leq t$, and we set $w=u+s(v-u)$ to lighten the notations. We immediately observe that $x_{u+s(v-u)}^{u+t(v-u)}(f)=x_s^t(h)$, and if $[u,v)\cap[a,b]=\emptyset$ then $x_a^b(f)=x_a^b(g)$. If $[a,b]\subset [u,v)$ then we still have $x_a^b(f)=x_a^b(g)$ because $\inf_{[u,v)}f=f(u-)=\inf_{[u,v)}g$. Concerning $x_a^w$, if $w<a$ then $x_a^w(f)=x_a^u(g)=0$ by definition, since $u\leq w$. If $a\leq w$ then $a<u\leq w<v$, so $x_a^w(f)=x_a^u(g)$ still holds, because\[\inf_{[u,w]}f\geq \inf_{[u,v]}f\geq f(u-)=\inf_{[u,w]}g\geq\inf_{[a,u]}f=\inf_{[a,u]}g.\] Concerning $x_w^a$, if $a<w$ then $x_w^a(f)=x_w^a(g)=0$ by definition. If $w\leq a$ then $u\leq w<v\leq a$, so $x_w^a(f)=x_w^a(g)=0$ still holds because \[f(w-)\geq \min\left(f(u-),\inf_{[u,v]}f\right)\geq f(u-)=g(w-)=f(v-)=g(v-)\geq \max\left(\inf_{[w,a]}f,\inf_{[w,a]}g\right).\] Moreover, $x_w^u(g)\leq \Delta_w(g)=0$ and $x_s^1(h)\leq h(1)=0$, so that $x_w^u(g)=x_s^1(h)=0$. To sum up, we have 
\begin{align*}
\Delta_a(f)=\Delta_a(g)&\quad\text{ and }\quad\Delta_{u+s(v-u)}(f)=\Delta_s(h),\\
x_a^b(f)=x_a^b(g)&\quad\text{ and }\quad x_{u+s(v-u)}^{u+t(v-u)}(f)=x_s^t(h),\\
x_a^w(f)=x_a^u(g)&\quad\text{ and }\quad x_w^a(f)=x_w^a(g)=x_w^u(g)=x_{s}^1(h)=0.
\end{align*}
Combining these identities with the formula (\ref{dl_easy}), we obtain the branching property for $\dl$. Some easy other observations together with the formula (\ref{dtree_easy}) and the above identities lead to the branching property for $\dtree$.
\end{proof}

Let us discuss how one can visualize the looptree or the vernation tree coded by an excursion $f$ from a glance at its graph. On the one hand, if $f$ does not have any jump, then (\ref{dl_easy}) and (\ref{dtree_easy}) ensure that $\dl=0$ and that $\dtree$ is equal to the classic tree pseudo-distance defined by (\ref{classical_tree_distance}). On the other hand, each jump of $f$ indeed corresponds to a loop in the associated looptree or vernation tree. Recall from (\ref{circle_dist_cano}) that $(\mathcal{C},\delta)$ is the circle with unit length. Let $s\in[0,1]$ such that $\Delta_s>0$ and let $x\in\mathcal{C}$, the following times are well-defined and satisfy $x_s^{\tau(x)}/\Delta_s=x$ and $\phi_{\Delta_s}(x_s^{\tilde{\tau}(x)}/\Delta_s)=x$:
\[\tau(x)=\inf\left\{t\geq s\ :\ \frac{f(t)-f(s-)}{\Delta_s}=x\right\}\quad\text{ and }\quad\tilde{\tau}(x)=\inf\left\{t\geq s\ :\ \phi_{\Delta_s}\left(\frac{f(t)-f(s-)}{\Delta_s}\right)=x\right\}.\]
Then, we check that $\dtree(\tau(x),\tau(y))=\dtree(\tilde{\tau}(x),\tilde{\tau}(y))=0$ and
\begin{equation}
\label{loop_on_coded}
\dl(\tau(x),\tau(y))=\delta_s(x_s^{\tau(x)},x_s^{\tau(y)})=\Delta_s\delta(x,y)=\tilde{\delta}_s(x_s^{\tilde{\tau}(x)},x_s^{\tilde{\tau}(y)})=\dlt(\tilde{\tau}(x),\tilde{\tau}(y))
\end{equation}
 for any $x,y\in\mathcal{C}$. Hence, the metric subspaces of $\lop[f]$ and $\lopt[f]$ respectively induced by $\{\tau(x)\ :\ x\in\mathcal{C}\}$ and $\{\tilde{\tau}(x)\ :\ x\in\mathcal{C}\}$ are both isometric to $\Delta_s\cdot \mathcal{C}$, namely a metric circle of length $\Delta_s$. The metric subspaces of $\vern[f]$ and $\vernt[f]$ induced by the same previous subsets of $[0,1]$ are isometric to $2\Delta_s\cdot \mathcal{C}$, namely a metric circle of length $2\Delta_s$. These observations together with the homogeneity, the time-change property, and the branching property are generally enough to accurately draw $\tree[f]$, $\lop[f]$, or $\vern[f]$ from the graph of $f$ when it has simple variations. We give an example with Figure~\ref{example_codings} above. For $\lopt[f]$ and $\vernt[f]$, the lack of homogeneity and the complexity of the shuffle can make them hard to be precisely drawn. Nevertheless, we can understand them as transformed versions of $\lop[f]$ and $\vern[f]$ where the positions of the joint points between loops or trees were shuffled. See Figure~\ref{example_codings} once again.

\section{Unification and relations between trees, looptrees, and vernation trees}
\label{unification}

Recall Definition~\ref{excursion} of excursions and Definition~\ref{type_exc} of continuous and PJG excursions. In this section, we inspect the decomposition of an excursion into a continuous part and a PJG part, and we explain how it relates to the expression of the vernation-tree distance $\dv$ as the combination of the tree distance $\dtree$ and of the looptree distance $\dl$. We eventually obtain a proof of Theorem~\ref{intro_uni}, which formalizes the principle that jumps correspond to loops and continuous growths correspond to branches in the associated vernation tree. This also justifies that the coding of vernation trees is a natural extension of the classic coding (\ref{classical_tree_distance}) of trees and the coding of looptrees introduced by Curien \& Kortchemski~\cite{curien2014}.

Recall the notations $\Delta_s$ and $x_s^t$ from (\ref{delta_x}). Our study heavily relies on the operator $J$ given by (\ref{J_def}) and recalled below, and on an approximating operator $J^\varepsilon$ defined as follows: for all $\varepsilon>0$ and for any excursion $f$, we set
\begin{equation}
\label{Jeps_def}
\forall t\in[-1,1],\quad Jf(t)=\sum_{r\in[0,1]}x_r^t\quad\text{ and }\quad J^\varepsilon f(t)=\sum_{\substack{r\in[0,1]\\ \Delta_r\geq\varepsilon}}x_r^t.
\end{equation}
Let us show that $Jf$ and $J^\varepsilon f$ are well-defined càdlàg functions. For any $t\in[0,1]$, the quantities $Jf(t)$ and $J^\varepsilon f(t)$ are sums of nonnegative terms, so $0\leq J^\varepsilon f(t)\leq Jf(t)$. Then, observe from (\ref{maj_dtreeO}) that $\dtree[f](0,1)=0$, where $\dtree[f]$ is the tree pseudo-distance given by (\ref{dtree_easy}). Note also that for all $r\in[0,1]$, we have $0\leq x_r^1\leq f(1)=0$. Thus, the formula (\ref{dtree_easy}) becomes
\begin{equation}
\label{calcul_Jf}
\forall t\in[-1,1],\quad \dtree[f](t,0)=\dtree[f](t,1)=f(t)-Jf(t).
\end{equation}
By Proposition~\ref{d_continu}, $\dtree[f]$ is a continuous so (\ref{calcul_Jf}) entails that $Jf$ is c\`adl\`ag and well-defined. For any $r\in[0,1]$, Definition~\ref{excursion} ensures that the function $t\in[-1,1]\mapsto x_r^t(f)$ is c\`adl\`ag and non-increasing on $[0,r)$ and $(r,1]$. Since the càdlàg function $f$ has a finite number of jumps higher than $\varepsilon$, $J^\varepsilon f$ is thus càdlàg as a finite sum of càdlàg functions. Furthermore, $J^\varepsilon f$ satisfies the assumption of the following lemma, which is helpful to identify PJG excursions in simple cases.

\begin{lemma}
\label{discrete_PJG}
Let $f$ be an excursion. If there exists a partition $0=r_0<r_1<\ldots<r_{n}<r_{n+1}=1$ such that $f$ is non-increasing on $[r_i,r_{i+1})$ for all $0\leq i\leq n$, then $f$ is PJG.
\end{lemma}

\begin{proof}
Thanks to (\ref{calcul_Jf}), we exactly need to show that $\dtree(t,1)=0$ for all $t\in[0,1]$. We do this by induction on $n$. If $n=0$, then $f$ is non-increasing on $[0,1]$ so $\dtree(t,1)\leq f(t)-x_0^t=f(t)-\inf_{[0,t]}f=0$, and it follows that $f$ is PJG. If $n\geq 1$, we set $\bar{r}=\inf\{r\geq r_n\ :\ f(r)=f(r_n-)\}$ so that $r_n\preceq \bar{r}$ and $f(r_n-)=f(\bar{r})=f(\bar{r}-)$. We can assume $r_n<\bar{r}$ because otherwise, $f(r_n)=f(r_n-)$ which means that $f$ is in fact non-increasing on $[r_{n-1},1]$ and so $f$ is PJG by induction. We apply the branching property given by Proposition~\ref{branching_property} with $u=r_n$ and $v=\bar{r}$ and we keep the same notations. We see that $h$ is non-increasing on  $[0,1]$ and that $g$ is non-increasing on $[r_{n-1},1]$ and on $[r_i,r_{i+1})$ for all $0\leq i\leq n-2$. By induction, $g$ and $h$ are PJG, so $\dtree[g](t,1)=\dtree[h](t,1)=0$ for all $t\in[0,1]$. We conclude thanks to the branching property.
\end{proof}

\begin{theorem}
\label{J_prop}
Let $\varepsilon>0$. If $f$ is an excursion, then the following points hold.
\begin{enumerate}
\item[$(i)$] The functions $Jf$ and $J^\varepsilon f$ are PJG excursions.

\item[$(ii)$] The function $f-Jf$ is a continuous excursion.

\item[$(iii)$] The function $f-Jf+J^\varepsilon f$ is an excursion that verifies $J(f-Jf+J^\varepsilon f)=J^\varepsilon f$.

\item[$(iv)$] We have $Jf=J^\varepsilon f$ if and only if $\Delta_t(f)=\Delta_t(f)\I{\Delta_t(f)\geq \varepsilon}$ for all $t\in[0,1]$.
\end{enumerate}
\end{theorem}

\begin{proof}
We first note that $0=f(s)=f(1)=Jf(s)=Jf(1)=J^\varepsilon f(s)=J^\varepsilon f(1)$ for any $s\in[-1,0)$. We already saw that $Jf$ and $J^\varepsilon f$ are càdlàg and nonnegative. Proposition~\ref{d_continu} and (\ref{calcul_Jf}) yield that $f-Jf$ is continuous and nonnegative. Thus, we get $(ii)$. This also gives us that $\Delta_t(Jf)=\Delta_t(f)\geq 0$ for any $t\in[0,1]$ so $Jf$ is an excursion. Since $J^\varepsilon f$ is a finite sum of càdlàg functions with a unique jump, we readily get $\Delta_t(J^\varepsilon f)=\Delta_t(f)\I{\Delta_t(f)\geq\varepsilon}\geq 0$. Hence, $J^\varepsilon f$ and $f-Jf+J^\varepsilon f$ are excursions. Moreover, if $Jf=J^\varepsilon f$ then we find $\Delta_t(f)=\Delta_t(f)\I{\Delta_t(f)\geq \varepsilon}$ for any $t\in[0,1]$. The converse is immediate with (\ref{Jeps_def}), so we obtain $(iv)$. Recall that $J^\varepsilon f$ satisfies the assumption of Lemma~\ref{discrete_PJG}, so it is PJG.

It only remains to show that $J(Jf)=Jf$ and $J(f-Jf+J^\varepsilon f)=J^\varepsilon f$. In fact, we prove that for all $s,t\in[-1,1]$,
\begin{equation}
\label{x(Jf)}
x_s^t(Jf)=x_s^t(f)\quad\text{ and }\quad x_s^t(f-Jf+J^\varepsilon f)=x_s^t(f)\I{\Delta_s(f)\geq\varepsilon}.
\end{equation}
Observe with (\ref{Jeps_def}) that (\ref{x(Jf)}) indeed implies the desired result. We assume that $0\leq s\leq t\leq 1$, because (\ref{x(Jf)}) is obvious otherwise. Using the expressions for the jumps of $Jf$ and $J^\varepsilon f$ we have found, we obtain that
\[J^\varepsilon f(s-)=\sum_{r\in[0,s)}x_r^s\I{\Delta_r\geq \varepsilon}\quad\text{ and }\quad Jf(s-)-J^\varepsilon f(s-)=\sum_{r\in[0,s)}x_r^s\I{\Delta_r<\varepsilon}.\]
Recall the genealogical order $\preceq_f$ given by (\ref{genealogy_def}) and let $u\in[0,1]$ such that $s\preceq_f u$.  Lemma~\ref{tri_utile} then implies that
\begin{equation}
\label{step1_Jprop}
J^\varepsilon f(s-)\leq J^\varepsilon f(u)-x_s^u(f)\I{\Delta_s(f)\geq\varepsilon}\quad\text{ and }\quad Jf(s-)-J^\varepsilon f(s-)\leq Jf(u)-J^\varepsilon f(u)-x_s^u(f)\I{\Delta_s(f)<\varepsilon}.
\end{equation}
Moreover, the continuity of $\dtree[f]$, (\ref{calcul_Jf}), and Proposition~\ref{geo_tree} entail that $f(u)-Jf(u)-f(s-)+Jf(s-)=\dtree[f](s,u)\geq 0$. Setting $g=f-Jf+J^\varepsilon f$ to lighten the notation, we deduce from (\ref{step1_Jprop}) and from the previous inequality that
\begin{equation}
\label{step2_Jprop}
x_s^u(f)\leq Jf(u)-Jf(s-)\leq f(u)-f(s-)\quad\text{ and }\quad x_s^u(f)\I{\Delta_s(f)\geq\varepsilon}\leq g(u)-g(s-)\leq f(u)-f(s-).
\end{equation}
Note that $x_s^t(g)=x_s^t(g)\I{\Delta_s(g)>0}=x_s^t(g)\I{\Delta_s(f)\geq\varepsilon}$ because $\Delta_s(g)=\Delta_s(J^\varepsilon f)$. Next, observe from (\ref{genealogy_def}) that if $s\preceq_f t$ then $s\preceq_f u$ for any $u\in[s,t]$. In that case, taking the infimum over $u\in[s,t]$ in (\ref{step2_Jprop}) yields that if $s\preceq_f t$ then (\ref{x(Jf)}) holds. To get (\ref{x(Jf)}) in general, we consider $\bar{s}=\inf\{r\geq s\, :\, f(r)=f(s-)\}$, so that $s\preceq_f \bar{s}$ but $x_s^{\bar{s}}(f)=0$. In particular, either $s\preceq_f t$ or $\bar{s}<t$. The identities (\ref{x(Jf)}) under the assumption that $s\preceq_f t$ entail that $x_s^{\bar{s}}(f)=x_s^{\bar{s}}(Jf)=x_s^{\bar{s}}(g)=0$. Since the functions $u\in[s,1]\mapsto x_s^u$ are non-increasing and nonnegative, we obtain (\ref{x(Jf)}) in all its generality.
\end{proof}

\noindent
This result justifies that if $f$ is an excursion, then $f-Jf$ is its continuous part and $Jf$ is its pure jump growth part. Informally, the excursion $f-Jf+J^\varepsilon f$ is obtained by stripping $f$ of all its small jumps. In contrast, the excursion $J^\varepsilon f$ contains exactly the information about the big jumps.

Let us now fix a shuffle $\Phi$ as in Definition~\ref{shuffle}. For all $r\in[0,1]$, recall from (\ref{circle_dist_s}) and (\ref{circle_shuffled_dist_s}) the pseudo-distances $\delta_r^f$ and $\tilde{\delta}_r^f$ that respectively induce the loop and the shuffled loop coded by $f$ and associated with $r$. Then recall the pseudo-distances $\dl[f]$ and $\dlt[f]$ from (\ref{dl_easy}) and (\ref{dlt_easy}). The following theorem completes the proof of Theorem~\ref{intro_uni}. 

\begin{theorem}
\label{d_J}
Let $f$ be an excursion and let $\varepsilon>0$. It holds that $\dtree[f]=\dtree[f-Jf]$, and $\dl[f]=\dl[Jf]$, and $\dlt[f]=\dlt[Jf]$. Moreover,
\begin{equation}
\label{dl_Jeps}
\forall s,t\in[-1,1],\quad \quad \dl[J^\varepsilon \!f](s,t)=\sum_{\substack{r\in[0,1]\\ \Delta_r(f)\geq\varepsilon}}\delta_r^f\left(x_r^s(f),x_r^t(f)\right)\quad\text{ and }\quad \dlt[J^\varepsilon\! f](s,t)=\sum_{\substack{r\in[0,1]\\ \Delta_r(f)\geq\varepsilon}}\tilde{\delta}_r^f\left(x_r^s(f),x_r^t(f)\right).
\end{equation}
\end{theorem}

\noindent
Indeed, Theorems~\ref{J_prop} and \ref{d_J} contain the assertions $(i)$ and $(ii)$ of Theorem~\ref{intro_uni}. Then, the expressions of $\dv[Jf]$ and $\dv[f-Jf]$ claimed by Theorem~\ref{intro_uni} follow from the simple fact that $\dl[0]=\dtree[0]=0$. Let us now prove Theorem~\ref{d_J}.

\begin{proof}
Recall that for any $r\in[0,1]$, $\delta_r$ and $\tilde{\delta}_r$ only depend on $\Delta_r=x_r^r$. 
Plugging the left identity of (\ref{x(Jf)}) into the formulas (\ref{dl_easy}) and (\ref{dlt_easy}) entails that $\dl[Jf]=\dl[f]$ and $\dlt[Jf]=\dlt[f]$. In particular, we get $\dl[J^\varepsilon f]=\dl[f-Jf+J^\varepsilon f]$ and $\dlt[J^\varepsilon f]=\dlt[f-Jf+J^\varepsilon f]$ thanks to Theorem~\ref{J_prop} $(iii)$. Then, combining the right identity of (\ref{x(Jf)}) with (\ref{dl_easy}) and (\ref{dlt_easy}) yields (\ref{dl_Jeps}), since the $\delta_r$ and $\tilde{\delta}_r$ are pseudo-distances. It remains to show that $\dtree[f-Jf]=\dtree[f]$. Let $s,t\in[0,1]$ with $s<t$. The identities (\ref{calcul_Jf}) and (\ref{dtree_easy}) yield
\[\dtree[f-Jf](s,t)=\dtree[f](0,s)+\dtree[f](0,t)-2\inf_{u\in[s,t]}\dtree[f](0,u)\]because $f-Jf$ is a continuous excursion. Recall from Section~\ref{genealogy_sec} that $s\wedge_f t\leq s$ and $s\wedge_f t\preceq_f t$. Thus, if $u\in[s,t]$ then $s\wedge_f t\preceq_f u$ by (\ref{genealogy_def}), and $\dtree[f](0,u)\geq\dtree[f](0,s\wedge_f t)$ by Proposition~\ref{geo_tree}. Moreover, applying (\ref{maj_dtreeO}) and then Lemma~\ref{autre_utile} yields \[\inf_{u\in[s,t]}\dtree[f](0,u)-\dtree[f](0,s\wedge_f t)=\inf_{u\in[s,t]}\dtree[f]\left(s\wedge_f t,u\right)\leq \inf_{u\in[s,t]}\big(f(u)-\inf_{[s\wedge_f t,u]}f\big)\leq \inf_{u\in[s,t]}f(u)-\inf_{[s\wedge_f t,t]}f=0.\]
Hence, we conclude that $\dtree[f-Jf](s,t)=\dtree[f](0,s)+\dtree[f](0,t)-2\dtree[f](0,s\wedge_f t)=\dtree[f](s,t)$ thanks to Proposition~\ref{geo_tree}.
\end{proof}

\begin{corollary}
\label{propo}
For all excursion $f$ and for all $\varepsilon\in(0,1)$, it both holds that $||\dl[f]-\dl[J^\varepsilon f]||_\infty\leq 2||Jf-J^\varepsilon f||_\infty$ and that $||\dlt[f]-\dlt[J^\varepsilon f]||_\infty\leq 2K_\Phi||Jf-J^\varepsilon f||_\infty$,
where $K_\Phi$ is the finite constant, which only depends on the shuffle $\Phi$, given by\[K_\Phi=\sup_{\Delta\in(0,1]}\sup_{x\in(0,1]}\frac{1}{x}\delta\left(\phi_\Delta(0),\phi_\Delta(x)\right).\]
\end{corollary}

\begin{proof}
From Definition~\ref{shuffle} of a shuffle, we know $K_\Phi$ is indeed finite. By definition, we have $\tilde{\delta}_r^f(0,x)\leq K_\Phi x$ for any $r\in[0,1]$ such that $\Delta_r(f)\leq 1$ and for any $x\in[0,\Delta_r(f)]$. Plus, it is clear from (\ref{circle_dist_s}) that $\delta_r^f(0,x)\leq x$. It follows that
\begin{align*}
\sum_{\substack{r\in[0,1]\\ \Delta_r(f)<\varepsilon}}\delta_r^f\left(0,x_r^t(f)\right)&\leq \sum_{\substack{r\in[0,1]\\ \Delta_r(f)<\varepsilon}}x_r^t(f)=Jf(t)-J^\varepsilon f(t),\\
\sum_{\substack{r\in[0,1]\\ \Delta_r(f)<\varepsilon}}\tilde{\delta}_r^f\left(0,x_r^t(f)\right)&\leq \sum_{\substack{r\in[0,1]\\ \Delta_r(f)<\varepsilon}}K_\Phi x_r^t(f)=K_\Phi\left(Jf(t)-J^\varepsilon f(t)\right).
\end{align*}
We conclude by applying the formulas (\ref{dl_easy}), (\ref{dlt_easy}), and (\ref{dl_Jeps}), as well as the triangular inequalities of the $\delta_r^f$ and the $\tilde{\delta}_r^f$.
\end{proof}

We finish this section by providing some tools to study convergences involving $J$ or $J^\varepsilon$.

\begin{proposition}
\label{cv_J}
Let $f$ be an excursion and let $\varepsilon>0$. The following holds.
\begin{enumerate}
\item[$(i)$] If $\lambda:[-1,1]\to[-1,1]$ is an increasing bijection with $\lambda(0)\leq 0$, then $J(f\circ\lambda)=(Jf)\circ\lambda$ and $J^\varepsilon(f\circ\lambda)=(J^\varepsilon f)\circ\lambda$.
\item[$(ii)$] If $\varepsilon\notin\{\Delta_t(f)\, :\, t\in[0,1]\}$, then $J^\varepsilon$ is continuous at $f$ with respect to the topology of uniform convergence.
\item[$(iii)$] The excursion $J^\varepsilon f$ uniformly converges on $[-1,1]$ towards $Jf$ as $\varepsilon\to 0^+$.
\end{enumerate}
\end{proposition}

\begin{proof}
As in the proof of Proposition~\ref{time-change}, we deduce $(i)$ from the equalities $x_s^t(f\circ\lambda)=x_{\lambda(s)}^{\lambda(t)}(f)$ for all $s,t\in[-1,1]$. Let us prove $(ii)$. The excursion $f$ is c\`adl\`ag so the set $\{r\in[0,1]\ :\ \Delta_r(f)>\varepsilon/2\}$ is finite. Thus,
\begin{align*}
\eta_1&=\max\{\Delta_r(f)\ :\ r\in[0,1]\text{ such that }\varepsilon/2<\Delta_r(f)<\varepsilon\}\cup\{\varepsilon/2\},\\
\eta_2&=\min\{\Delta_r(f)\ :\ r\in[0,1]\text{ such that }\Delta_r(f)\geq\varepsilon\}\cup\{2\varepsilon\}
\end{align*}
are well-defined and $\eta_1<\varepsilon\leq \eta_2$. In fact, we have $\varepsilon<\eta_2$ by hypothesis. Let $g$ be an excursion such that $||f-g||_\infty<\min(\varepsilon-\eta_1,\eta_2-\varepsilon)/2$. If $\Delta_r(g)\geq\varepsilon$ then $\Delta_r(f)>\eta_1$ and we obtain $\Delta_r(f)\geq\varepsilon$ by definition of $\eta_1$. Likewise, if $\Delta_r(g)<\varepsilon$ then $\Delta_r(f)<\eta_2$ and we obtain $\Delta_r(f)<\varepsilon$ by definition of $\eta_2$. Hence, we can write\[J^\varepsilon\! g(t)=\sum_{\substack{r\in[0,1]\\ \Delta_r(f)\geq\varepsilon}}x_r^t(g)\] for all $t\in[-1,1]$, so $\left|\left|J^\varepsilon\! f-J^\varepsilon\! g\right|\right|_\infty \leq2N(f)\times ||f-g||_\infty$ where we denote by $N(f)$ the number of jumps of $f$ of height at least $\varepsilon$. The point $(ii)$ follows. To show $(iii)$, we mimic the proof of Dini's theorem. We fix $\nu>0$ and we set
\[U_n=\{t\in[-1,1]\ :\ Jf(t)-J^{\frac{1}{n}}\!f(t)<\nu\text{ and } Jf(t-)-J^{\frac{1}{n}}\!f(t-)<\nu\}\] for all $n\geq 1$. We remark that $U_n\subset U_{n+1}$ because $J^{1/n}f\leq J^{1/(n+1)}f\leq Jf$. The functions $Jf$ and $J^{1/n}f$ are c\`adl\`ag, so $U_n$ is open for all $n\geq 1$. Furthermore, the monotone convergence theorem implies that $J^{1/n}f(t)\underset{n\rightarrow\infty}{\longrightarrow}Jf(t)$ for all $t\in[-1,1]$. Recall (\ref{x(Jf)}) and that $x_t^t(f)=\Delta_t(f)$. When $n$ is large enough (depending on $t$), we have \[\Delta_t(J^{1/n}f)=\Delta_t(f)\I{\Delta_t(f)\geq 1/n}=\Delta_t(f)=\Delta_t(Jf),\] which yields $Jf(t-)-J^{1/n}f(t-)=Jf(t)-J^{1/n}f(t)$. Therefore, the family $(U_n)_{n\geq 1}$ covers $[-1,1]$, so there exists $N\geq 1$ such that $U_N=[-1,1]$ by compactness. As a result, $||Jf-J^{1/n}f||_\infty\leq\nu$ when $n\geq N$.
\end{proof}

\section{Limit theorems}
\label{section_cv}

\subsection{Functional continuity of the shuffled vernation-tree pseudo-distance}
\label{section_cv_tec}

Recall from Definition~\ref{excursion} that the space of excursions $\mathbb{H}$ is a closed subset of the Skorokhod space $\mathbb{D}([-1,1])$, endowed with the Skorokhod distance $\rho$ given by (\ref{sko_dist}). Also recall from Definition~\ref{bivariate} the space $\mathbb{D}([-1,1]^2)$ equipped with the Skorokhod topology on $[-1,1]^2$, which is induced by the distance $\rho_2$ as in (\ref{sko_dist_biv}). Let us fix a shuffle $\Phi$ as in Definition~\ref{shuffle}. We know from Proposition~\ref{dt_continu} that for any $f\in\mathbb{H}$, the shuffled vernation-tree pseudo-distance $\dvt[f]$ given by (\ref{dlt_easy}) is an element of $\mathbb{D}([-1,1]^2)$. Recall the sets $\mathrm{B}(\Phi)$ and $B_f$ that are involved in the assumption of Theorem~\ref{gh-continu}. Here, instead of showing Theorem~\ref{gh-continu}, we first prove the following statement in terms of real functions, which is slightly stronger.

\begin{theorem}
\label{continu}
Let $f$ be an excursion. If $B_f\cap\mathrm{B}(\Phi)=\emptyset$, then the map $g\in\mathbb{H}\longmapsto \dvt[g]\in\mathbb{D}([-1,1]^2)$ is continuous at $f$.
\end{theorem}

\begin{proof}
The main idea of the proof is to use the convergence in (\ref{11}) to forget the small jumps of $f$. Then, the hypothesis $B_f\cap\mathrm{B}(\Phi)=\emptyset$ and the continuity of $\Phi$ allow us to control the variations generated by the big jumps. 

\paragraph{Choice of some parameters.}Let $\varepsilon\in(0,1)$. We are able to choose some $\Delta\in(0,\varepsilon)$ such that
\begin{equation}
\label{13}
\text{if }u\in(0,2\Delta)\text{ then }\sup_{x\in(0,1]}\left|\frac{2}{x}\delta\left(\phi_u(0),\phi_u(x)\right)-1\right|\leq\varepsilon
\end{equation}
thanks to (\ref{11}). The set $\{s\in[0,1]\ :\ \Delta_s\geq\Delta\}$ is finite because $f$ is c\`adl\`ag. We denote by $N$ its size and by $r_1<r_2<\ldots<r_N$ its elements. The functions $\phi_{\Delta_{r_n}}$ are c\`agl\`ad and continuous at $0$. Hence, we have some $\nu>0$ and some partitions of $[0,1]$ denoted by $0<y_{n,1}<y_{n,2}<\ldots<y_{n,k_n}<1$, where the $y_{n,i}$ are discontinuity points of $\phi_{\Delta_{r_n}}$ for any $1\leq i\leq k_n$, such that for all $1\leq n\leq N$,
\begin{align}
\max_{1\leq i\leq k_n}\sup_{\substack{y,y'\in(y_{n,i},y_{n,i+1}]\\|y-y'|\leq\nu}}\delta(\phi_{\Delta_{r_n}}(y),\phi_{\Delta_{r_n}}(y'))&\leq\frac{\varepsilon}{N+1}, \label{14}\\
\sup_{\substack{y,y'\in[0,y_{n,1}]\\|y-y'|\leq\nu}}\delta(\phi_{\Delta_{r_n}}(y),\phi_{\Delta_{r_n}}(y'))&\leq\frac{\varepsilon}{N+1}, \label{15}
\end{align}
with $y_{n,k_n+1}=1$ by convention. We now take a parameter $\eta>0$ that respects the following inequalities:
\begin{equation}
\label{16}
\eta\leq \min\left(\frac{\Delta}{4},\frac{\varepsilon}{N+1}\right)\quad\text{ and }\quad\left(1+\frac{4}{\Delta}+\frac{4}{\Delta^2}||f||_\infty\right)\eta\leq\nu.
\end{equation}
The fact that $f$ is an excursion ensures that for $1\leq n\leq N$ and $1\leq i\leq k_n$, \[t_{n,i}=\inf\{t\geq r_n\ :\ f(t)-f(r_n-)=y_{n,i}\Delta_{r_n}\}\] is well-defined, $f(t_{n,i})=f(t_{n,i}-)$, and $x_{r_n}^{t_{n,i}}=y_{n,i}\Delta_{r_n}>0$. In particular, the $t_{n,i}$, for $1\leq n\leq N$ and $1\leq i\leq k_n$, are continuity points of $f$, so they are distinct from the $r_n$, for $1\leq n\leq N$. Moreover, if $t_{n,i}=t_{m,j}$ with $n<m$, $1\leq i\leq k_n$ and $1\leq j\leq k_m$, then $x_{r_n}^{t_{n,i}}=x_{r_n}^{r_m}$ because $f(r_m-)\leq \inf_{[r_m,t_{m,j}]}f$. Since $t_{n,i}\neq r_m$, we find $(\Delta_{r_n},y_{n,i})\in B_f$. By choice of $y_{n,i}$, we also have $(\Delta_{r_n},y_{n,i})\in\mathrm{B}(\Phi)$, which contradicts the assumption $B_f\cap\mathrm{B}(\Phi)=\emptyset$. If $t_{n,i}=t_{n,j}$, with $1\leq i,j\leq k_n$, then $y_{n,i}\Delta_{r_n}=x_{r_n}^{t_{n,i}}=x_{r_n}^{t_{n,j}}=y_{n,j}\Delta_{r_n}$ so $i=j$. Thus, the $t_{n,i}$ are distinct.

Eventually, we obtain the existence of some $\gamma\in(0,\varepsilon)$ such that the intervals $[t_{n,i}-\gamma,t_{n,i}+\gamma]$, for $1\leq n\leq N$ and $1\leq i\leq k_n$, are included in $[0,1]$, disjoint, do not contain any of the $r_m$, and such that
\begin{equation}
\label{17}
\text{if }t\in[0,1]\text{ with }|t-t_{n,i}|\leq\gamma\quad\text{ then }\quad|f(t)-f(t_{n,i})|\leq\eta/3.
\end{equation}
We set
\[\omega=\min_{\substack{1\leq n\leq N\\ 1\leq i\leq k_n}}\min(x_{r_n}^{t_{n,i}-\gamma}-x_{r_n}^{t_{n,i}},x_{r_n}^{t_{n,i}}-x_{r_n}^{t_{n,i}+\gamma}),\] and we know $\omega>0$ thanks to the assumption $B_f\cap\mathrm{B}(\Phi)=\emptyset$. Then, we choose $\xi\in(0,\eta/3)$ such that
\begin{equation}
\label{19}
\left(\frac{8}{\Delta}+1\right)\xi<\frac{\omega}{2+||f||_\infty}.
\end{equation}
Recall from (\ref{sko_dist_circle}) the Skorokhod distance $\overleftarrow{\rho}$ for $\mathcal{C}$-valued c\`agl\`ad functions. Since $\Phi:u\in(0,\infty)\longmapsto \phi_u$ is continuous at $\Delta_{r_n}$ for $1\leq n\leq N$ by assumption $B_f\cap\mathrm{B}(\Phi)=\emptyset$, we can choose some $\nu'\in(0,\Delta)$ such that for all $1\leq n\leq N$,
\begin{equation}
\label{18}
\sup_{|u-\Delta_{r_n}|\leq\nu'}\overleftarrow{\rho}(\phi_u,\phi_{\Delta_{r_n}})\leq\xi/2.
\end{equation}

\paragraph{Correspondence between big discontinuities.}Let $g\in\mathbb{H}$ such that $||f-g||_\infty\leq\min(\xi,\nu'/2)$. We are going to construct an increasing bijection from $[0,1]$ to itself which makes correspond the discontinuity points of $\dvt[f]$ and $\dvt[g]$. If $s,t\in[0,1]$ then $|x_s^t(f)-x_s^t(g)|\leq 2||f-g||_\infty$. In particular for $s=t$, we have $|\Delta_s(f)-\Delta_s(g)|\leq\nu'$, so (\ref{18}) gives us some increasing bijections from $[0,1]$ to itself, denoted by $\mu_n$ for $1\leq n\leq N$, such that
\begin{equation}
\label{20}
||\mu_n-\mathsf{id}||_\infty\leq\xi\quad\text{ and }\quad\sup_{x\in[0,1]}\delta\left(\phi_{\Delta_{r_n}(g)}(x),\phi_{\Delta_{r_n}(f)}\circ\mu_n(x)\right)\leq\xi.
\end{equation}
Let $\lambda$ be the continuous function on $[0,1]$ which is affine on the intervals generated by the points \[\{t_{n,i},\ t_{n,i}+\gamma,\ t_{n,i}-\gamma\ :\ 1\leq n\leq N\text{ and }1\leq i\leq k_n\}\] such that $\lambda(0)=0,\ \lambda(1)=1,\ \lambda(t_{n,i}-\gamma)=t_{n,i}-\gamma,\ \lambda(t_{n,i}+\gamma)=t_{n,i}+\gamma$, and \[\lambda(t_{n,i})=\inf\left\{t\geq r_n\ :\ \frac{x_{r_n}^t}{\Delta_{r_n}}(g)\leq\mu_n^{-1}(y_{n,i})\right\}.\]It is well-defined because $g$ is an excursion. We extend $\lambda$ on $[-1,1]$ by setting $\lambda(s)=s$ for any $s\in[0,-1)$. Observe that $x_{r_n}^{\lambda(t_{n,i})}(g)=\Delta_{r_n}(g)\times \mu_n^{-1}(y_{n,i})$. Let us prove that $\lambda$ is an increasing function, close to the identity. Let $1\leq n\leq N$ and $1\leq i\leq k_n$. We use $||f-g||_\infty\leq\xi$ then we recall the definitions of $t_{n,i}$ and $\omega$ to find
\[\frac{x_{r_n}^{t_{n,i}-\gamma}}{\Delta_{r_n}}(g)\geq\frac{x_{r_n}^{t_{n,i}-\gamma}(f)-2\xi}{\Delta_{r_n}(g)}\geq\frac{y_{n,i}\Delta_{r_n}(f)+\omega-2\xi}{\Delta_{r_n}(g)}.\]
Plus, we found $\Delta_{r_n}(g)\geq \Delta_{r_n}(f)-2\xi\geq \Delta-2\eta$, so $\Delta_{r_n}(g)\geq\Delta/2$ according to (\ref{16}). We also have $\Delta_{r_n}(g)\leq \Delta_{r_n}(f)+2\xi\leq 2+||f||_\infty$. It follows that
\[\frac{x_{r_n}^{t_{n,i}-\gamma}}{\Delta_{r_n}}(g)\geq y_{n,i}+\frac{\omega}{2+||f||_\infty}-\frac{8\xi}{\Delta},\]
then successive applications of (\ref{20}) and (\ref{19}) lead to
\[\frac{x_{r_n}^{t_{n,i}-\gamma}}{\Delta_{r_n}}(g)\geq \mu_n^{-1}(y_{n,i})+\frac{\omega}{2+||f||_\infty}-\left(\frac{8}{\Delta}+1\right)\xi>\mu_n^{-1}(y_{n,i}).\]
We prove $x_{r_n}^{t_{n,i}+\gamma}(g)/\Delta_{r_n}(g)<\mu_n^{-1}(y_{n,i})$ in the same way. Thus, $|\lambda(t_{n,i})-t_{n,i}|<\gamma$. We deduce that $\lambda$ is strictly increasing and 
\begin{equation}
\label{lambda_near_id}
||\lambda-\mathsf{id}||_\infty\leq\gamma\leq\varepsilon.
\end{equation}
Moreover, $\lambda$ is a bijection from $[-1,1]$ to itself because $\lambda(0)=0$, $\lambda(1)=1$, and $\lambda$ is continuous.

\paragraph{Control on the big loops.}We easily see that if $t$ is outside all the intervals $[t_{n,i}-\gamma,t_{n,i}+\gamma]$ then $\lambda(t)=t$. But if $t\in[t_{n,i}-\gamma,t_{n,i}+\gamma]$ then $\lambda(t)\in[t_{n,i}-\gamma,t_{n,i}+\gamma]$ too. Therefore, (\ref{17}) allows us to deduce \[||f-g\circ\lambda||_\infty\leq||f-f\circ\lambda||_\infty+||f-g||_\infty\leq2\eta/3+\eta/3=\eta.\]
Now, we set $h=g\circ\lambda$ so that $||f-h||_\infty\leq\eta$. Our goal is to show that $||\dvt[f]-\dvt[h]||_\infty$ is small, but before that, we are going to bound the difference of the distances on the big loops. We begin by recalling that the point $r_n$ is outside all the intervals $[t_{n,i}-\gamma,t_{n,i}+\gamma]$ so $\lambda(r_n)=r_n$, which implies $\Delta_{r_n}(h)=\Delta_{r_n}(g)$ and $x_{r_n}^t(h)=x_{r_n}^{\lambda(t)}(g)$. Let $t\geq r_n$ and $1\leq i\leq k_n$. Recall $(\Delta_{r_n}(f),y_{n,i})\notin B_f$ because $(\Delta_{r_n}(f),y_{n,i})\in\mathrm{B}(\Phi)$, so \[\frac{x_{r_n}^t}{\Delta_{r_n}}(f)\leq y_{n,i}\Longleftrightarrow t_{n,i}\leq t.\]Then, $\lambda(r_n)=r_n\leq \lambda(t)$ because $\lambda$ is increasing. By definition of $\lambda(t_{n,i})$ together with the observation that $x_{r_n}^{\lambda(t_{n,i})}(g)=\Delta_{r_n}(g)\times \mu_n^{-1}(y_{n,i})$, we have $x_{r_n}^{\lambda(t)}(g)/\Delta_{r_n}(g)\leq\mu_n^{-1}(y_{n,i})\Longleftrightarrow\lambda(t_{n,i})\leq\lambda(t)$. The functions $\lambda$ and $\mu_n$ are strictly increasing so it follows that
\[\frac{x_{r_n}^t}{\Delta_{r_n}}(f)\leq y_{n,i}\Longleftrightarrow \mu_n\left(\frac{x_{r_n}^{t}}{\Delta_{r_n}}(h)\right)\leq y_{n,i}.\]
Hence, the points $x_{r_n}^t(f)/\Delta_{r_n}(f)$ and $\mu_n\left(x_{r_n}^{t}(h)/\Delta_{r_n}(h)\right)$ are either both in the interval $[0,y_{n,1}]$ or both in the same interval $(y_{n,i},y_{n,i+1}]$ with $1\leq i\leq k_n$. Furthermore, we can bound
\[\left|\frac{x_{r_n}^t}{\Delta_{r_n}}(f)-\mu_n\left(\frac{x_{r_n}^{t}}{\Delta_{r_n}}(h)\right)\right|\leq||\mu_n-\mathsf{id}||_\infty+\frac{2||f-h||_\infty}{\Delta_{r_n}(g)}+|x_{r_n}^t(f)|\frac{2||f-h||_\infty}{\Delta_{r_n}(f)\Delta_{r_n}(g)},\]so that the inequalities $\Delta_{r_n}(g)\geq\Delta/2$, (\ref{20}), and (\ref{16}) give
\[\left|\frac{x_{r_n}^t}{\Delta_{r_n}}(f)-\mu_n\left(\frac{x_{r_n}^{t}}{\Delta_{r_n}}(h)\right)\right|\leq\left(1+\frac{4}{\Delta}+\frac{4}{\Delta^2}||f||_\infty\right)\eta\leq \nu.\]
Therefore, it follows from (\ref{14}) and (\ref{15}) that
\begin{equation}
\label{21}
\delta\left(\phi_{\Delta_{r_n}(f)}\left(\frac{x_{r_n}^t}{\Delta_{r_n}}(f)\right),\phi_{\Delta_{r_n}(f)}\left(\mu_n\left(\frac{x_{r_n}^{t}}{\Delta_{r_n}}(h)\right)\right)\right)\leq\frac{\varepsilon}{N+1}
\end{equation}
for all $t\in[0,1]$, the case $t<r_n$ being obvious. Now, let $s,t\in[0,1]$, the triangular inequality on $\delta$ implies
\begin{multline*}
\left|\tilde{\delta}_{r_n}^f\left(x_{r_n}^s(f),x_{r_n}^t(f)\right)-\tilde{\delta}_{r_n}^h\left(x_{r_n}^s(h),x_{r_n}^t(h)\right)\right|\leq |\Delta_{r_n}(f)-\Delta_{r_n}(h)|\\
+\Delta_{r_n}(f)\times\delta\left(\phi_{\Delta_{r_n}(f)}\left(\frac{x_{r_n}^t}{\Delta_{r_n}}(f)\right),\phi_{\Delta_{r_n}(h)}\left(\frac{x_{r_n}^{t}}{\Delta_{r_n}}(h)\right)\right)\\
+\Delta_{r_n}(f)\times\delta\left(\phi_{\Delta_{r_n}(f)}\left(\frac{x_{r_n}^s}{\Delta_{r_n}}(f)\right),\phi_{\Delta_{r_n}(h)}\left(\frac{x_{r_n}^{s}}{\Delta_{r_n}}(h)\right)\right).
\end{multline*}
Recall that $\Delta_{r_n}(h)=\Delta_{r_n}(g)$, $||f-h||_\infty\leq \eta$, and $\xi\leq \eta$. We apply the inequalities (\ref{16}), (\ref{20}), and (\ref{21}) into the previous bound to conclude that for all $s,t\in[0,1]$,
\begin{equation}
\label{22}
\left|\tilde{\delta}_{r_n}^f\left(x_{r_n}^s(f),x_{r_n}^t(f)\right)-\tilde{\delta}_{r_n}^h\left(x_{r_n}^s(h),x_{r_n}^t(h)\right)\right|\leq(2+4||f||_\infty)\frac{\varepsilon}{N+1}.
\end{equation}
\paragraph{Uniform control.}We are now ready to bound $||\dvt[f]-\dvt[h]||_\infty$. Note that if $s\in[-1,0)$ then $\dvt(t,s)=\dvt(t,1)$ for any $t\in[-1,1]$. Thus, we focus on the case where $0\leq s\leq t\leq 1$. Using both (\ref{dtree_easy}) and (\ref{dlt_easy}) for the big jumps, and (\ref{dtree_hard}) and (\ref{dlt_hard}) for the small jumps, we verify that we can write
\begin{align*}
\left|\dvt[f](s,t)-\dvt[h](s,t)\right|\leq &|f(s)-h(s)|+|f(t)-h(t)|+2\left|\inf_{[s,t]}f-\inf_{[s,t]}h\right|\\
&+\sum_{n=1}^N\left|x_{r_n}^s(f)-x_{r_n}^s(h)\right|+\sum_{n=1}^N\left|x_{r_n}^t(f)-x_{r_n}^t(h)\right|\\
&+2\sum_{n=1}^N\left|\tilde{\delta}_{r_n}^f\left(x_{r_n}^s(f),x_{r_n}^t(f)\right)-\tilde{\delta}_{r_n}^h\left(x_{r_n}^s(h),x_{r_n}^t(h)\right)\right|\\
&+3\Delta_{s\wedge_f t}(f)\I{\Delta_{s\wedge_f t}(f)<\Delta}+3\Delta_{s\wedge_h t}(h)\I{\Delta_{s\wedge_h t}(f)<\Delta}\\
&+\sum_{\substack{r\in[0,1]\\ \Delta_r(f)<\Delta}}\left|2\tilde{\delta}_r^f\left(0,x_r^s(f)\right)-x_r^s(f)\right|+\sum_{\substack{r\in[0,1]\\ \Delta_r(f)<\Delta}}\left|2\tilde{\delta}_r^f\left(0,x_r^t(f)\right)-x_r^t(f)\right|\\
&+\sum_{\substack{r\in[0,1]\\ \Delta_r(f)<\Delta}}\left|2\tilde{\delta}_r^h\left(0,x_r^s(h)\right)-x_r^s(h)\right|+\sum_{\substack{r\in[0,1]\\ \Delta_r(f)<\Delta}}\left|2\tilde{\delta}_r^h\left(0,x_r^t(h)\right)-x_r^t(h)\right|.
\end{align*}
The sum of all the terms in the first two rows of the right-hand side can be easily bounded by $8\varepsilon$ because $||f-h||_\infty\leq\eta\leq\varepsilon/(N+1)$, thanks to (\ref{16}). The term of the third row is smaller than $(4+8||f||_\infty)\varepsilon$ by a direct application of (\ref{22}). Next, if $\Delta_r(f)<\Delta$ then $\Delta_r(h)<\Delta+2\eta< 2\Delta$, since the inequality (\ref{16}) justifies $\eta\leq \Delta/4$. As we chose $\Delta$ so that $\Delta<\varepsilon$, the sum of the terms in the fourth row is bounded by $9\varepsilon$. Moreover, the fact (\ref{13}) ensures that if $\Delta_r(f)<\Delta$, then \[\left|2\tilde{\delta}_r^f(0,x)-x\right|\leq \varepsilon x\quad\text{ and }\quad\left|2\tilde{\delta}_r^h(0,y)-y\right|\leq \varepsilon y\] for all $x\in[0,\Delta_r(f)]$ and all $y\in[0,\Delta_r(h)]$. This allows controlling the last four terms of the right-hand side. For instance,
\begin{align*}
\sum_{\substack{r\in[0,1]\\ \Delta_r(f)<\Delta}}\left|2\tilde{\delta}_r^f\left(0,x_r^s(f)\right)-x_r^s(f)\right|&\leq \varepsilon\sum_{\substack{r\in[0,1]\\ \Delta_r(f)<\Delta}}x_r^s(f)\leq\varepsilon Jf(s)\leq ||f||_\infty\varepsilon,\\
\sum_{\substack{r\in[0,1]\\ \Delta_r(f)<\Delta}}\left|2\tilde{\delta}_r^h\left(0,x_r^s(h)\right)-x_r^s(h)\right|&\leq \varepsilon\sum_{\substack{r\in[0,1]\\ \Delta_r(f)<\Delta}}x_r^s(h)\leq\varepsilon Jh(s)\leq\left(1+||f||_\infty\right)\varepsilon,
\end{align*}
because we know from Theorem~\ref{intro_uni} that $f-Jf\geq 0$. Eventually, we find\[\left|\dvt[f](s,t)-\dvt[h](s,t)\right|\leq \left(23+12||f||_\infty\right)\varepsilon.\]
\paragraph{Conclusion.}To sum up, we have shown that for all $\varepsilon>0$, there exists $\eta>0$ such that if $||f-g||_\infty\leq\eta$ then there exists $\lambda$ an increasing bijection from $[-1,1]$ to itself with $||\dvt[f]-\dvt[g\circ\lambda]||_\infty\leq \varepsilon$ and $||\lambda-\mathsf{id}||_\infty\leq\varepsilon$. Indeed, recall (\ref{lambda_near_id}). Now, if the Skorokhod distance between $f$ and $g$ is smaller than $\min(\eta,\varepsilon)/2$, then we have $\mu$ an increasing bijection from $[-1,1]$ to itself such that $||\mu-\mathsf{id}||_\infty\leq\varepsilon$ and $||f-g\circ\mu||_\infty\leq\eta$. Hence, there exists $\lambda$ an increasing bijection from $[-1,1]$ to itself such that $||\mu\circ\lambda-\mathsf{id}||_\infty\leq2\varepsilon$ and $||\dvt[f]-\dvt[g\circ\mu\circ\lambda]||_\infty\leq\varepsilon$. Eventually, thanks to Proposition~\ref{time-change}, $\dvt[g\circ\mu\circ\lambda]=\dvt[g]\circ(\mu\circ\lambda,\mu\circ\lambda)$. Thereby, if $\rho(f,g)\leq\min(\eta,\varepsilon)/2$ then $\rho_2\left(\dvt[f],\dvt[g]\right)\leq 2\varepsilon$, where we remind that $\rho$ is the Skorokhod distance on $\mathbb{H}$ given by (\ref{sko_dist}) and $\rho_2$ is the Skorokhod distance for bivariate functions given by (\ref{sko_dist_biv}).
\end{proof}

\subsection{Two particular cases for convergences of unshuffled pseudo-distances}

Here, we deduce from Theorem~\ref{continu} a limit theorem for the codings $\dl,\dtree,\dv$ respectively given by (\ref{dl_easy}), (\ref{dtree_easy}), and (\ref{dv_easy}). This result in terms of real functions is slightly stronger than Theorem~\ref{gh-continu_particular}, but requires the same specific assumptions.

\begin{corollary}
\label{continu_particular}
Let $f$ be an excursion and let $(f_n)$ be a sequence of excursions such that $f_n\longrightarrow f$ for the Skorokhod topology on $[-1,1]$. We assume that at least one of the following assumptions holds true.
\begin{enumerate}
\item[$(a)$] There is $N\geq 1$ such that all the $f_n$ have at most $N$ jumps.
\item[$(b)$] The excursion $f$ is PJG in the sense of Definition~\ref{type_exc}.
\end{enumerate}
Then, the three convergences $\dl[f_n]\longrightarrow\dl[f]$, $\dtree[f_n]\longrightarrow\dtree[f]$, and $\dv[f_n]\longrightarrow\dv[f]$ hold uniformly on $[-1,1]^2$.
\end{corollary}

\begin{proof}[Proof under assumption $(a)$]
The limiting excursion $f$ also has at most $N$ jumps. Let $\varepsilon\in(0,1)$ such that all the jumps of $f$ are higher than $2\varepsilon$. Hence, we get $J^{\varepsilon}\!f_n\longrightarrow J^{\varepsilon}\!f=Jf$ for the Skorokhod topology thanks to Proposition~\ref{cv_J} and to Theorem~\ref{J_prop} $(iv)$. Let $s\in[-1,1]$, we know that there is $t\in[-1,1]$ such that $|\Delta_t(f)-\Delta_s(f_n)|\leq 2\rho(f_n,f)$, where $\rho$ is the Skorokhod distance given by (\ref{sko_dist}). If $\rho(f_n,f)<\varepsilon/2$ and $\Delta_s(f_n)<\varepsilon$, then $\Delta_t(f)<2\varepsilon$, so that $\Delta_t(f)=0$ by assumption. So, we get $\Delta_s(f_n)\leq 2\rho(f_n,f)$ in that case. Hence, when $n$ is large enough, we find
\[||Jf_n-J^\varepsilon \! f_n||_\infty\leq\sum_{s\in[0,1]}\Delta_s(f_n)\I{0<\Delta_s(f_n)<\varepsilon}\leq 2N\rho(f_n,f)\underset{n\rightarrow\infty}{\longrightarrow}0.\] In particular, the sequence of excursions $g_n:=f_n-Jf_n+J^\varepsilon f_n$ converges to $f$ for the Skorokhod topology (recall Theorem~\ref{J_prop} $(iii)$). Let us fix $\Phi:\Delta>0\longmapsto \phi_\Delta$ such that $\phi_\Delta(x)=x$ for all $\Delta\geq\varepsilon$ and $\phi_{1-\Delta}(x)$ is expressed as in (\ref{remark_shuffle}) for all $1-\Delta\in(0,\varepsilon)$, for all $x\in[0,1]$. It is straightforward to check that $\Phi$ is a shuffle as in Definition~\ref{shuffle}. It is clear that $B_f\cap\mathrm{B}(\Phi)=\emptyset$ and $B_{Jf}\cap\mathrm{B}(\Phi)=\emptyset$, so Theorem~\ref{continu} yields that $\dvt[J^\varepsilon f_n]\longrightarrow\dvt[Jf]$ and $\dvt[g_n]\longrightarrow\dvt[f]$ for the Skorokhod topology on $[-1,1]^2$. Then, for all $n\geq1$, we see with (\ref{dl_easy}) and (\ref{dlt_easy}) that $\dlt[J^\varepsilon f_n]=\dl[J^\varepsilon f_n]$ and $\dlt[f]=\dl[f]$. Indeed, it holds that $\Delta_t(J^\varepsilon f_n)$ and $\Delta_t(f)$ are not in $(0,\varepsilon)$ for all $t\in[0,1]$. By Theorems~\ref{J_prop} $(i)$, the excursions $J^\varepsilon f_n$ and $Jf$ are PJG, so we deduce from Theorem~\ref{d_J} that $\dl[J^\varepsilon f_n]$ converges to $\dl[f]$. Since we know from Proposition~\ref{d_continu} that $\dl[f]$ is continuous on $[-1,1]^2$, this convergence also happens uniformly on $[-1,1]^2$: see Remark~\ref{bivariate_unif}. Similarly, Theorems~\ref{J_prop} and \ref{d_J} ensure that $\dlt[g_n]=\dlt[Jg_n]=\dlt[J^\varepsilon f_n]=\dl[J^\varepsilon f_n]$ and $\dtree[g_n]=\dtree[g_n-Jg_n]=\dtree[f_n-Jf_n]=\dtree[f_n]$. Hence, $\dtree[f_n]+2\dl[J^\varepsilon f_n]\longrightarrow \dtree[f]+2\dl[f]$ for the Skorokhod topology on $[-1,1]^2$, which implies that $\dtree[f_n]$ uniformly converges to $\dtree[f]$ on $[-1,1]^2$, because $\dtree[f]$ is continuous too. To conclude, recall that $||Jf_n-J^\varepsilon f_n||_\infty\underset{n\rightarrow\infty}{\longrightarrow}0$. We deduce the uniform convergence of $\dl[f_n]$ to $\dl[f]$ on $[-1,1]^2$ by applying Corollary~\ref{propo}.
\end{proof}

\begin{proof}[Proof under assumption $(b)$] We know that $\dtree[f]=0$ thanks to Theorem~\ref{intro_uni}.
Let $\varepsilon>0$ such that $\varepsilon\notin\{\Delta_t(f)\ :\ t\in[0,1]\}$. This condition ensures that $J^{\varepsilon} f_n\longrightarrow J^\varepsilon f$ for the Skorokhod topology on $[-1,1]$ by Proposition~\ref{cv_J}, and even more, it holds that $||f_n-J^\varepsilon f_n||_\infty\longrightarrow||f-J^\varepsilon f||_\infty$. The triangular inequality of $\dtree[f_n]$ and the formula (\ref{calcul_Jf}) lead to $||\dtree[f_n]||_\infty\leq 2||f_n-J^\varepsilon f_n||_\infty$, because $J^\varepsilon f_n\leq Jf_n\leq f_n$. Moreover, Corollary~\ref{propo} implies that
\[||\dl[f_n]-\dl[f]||_\infty\leq ||\dl[J^\varepsilon f_n]-\dl[J^\varepsilon f]||_\infty+2||f_n-J^\varepsilon f_n||_\infty+2||f-J^\varepsilon f||_\infty.\]
All the jumps of the $J^\varepsilon f_n$ are higher or equal to $\varepsilon$ and $J^\varepsilon f_n$ converges to $J^\varepsilon f$, thus there is $N\geq 1$ such that all the $J^\varepsilon f_n$ have at most $N$ jumps. By the previous case, $||\dl[J^\varepsilon f_n]-\dl[J^\varepsilon f]||_\infty\longrightarrow0$, so we get
\[\limsup||\dtree[f_n]||_\infty+\limsup||\dl[f_n]-\dl[f]||_\infty\leq 6||f-J^\varepsilon f||_\infty.\]
Since the excursion $f$ is PJG, we have $||f-J^\varepsilon f||_\infty\underset{\varepsilon\rightarrow0}{\longrightarrow}0$ thanks to Proposition~\ref{cv_J} $(iii)$. The set $\{\Delta_t(f)\, :\, t\in[0,1]\}$ is countable, so we can make $\varepsilon\rightarrow 0$ within the above inequality.
\end{proof}

\subsection{From convergences of pseudo-distances to convergences of quotient metric spaces}
\label{def_GH}

Recall the notions presented at the start of Section~\ref{fst_properties_sec}. To see how Theorem~\ref{continu} and Corollary~\ref{continu_particular} yield convergences of pointed weighted metric spaces under the form of Theorems~\ref{gh-continu_particular} and \ref{gh-continu}, we have to present the definition of the pointed Gromov--Hausdorff--Prokhorov distance, which truly induces a topology on the set of the GHP-isometry classes of pointed weighted compact metric spaces. Recall that when no confusion is possible, we simply denote a metric space (possibly endowed with some additional structure) by its underlying set.

A \emph{correspondence} between two compact metric spaces $(X_1,d_1)$ and $(X_2,d_2)$ is a subset $\mathcal{R}$ of $X_1\times X_2$ such that for all $x_1\in X_1$ and $y_2\in X_2$, there are $x_2\in X_2$ and $y_1\in X_1$ such that $(x_1,x_2),(y_1,y_2)\in \mathcal{R}$. We say that $\mathcal{R}$ is \emph{compact} when it is a compact subset of the product space $X_1\times X_2$. The \emph{distortion} of a correspondence $\mathcal{R}$ is defined by \[\mathrm{dis}(\mathcal{R})=\sup\left\{\left|d_1(x_1,y_1)-d_2(x_2,y_2)\right|\ :\ (x_1,x_2),(y_1,y_2)\in\mathcal{R}\right\}.\]The \emph{Gromov--Hausdorff (GH) distance} between $X_1$ and $X_2$ is then expressed by the formula \[\mathtt{d}_{\mathrm{GH}}(X_1,X_2)=\frac{1}{2}\inf_{\mathcal{R}}\mathrm{dis}(\mathcal{R}),\]
where the infimum is taken over all correspondences $\mathcal{R}$ between $X_1$ and $X_2$. Similarly, the \emph{pointed Gromov--Hausdorff (GH) distance} $\mathtt{d}_{\mathrm{GH}}^\bullet(X_1,X_2)$ between two pointed compact metric spaces $(X_1,d_1,a_1)$ and $(X_2,d_2,a_2)$ is expressed by the same formula up to the difference that the infimum is instead taken over all correspondences $\mathcal{R}$ between $X_1$ and $X_2$ such that $(a_1,a_2)\in\mathcal{R}$. For both distances, we may restrict the infimum to compact correspondences without modifying the value. Indeed, a correspondence and its closure have the same distortion, and the latter is a compact correspondence because $X_1\times X_2$ is compact. Let $(X_1,d_1,a_1,\mu_1)$ and $(X_2,d_2,a_2,\mu_2)$ be two pointed weighted compact metric spaces. A \emph{coupling} between $\mu_1$ and $\mu_2$ is a Borel probability measure on $X_1\times X_2$ whose marginals on $X_1$ and $X_2$ are $\mu_1$ and $\mu_2$. The \emph{pointed Gromov--Hausdorff--Prokhorov (GHP) distance} between $X_1$ and $X_2$ is then defined by the formula
\[\mathtt{d}_{\mathrm{GHP}}^\bullet(X_1,X_2)=\inf_{\mathcal{R},\nu}\max\left(\frac{1}{2}\mathrm{dis}(\mathcal{R}),1-\nu(\mathcal{R})\right)\]
where the infimum is taken over all couplings $\nu$ between $\mu_1$ and $\mu_2$ and all compact correspondences $\mathcal{R}$ between $X_1$ and $X_2$ such that $(a_1,a_2)\in\mathcal{R}$. The functions $\mathtt{d}_{\mathrm{GH}}$, $\mathtt{d}_{\mathrm{GH}}^\bullet$, and $\mathtt{d}_{\mathrm{GHP}}^\bullet$ are only pseudo-distances but $\mathtt{d}_{\mathrm{GH}}(X_1,X_2)=0$ if and only if $X_1$ and $X_2$ are isometric, $\mathtt{d}_{\mathrm{GH}}^\bullet(X_1,X_2)=0$ if and only if $X_1$ and $X_2$ are pointed-isometric, and $\mathtt{d}_{\mathrm{GHP}}^\bullet(X_1,X_2)=0$ if and only if $X_1$ and $X_2$ are GHP-isometric. Hence, $\mathtt{d}_{\mathrm{GH}},\mathtt{d}_{\mathrm{GH}}^\bullet,\mathtt{d}_{\mathrm{GHP}}^\bullet$ respectively define a genuine distance on the space $\mathbb{K}$ of isometry classes of compact metric spaces, on the space $\mathbb{K}^\bullet$ of pointed-isometry classes of pointed compact metric spaces, and on the space $\mathbb{K}_{\mathrm{w}}^\bullet$ of GHP-isometry classes of pointed weighted compact metric spaces. We call their respective topologies the Gromov--Hausdorff (GH), the pointed Gromov--Hausdorff (GH), and the pointed Gromov--Hausdorff--Prokhorov (GHP) topologies. The metric spaces $\mathbb{K}$, $\mathbb{K}^\bullet$, and $\mathbb{K}_{\mathrm{w}}^\bullet$ are separable and complete. See Khezeli~\cite{GHP_correspondence} and Abraham, Delmas \& Hoscheit~\cite{GHP_polish} for proof of the stated facts and for more information.

Now, let $d$ and $d_n,n\geq 1$ be pseudo-distances on $[-1,1]$. For any $n\geq 1$, we denote by $(X,d)$ and $(X_n,d_n)$ the quotient metric spaces of $[0,1]$ respectively induced by $d$ and $d_n$, as defined at the start of Section~\ref{space}. We also denote by $\mathsf{p}:[0,1]\to X$ and $\mathsf{p}_n:[0,1]\to X_n$ the canonical projections. We make $X$ and $X_n$ into pointed weighted metric spaces by endowing them with $\mathsf{p}(1)$ and $\mathsf{p}_n(1)$ as their root, and with the push-forward measures $\mu$ and $\mu_n$ of the Lebesgue measure on $[0,1]$ by $\mathsf{p}$ and $\mathsf{p}_n$ as their Borel probability measures. Then, recall from Definition~\ref{bivariate} the space $\mathbb{D}([-1,1]^2)$ equipped with the distance $\rho_2$, given by (\ref{sko_dist_biv}), that induces the Skorokhod topology on $[-1,1]^2$.

\begin{lemma}
\label{GHP_lemma}
We keep the above notation. We assume that $d_n,d\in\mathbb{D}([-1,1]^2)$ and that $X,X_n$ are compact for all $n\geq 1$. We also suppose that $d(s,1)=d_n(s,1)=0$ for all $s\in[-1,0)$ and $n\geq 1$. Then, it holds
\begin{align}
\label{GH_lips}
\mathtt{d}_{\mathrm{GH}}^\bullet(X_n,X)&\leq\rho_2(d_n,d),\\
\label{GHP_lips}
\mathtt{d}_{\mathrm{GHP}}^\bullet(X_n,X)&\leq||d_n-d||_\infty.
\end{align}
Furthermore, if $d_n\longrightarrow d$ for the Skorokhod topology on $[-1,1]^2$, then $X_n\longrightarrow X$ for the pointed GHP topology.
\end{lemma}

\begin{proof}
We set $\mathsf{p}(s)=\mathsf{p}(1)$ and $\mathsf{p}_n(s)=\mathsf{p}_n(1)$ for all $s\in[-1,0)$. Let $\lambda_n,\mu_n:[-1,1]\to [-1,1]$ be two increasing bijections. We have $||d_n-d\circ(\lambda_n,\lambda_n)||_{\infty}\leq2||d_n-d\circ(\lambda_n,\mu_n)||_{\infty}$ by triangular inequality. We observe that the quantity $||d_n-d\circ(\lambda_n,\lambda_n)||_{\infty}$ is the distortion of the correspondence $\mathcal{R}_n=\{(\mathsf{p}_n(s),\mathsf{p}(\lambda_n(s)))\, :\, s\in [-1,1]\}$. This proves (\ref{GH_lips}) by taking the infimum over $\lambda_n,\mu_n$. For $\lambda_n=\mathsf{id}$, $\mathcal{R}_n$ has total mass for the coupling $\nu_n$ between $\mu_n$ and $\mu$ defined by \[\int_{X_n\times X}g\, \dd\nu_n=\int_0^1 g(\mathsf{p}_n(s),\mathsf{p}(s))\, \dd s\]for any bounded measurable $g$. This shows (\ref{GHP_lips}). If $d_n\longrightarrow d$ for the Skorokhod topology on $[-1,1]^2$ then we can assume that $||\lambda_n-\mathsf{id}||_\infty+||d_n-d\circ(\lambda_n,\lambda_n)||_{\infty}\longrightarrow 0$. For all $\varepsilon>0$, we define another compact correspondence between $X_n$ and $X$ by setting $\mathcal{R}_n^\varepsilon=\{(\mathsf{p}_n(s),\mathsf{p}(x))\ :\ s,x\in [-1,1]\text{ such that }d(\lambda_n(s),x)\leq\varepsilon\}$. We have $\limsup \mathrm{dis}(\mathcal{R}_n^\varepsilon)\leq 2\varepsilon$. It is not hard to check that $\mathsf{p}:[0,1]\to X$ is c\`adl\`ag by using the compactness of $X$ and that $d\in\mathbb{D}([-1,1]^2)$. Hence, $\mathsf{p}$ is continuous almost everywhere and $\mathsf{p}(\lambda_n(s))\longrightarrow \mathsf{p}(s)$ almost everywhere. Thanks to the dominated convergence theorem, it follows that $\nu_n(\mathcal{R}_n^\varepsilon)\longrightarrow 1$, which ends the proof.
\end{proof}

\begin{proof}[Proof of Theorems~\ref{gh-continu_particular} and \ref{gh-continu}]
Let $f$ be an excursion and let $s\in[-1,0)$. It is clear from the formulas (\ref{dl_easy}), (\ref{dtree_easy}), (\ref{dv_easy}), and (\ref{dlt_easy}) that $\dl[f](s,1)=\dtree[f](s,1)=\dv[f](s,1)=\dvt[f](s,1)=0$. Then, Propositions~\ref{d_continu} and \ref{dt_continu} ensure that the pseudo-distances $\dl[f],\dtree[f],\dv[f]$ and $\dvt[f]$ satisfy the assumptions in Lemma~\ref{GHP_lemma}. Thanks to this lemma, Theorem~\ref{gh-continu_particular} follows from Corollary~\ref{continu_particular}, and Theorem~\ref{gh-continu} follows from Theorem~\ref{continu}.
\end{proof}

\section{Topological space of vernation trees}
\label{topo}

In this section, we study the topological notion of vernation trees that we have introduced as Definition~\ref{vernation_topo_def}. We eventually relate it with our codings by excursions by proving Theorem~\ref{topo_conclusion}. First, we fix our framework and our notation. Let $(X,d)$ be a metric space and let $x,y\in X$. We say that a map $\gamma:[0,1]\to X$ is a \emph{path on $X$ from $x$ to $y$} when $\gamma$ is continuous with $\gamma(0)=x$ and $\gamma(1)=y$. In that case, we will denote by $\Ima \gamma=\gamma([0,1])$ its image. Observe that $\Ima \gamma$ is compact and closed as a continuous image of the compact space $[0,1]$. If a path $\gamma$ is also injective, we say that $\gamma$ is an \emph{arc on $X$ from $x$ to $y$}. By compactness of $[0,1]$, an arc provides a topological embedding of the segment $[0,1]$ into $X$. Moreover, we say that $\gamma$ is a \emph{geodesic on $X$ from $x$ to $y$} when $d(\gamma(t),\gamma(s))=|t-s|d(x,y)$ for all $t,s\in[0,1]$. A geodesic from $x$ to $y$ is an arc if and only if $x\neq y$. Finally, we say that $\gamma$ is a \emph{loop on $X$ based at $x$} when $\gamma:[0,1]\longrightarrow X$ is continuous, injective on $[0,1)$, and such that $\gamma(0)=\gamma(1)=x$. A loop provides a topological embedding of the circle $\mathcal{C}$ into $X$.

In addition to vernation trees, we will also consider the following types of metric spaces.

\begin{definition}
\label{geo_def}
A \emph{geodesic space} is a metric space $(X,d)$ such that for any $x,y\in X$, there is a geodesic on $X$ from $x$ to $y$. A \emph{real tree} is a geodesic space $(X,d)$ such that for any $x,y\in X$, all arcs on $X$ from $x$ to $y$ share the same image.\cq
\end{definition}

\noindent 
Note that Definition~\ref{vernation_topo_def} asserts that a \emph{vernation tree} is exactly a geodesic space $(X,d)$ such that for any two loops $\gamma_1,\gamma_2$ on $X$ based at the same point $x$, then either $\Ima\gamma_1=\Ima\gamma_2$ or $\Ima\gamma_1\cap\Ima\gamma_2=\{x\}$. For our work, we mostly are interested in \emph{compact} metric spaces. When a metric space is indeed compact, it is geodesic if and only if it is a length space. See Burago \& Burago~\cite[Chapter 2]{burago} for a definition of length spaces, a proof of the just mentioned fact, and an extensive study of the notion. For a detailed study of real trees, we refer to Evans~\cite{evans} and to Paulin~\cite{PAULIN1989197}.

The next lemma asserts that when working with (geodesic) metric spaces, different notions of connectedness coincide. 

\begin{lemma}
\label{arcwise_lemma}
A metric space is pathwise connected if and only if it is arcwise connected. If $F$ is a finite subset of a geodesic space $X$, then the connected components of $X\backslash F$ are arcwise connected.
\end{lemma}

\noindent
The first statement is well-known since metric spaces are Hausdorff: see Willard~\cite[Chapter 31]{willard2004general} or Börger~\cite{arcwise} for example. We deduce the second statement from the fact that a geodesic space is locally arcwise connected. Indeed, for any $x\in X$ and $r>0$, if $d(x,y)<r$ then a geodesic between $x$ and $y$ stays inside the ball of radius $r$ and of center $x$.

\subsection{From arcs and geodesics to loops}

Here, we gather some general results that will be useful to construct loops on a metric space out of arcs or geodesics. 

\begin{lemma}
\label{three_arcs_loop}
Let $(X,d)$ be a metric space and let $x_0,x_1,x_2\in X$ be distinct. For all $i\in\mathbb{Z}/3\mathbb{Z}$, let $\gamma_i$ be an arc on $X$ from $x_{i-1}$ to $x_{i+1}$. If $\Ima \gamma_0\cap \Ima \gamma_1\cap \Ima\gamma_2=\emptyset$, then there exist three points $x_0',x_1',x_2'\in X$ and a loop $\gamma$ on $X$ such that $\Ima\gamma$ is a subset of $\Ima \gamma_0\cup \Ima \gamma_1\cup \Ima\gamma_2$ and $x_i'\in \Ima\gamma\cap\Ima \gamma_{i-1}\cap \Ima\gamma_{i+1}$ for all $i\in\mathbb{Z}/3\mathbb{Z}$. Furthermore, if $\gamma_0,\gamma_1,\gamma_2$ are geodesics, then it holds that $\gamma_i^{-1}(x_{i-1}')\leq \gamma_i^{-1}(x_{i+1}')$ for all $i\in\mathbb{Z}/3\mathbb{Z}$.
\end{lemma}

\begin{proof}
Since $\gamma_1(0)=x_0$ is an element of $\Ima\gamma_2$, we can define $s_{1,0}=\sup \gamma_1^{-1}(\Ima \gamma_2)$. We then set $x_0'=\gamma_1(s_{1,0})$, which is in $\Ima\gamma_2$ as this set is closed, and $s_{2,0}=\gamma_2^{-1}(x_0')$. Similarly, we can define $s_{1,2}=\inf \, [s_{1,0},1]\cap \gamma_1^{-1}(\Ima\gamma_0)$ because $\gamma_1(1)=x_2\in\Ima\gamma_0$. We set $x_2'=\gamma_1(s_{1,2})$ and $s_{0,2}=\gamma_0^{-1}(x_2')$. Next, we define $s_{0,1}=\inf\, [s_{0,2},1]\cap \gamma_0^{-1}(\Ima\gamma_2)$, and we set $x_1'=\gamma_0(s_{0,1})$ and $s_{2,1}=\gamma_2^{-1}(x_1')$. Finally, we construct a path $\gamma:[0,1]\to X$ as follows:
\[\gamma(t)=\begin{cases}
				\gamma_1\big(s_{1,0}+3t(s_{1,2}-s_{1,0})\big)&\text{ if }0\leq t\leq 1/3,\\
				\gamma_0\big(s_{0,2}+(3t-1)(s_{0,1}-s_{0,2})\big)&\text{ if }1/3\leq t\leq 2/3,\\
				\gamma_2\big(s_{2,1}+(3t-2)(s_{2,0}-s_{2,1})\big)&\text{ if }2/3\leq t\leq 1.
			\end{cases}\]
It is clear that $\gamma$ is continuous, that $\gamma(0)=\gamma(1)=x_0'$, and that $\Ima\gamma\subset \Ima \gamma_0\cup \Ima \gamma_1\cup \Ima\gamma_2$. Moreover, we have that $x_i'\in\Ima\gamma\cap\Ima\gamma_{i-1}\cap\Ima\gamma_{i+1}$ for all $i\in\mathbb{Z}/3\mathbb{Z}$ by construction. The assumption that $\Ima \gamma_0\cap \Ima \gamma_1\cap \Ima\gamma_2=\emptyset$ entails that $s_{1,0}<s_{1,2}$, $s_{0,2}<s_{0,1}$, and $s_{2,1}\neq s_{2,0}$. It follows that $\gamma$ is injective on $[0,1/3]$, on $[1/3,2/3]$, and on $[2/3,1]$. Furthermore, this also yields that $\gamma(0)=x_0'\notin \Ima\gamma_0$ and $\gamma(1/3)=x_2'\notin \Ima\gamma_2$. Then, observe that it holds by construction that $\gamma_1\big((s_{1,0},s_{1,2})\big)\cap (\Ima\gamma_2\cup\Ima\gamma_0)=\emptyset$ and $\gamma_0\big((s_{0,2},s_{0,1})\big)\cap\Ima \gamma_2=\emptyset$. Combining all these properties shows that $\gamma$ is injective on $[0,1)$. Thus, $\gamma$ is indeed a loop.

Now, let us assume that $\gamma_0,\gamma_1,\gamma_2$ are geodesics. We use the triangular inequality to write
\begin{align}
\label{three_arcs_loop1}
d(x_0,x_0')+d(x_0',x_1)=d(x_0,x_1)&\leq d(x_0,x_2')+d(x_2',x_1)\\\label{three_arcs_loop2}
d(x_2,x_2')+d(x_2',x_1)=d(x_2,x_1)&\leq d(x_2,x_0')+d(x_0',x_1)\\
\label{three_arcs_loop3}
d(x_0,x_2')+d(x_2',x_2)=d(x_0,x_2)&=d(x_0,x_0')+d(x_0',x_2).
\end{align}
The equalities inside (\ref{three_arcs_loop1}), (\ref{three_arcs_loop2}), and (\ref{three_arcs_loop3}) come respectively from $x_0'\in\Ima\gamma_2$, $x_2'\in\Ima\gamma_0$, and $x_2',x_0'\in\Ima\gamma_1$. Adding these inequalities yields that $d(x_2',x_2)\leq d(x_0',x_2)$. This entails that $\gamma_1^{-1}(x_0')\leq \gamma_1^{-1}(x_2')$ since $\gamma_1$ is a geodesic on $X$ from $x_0$ to $x_2$. We obtain the other desired inequalities by symmetry.
\end{proof}

\begin{lemma}
\label{two_arcs_loop}
Let $(X,d)$ be a metric space and let $x,y\in X$ be distinct. Let $s\in[0,1]$ and $\gamma_1,\gamma_2$ be two arcs on $X$ from $x$ to $y$. If $\gamma_1(s)\notin\Ima\gamma_2$, then there exists a loop $\gamma$ on $X$ based at $\gamma_1(s)$ such that $\Ima\gamma\backslash\Ima\gamma_i$ is an open subset of $\Ima\gamma_{3-i}$ for both $i\in\{1,2\}$. Moreover, if $\gamma_1$ is a geodesic, then there is $z\in \Ima \gamma\backslash\Ima\gamma_1$ such that $d(x,\gamma_1(s))=d(x,z)$.
\end{lemma}

\begin{proof}
We set $x_0=y$, $x_1=x$, and $x_2=\gamma_1(s)$. For all $t\in[0,1]$, we then define $\hat{\gamma}_0(t)=\gamma_1(s(1-t))$, and
$\hat{\gamma}_1(t)=\gamma_1(1-(1-s)t)$, and $\hat{\gamma}_2(t)=\gamma_2(t)$. Clearly, $\hat{\gamma}_0,\hat{\gamma}_1,\hat{\gamma}_2$ are arcs. Since $\gamma_1$ is injective, we have $\Ima\hat{\gamma}_0\cap\Ima\hat{\gamma}_1=\{\gamma_1(s)\}$. By assumption, this implies that $x_0,x_1,x_2$ are distinct and that $\Ima\hat{\gamma}_0\cap\Ima\hat{\gamma}_1\cap\Ima\hat{\gamma}_2=\emptyset$. Hence, Lemma~\ref{three_arcs_loop} provides a loop $\gamma$ on $X$ which hits $\gamma_1(s)$ and such that $\Ima\gamma \subset \Ima \gamma_1\cup \Ima \gamma_2$. We can assume that $\gamma$ is based at $\gamma_1(s)$ by doing a time-shift. Let $i\in\{1,2\}$, we already know that $\Ima\gamma\backslash\Ima\gamma_i$ is a subset of $\Ima\gamma_{3-i}\backslash\{x,y\}$. Recall that an arc is a topological embedding of $[0,1]$, so $\Ima\gamma_{3-i}\backslash\{x,y\}$ is homeomorphic to $\mathbb{R}$. Similarly, a loop is a topological embedding of the circle, which cannot be embedded into $\mathbb{R}$, so there is $u\in \Ima\gamma\cap \Ima\gamma_i$ and $\Ima\gamma\backslash\{u\}$ is homeomorphic to $\mathbb{R}$. Moreover, the subset $\Ima\gamma_i$ is closed in $X$, so $\Ima\gamma\backslash\Ima\gamma_i$ is open in $\Ima\gamma\backslash\{u\}$. Therefore, $\Ima\gamma\backslash \Ima\gamma_i$ is homeomorphic to an open subset of $\mathbb{R}$. The Brouwer invariance of domain theorem thus implies that it is an open subset of $\Ima\gamma_{3-i}\backslash\{x,y\}$. This completes the proof because $\Ima\gamma_{3-i}\backslash\{x,y\}$ is open in $\Ima\gamma_{3-i}$.

Let us now assume that $\gamma_1$ is a geodesic, which implies that $t\in[0,1]\mapsto d(x,\gamma_1(t))$ is strictly increasing. The fact that $\Ima\gamma\backslash\Ima\gamma_2$ is an open neighborhood of $\gamma_1(s)$ in $\Ima\gamma_1$ ensures that we have $\varepsilon\in(0,s)$ such that $\gamma_1(s-\varepsilon)$ and $\gamma_1(s+\varepsilon)$ belong to $\Ima\gamma$. Hence, there are $t_-,t_+\in(0,1)$ such that $d(x,\gamma(t_-))<d(x,\gamma_1(s))<d(x,\gamma(t_+))$. Then, the intermediate value theorem gives some $t_0\in(0,1)$ such that $d(x,\gamma(t_0))=d(x,\gamma_1(s))$. We set $z=\gamma(t_0)\in\Ima\gamma$ and we have $z\neq\gamma(0)=\gamma_1(s)$ because $\gamma$ is a loop based at $\gamma_1(s)$. Since $t\mapsto d(x,\gamma_1(t))$ is injective, we get that $z\notin\Ima\gamma_1$.
\end{proof}

Now, recall from (\ref{circle_dist_cano}) that $(\mathcal{C},\delta)$ stands for the circle with unit length. It is isometric to the quotient metric space (see Section~\ref{space}) of $[0,1]$ induced by the pseudo-distance $\delta(a,b)=\min(|a\!-\!b|,1\!-\!|a\!-\!b|)$. Recall that $r\cdot\mathcal{C}=(\mathcal{C},r\delta)$ for $r\!>\!0$.

\begin{proposition}
\label{loop_iso}
Let $X$ be a geodesic space. If $X$ is homeomorphic to $\mathcal{C}$, then $X$ is isometric to $2\Delta\cdot \mathcal{C}$ with some $\Delta>0$.
\end{proposition}

\begin{proof}
Let $o\in X$. We know that there is a loop $\gamma$ on $X$ based at $o$ such that $\Ima\gamma=X$. In particular, $X$ is compact so we have $s\in(0,1)$ such that $d(o,\gamma(s))=\sup_{x\in X}d(o,x)=:\Delta$. Let $t_1\in(0,s)$ and let $g$ be a geodesic on $X$ from $o$ to $\gamma(t_1)$. If $x\in\Ima g$ but $x\neq \gamma(t_1)$, then $d(o,x)< d(o,\gamma(t_1))\leq \Delta$. Thus, we get that $\gamma(s)\notin\Ima g$, because $\gamma(t_1)\neq \gamma(s)$ since $t_1<s$. If $\gamma([0,t_1])\neq\Ima g$ then Lemma~\ref{two_arcs_loop} implies that there is another loop $\gamma'$ on $X$ such that $\Ima\gamma'\subset \gamma([0,t_1])\cup\Ima g$. However, the only subset of $X$ homeomorphic to a circle is $X$ itself, because any proper connected subset of a circle is homeomorphic to an interval of $\mathbb{R}$. Hence, $\Ima\gamma'=X$ and then $\gamma(s)\in\gamma([0,t_1])$, which contradicts the injectivity of $\gamma$ on $[0,1)$. Therefore, we find that $\Ima g=\gamma([0,t_1])$. In particular, if $0\leq t_2< t_1<s$ then $d(o,\gamma(t_1))=d(o,\gamma(t_2))+d(\gamma(t_2),\gamma(t_1))$. By time-reversing $\gamma$, if $s<t_1<t_2\leq 1$ then $d(o,\gamma(t_1))=d(o,\gamma(t_2))+d(\gamma(t_2),\gamma(t_1))$.

It follows that the continuous functions $t\in[0,s]\mapsto d(o,\gamma(t))$ and $t\in[0,1-s]\mapsto d(o,\gamma(1-t))$ are strictly increasing. We have an increasing and continuous bijection $\lambda:[0,1]\to[0,1]$ such that $\lambda(1/2)=s$ and $d(o,\gamma\circ\lambda(t))=2\Delta\delta(0,t)$ for all $t\in[0,1]$. Moreover, if $t_1,t_2\in[0,s]$ or if $t_1,t_2\in[s,1]$, then $d(\gamma\circ\lambda(t_1),\gamma\circ\lambda(t_2))=2\Delta|t_1-t_2|=2\Delta\delta(t_1,t_2)$. A circle without two points has exactly two connected components. Since the subsets $\gamma((0,s))$ and $\gamma((s,1))$ are connected and disjoint, we identify them as the connected components of $X\backslash\{o,\gamma(s)\}$. This means that if $t_1\in(0,s)$ and $t_2\in(s,1)$, then a geodesic on $X$ from $\gamma(t_1)$ to $\gamma(t_2)$ hits $o$ or $\gamma(s)$, which further entails that
\[d(\gamma\circ\lambda(t_1),\gamma\circ\lambda(t_2))=\min\big(d(\gamma\circ\lambda(t_1),o)+d(o,\gamma\circ\lambda(t_2)),d(\gamma\circ\lambda(t_1),\gamma(s))+d(\gamma(s),\gamma\circ\lambda(t_2))\big).\]
The previous case yields that the right-hand side is equal to $2\Delta\delta(t_1,t_2)$. To sum up, the continuous and surjective function $\gamma\circ\lambda:[0,1]\to X$ satisfies $d(\gamma\circ\lambda(t_1),\gamma\circ\lambda(t_2))=2\Delta\delta(t_1,t_2)$ for any $t_1,t_2\in[0,1]$. Hence, $X$ is isometric to $2\Delta\cdot \mathcal{C}$.
\end{proof}

\subsection{General properties of topological vernation trees}

The properties of topological vernation trees studied in this section give a better understanding of Definition~\ref{vernation_topo_def} by echoing those of real trees and of the coding (\ref{dv_easy}) by excursions. Recall Definition~\ref{geo_def} of geodesic spaces and real trees.

\begin{proposition}
\label{equiv_tree}
If $(X,d)$ is a compact geodesic space, then the following statements are equivalent.
\begin{enumerate}
\item[$(i)$] The space $(X,d)$ is a real tree.
\item[$(ii)$] For all distinct $x,y\! \in \! X$, there is $u\! \in \! X\backslash\{x,y\}$ such that $x$ and $y$ are in different connected components of $X\backslash\{u\}$.
\item[$(iii)$] There are no loops on $X$.
\end{enumerate}
\end{proposition}

\noindent
The implications $(i)\Rightarrow(ii)\Rightarrow(iii)$ are obvious. The implication $(iii)\Rightarrow (i)$ follows from Lemma~\ref{two_arcs_loop}. See also Chiswell~\cite[Proposition 2.3]{chiswell2001introduction}. Note that the point $(iii)$ ensures that a real tree is also a vernation tree, which might not be so clear by only looking at the definitions. The interest of the point $(ii)$ is that it can be naturally adapted into an equivalent definition for vernation trees as follows.

\begin{proposition}
\label{equiv_vern}
If $(X,d)$ is a compact geodesic space, then the following statements are equivalent.
\begin{enumerate}
\item[$(i)$] The space $(X,d)$ is a vernation tree.
\item[$(ii)$] For all distinct $x,y\! \in\! X$, there are $u,v\! \in\! X\backslash\{x,y\}$ such that $x,y$ are in different connected components of $X\backslash\{u,v\}$.
\item[$(iii)$] For all distinct $x,y\in X$, there are $u,v\in X\backslash\{x,y\}$ such that $x,y$ are in different connected components of $X\backslash\{u,v\}$ and such that $2\min\left(d(x,u),d(x,v),d(y,u),d(y,v)\right)\geq d(x,y)$.
\end{enumerate}
\end{proposition}

\begin{proof}
The implication $(iii)\Rightarrow(ii)$ is obvious. Let us assume that $(ii)$ holds. Let $\gamma_1,\gamma_2$ be two loops on $X$ based at $x$ and let $s\in(0,1)$ such that $\gamma_1(s)\notin\Ima\gamma_2$. We set $a_1=\sup\, [0,s]\cap\gamma_1^{-1}(\Ima\gamma_2)$ and $b_1=\inf\, [s,1]\cap\gamma_1^{-1}(\Ima\gamma_2)$. We have $a_1<s<a_2$ and $\gamma_1(a_1),\gamma_2(b_1)\in\Ima\gamma_2$ because $\gamma_1$ is continuous and $\Ima\gamma_2$ is closed. We choose $a_2,b_2\in[0,1]$ such that $\gamma_2(a_2)=\gamma_1(a_1)$ and $\gamma_2(b_2)=\gamma_1(b_1)$. Even if it means reversing $\gamma_2$, we can assume that $a_2\leq b_2$ without loss of generality. We can define three paths $\gamma_1',\gamma_2',\gamma_3'$ on $X$ from $\gamma_1(a_1)$ to $\gamma_1(b_1)$ by setting for all $t\in[0,1]$,
\[\gamma_1'(t)=\gamma_1(a_1+t(b_1-a_1)),\quad \gamma_2'(t)=\gamma_2(a_2+t(b_2-a_2)),\quad\text{ and }\quad \gamma_3'(t)=\begin{cases}
								\gamma_2(a_2-t)&\text{ if }t\leq a_2,\\
								x&\text{ if }a_2\leq t\leq b_2,\\
								\gamma_2(b_2+1-t)&\text{ if }b_2\leq t\leq 1.
							   \end{cases}\]
By construction, it is clear that $\Ima\gamma_1'\cap\Ima\gamma_2=\{\gamma_1(a_1),\gamma_1(b_1)\}$, so it follows that $\Ima\gamma_1'\cap\Ima\gamma_2'=\Ima\gamma_1'\cap\Ima\gamma_3'=\{\gamma_1(a_1),\gamma_1(b_1)\}$. Moreover, $\Ima\gamma_2'\cap\Ima\gamma_3'=\{\gamma_1(a_1),\gamma_1(b_1)\}$ because $\gamma_2$ is injective on $[0,1)$. In the case where $\gamma_1(a_1)\neq\gamma_1(b_1)$, there exist $u,v\in X\backslash\{\gamma_1(a_1),\gamma_1(b_1)\}$ such that $\gamma_1(a_1)$ and $\gamma_1(b_1)$ are in different connected components of $X\backslash\{u,v\}$ thanks to $(ii)$. In particular, we would have $\Ima\gamma_i'\cap\{u,v\}\neq\emptyset$ for $i=1,2,3$, which is impossible regarding the intersections between the $\Ima\gamma_i'$. Hence, we find that $\gamma_1(a_1)=\gamma_1(b_1)$ and since $a_1<s<b_1$, we conclude that $a_1=0$ and $b_1=1$. This means that $\Ima\gamma_1\cap\Ima\gamma_2=\{x\}$. By contraposition, if $\Ima\gamma_1\cap\Ima\gamma_2\neq\{x\}$ then $\Ima\gamma_1\subset\Ima\gamma_2$ and $\Ima\gamma_1=\Ima\gamma_2$ by symmetry, which ends the proof of $(ii)\Rightarrow(i)$.

Let us assume that $(i)$ holds to prove $(iii)$. Let $g$ be a geodesic on $X$ from $x$ to $y$ with $x\neq y$, and let us set $u=g(1/2)$ so that $d(x,u)=d(y,u)=d(x,y)/2$. If $x$ and $y$ are in different connected components of $X\backslash\{u\}$, then we simply set $v=u$. Otherwise, we have by Lemma~\ref{arcwise_lemma} an arc $\gamma_2$ on $X$ from $x$ to $y$ such that $g(1/2)\notin\Ima\gamma_2$. Thus, Lemma~\ref{two_arcs_loop} provides a loop $\gamma$ on $X$ such that $\Ima\gamma\backslash\Ima\gamma_2$ is an open neighborhood of $u$ in $\Ima g$, and a point $v\in \Ima\gamma\backslash\Ima g$ such that $d(x,u)=d(x,v)$. Observe that $d(x,y)\leq d(x,y)/2+d(y,v)$, so we only need to verify that $x$ and $y$ are in different connected components of $X\backslash\{u,v\}$. By contradiction, let $\gamma_2'$ be an arc on $X$ from $x$ to $y$ such that $g(1/2),v\notin \Ima \gamma_2'$. By Lemma~\ref{two_arcs_loop}, there is a loop $\gamma'$ on $X$ such that $\Ima\gamma'\subset \Ima g\cup\Ima \gamma_2'$ and such that $\Ima\gamma'\backslash\Ima\gamma_2'$ is an open neighborhood of $u$ in $\Ima g$. In particular, $v$ does not belong to $\Ima\gamma'$ so $\Ima\gamma'\neq \Ima\gamma$. However, $\Ima\gamma\cap\Ima\gamma'$ contains an open neighborhood of $u$ in $\Ima g$. Thus, this intersection is infinite, which yields a contradiction because $X$ is a vernation tree.
\end{proof}

The next result describes the structure of a vernation tree decomposed at one of its loops.

\begin{proposition}
\label{branching_at_loop}
Let $(X,d)$ be a compact vernation tree and let $\gamma$ be a loop on $X$. We write $L=\Ima\gamma$ and $d(x,L)=\min_{u\in L} d(x,u)$ for all $x\in X$. For all $u\in L$, we set $X_u=\{x\in X\, :\, d(x,L)=d(x,u)\}$. The following points hold true.
\begin{enumerate}
\item[$(i)$] The collection $\{X_u\, :\, u\in L\}$ is a partition of $X$.
\item[$(ii)$] For all $u\in L$, the subset $X_u$ of $X$ is closed and its boundary $\partial X_u$ is $\{u\}$.
\item[$(iii)$] For all $u\in L$, the subspace $(X_u,d)$ is a compact vernation tree.
\item[$(iv)$] The subspace $(L,d)$ is a compact vernation tree.
\item[$(v)$] Let $u,v\in L$ be distinct. For all $x\in X_u$ and $y\in X_v$, it holds that
\begin{equation}
\label{central_loop}
d(x,y)=d(x,u)+d(u,y)=d(x,u)+d(u,v)+d(v,y).
\end{equation}
\end{enumerate}
\end{proposition}

\begin{proof}
First, the function $d(\cdot,L)$ is well-defined and continuous on $X$ because $L$ is compact. Thus, the subsets $X_u,u\in L$ are closed in $X$, and so are compact, and they cover $X$. For all $u\in L$, we denote by $Y_u$ the pathwise connected component of $X\backslash(L\backslash\{u\})$ that contains $u$. Clearly, $X_u\cap L=Y_u\cap L=\{u\}$. We claim that $X_u$ is a subset of $Y_u$. Indeed, if $g$ is a geodesic from $x\in X_u\backslash\{u\}$ to $u$, then $d(x,g(t))<d(x,L)$ for all $t\in[0,1)$, which ensures that $g$ is a path on $X\backslash(L\backslash\{u\})$. Now, let $u,v\in L$ with $u\neq v$ and let $x\in Y_u\cap Y_v$. The metric spaces $Y_u,Y_v,L$ are pathwise connected. By Lemma~\ref{arcwise_lemma}, we have an arc $\gamma_u$ on $Y_u$ from $x$ to $u$, an arc $\gamma_v$ on $Y_v$ from $v$ to $x$, and an arc $\gamma_L$ on $L$ from $u$ to $v$. Note that $\Ima\gamma_u\cap L=\{u\}$ and $\Ima\gamma_v\cap L=\{v\}$, and thus $\Ima\gamma_u\cap \Ima\gamma_v\cap L=\emptyset$. Lemma~\ref{three_arcs_loop} then gives us a loop $\gamma$ on $X$ which meets $\Ima\gamma_u\cap\Ima\gamma_v$ and which hits $u$ and $v$. In particular, $\Ima\gamma\neq L$ but $\Ima\gamma\cap L$ contains at least two points. This contradicts the fact that $X$ is a vernation tree. Hence, the $Y_u,u\in L$ are pairwise disjoint. Since the $X_u,u\in L$ cover $X$, we get that $X_u=Y_u$ for all $u\in L$ and that $(i)$ holds true.

We already showed that $X_u$ is closed in $X$. Let $u\in L$, $x\in X_u\backslash\{u\}$ and $y\in X$ with $d(x,y)<d(x,L)$. A geodesic from $x$ to $y$ cannot hit $L$, so $x,y$ are in the same pathwise connected component of $X\backslash(L\backslash\{u\})$. This means that $y\in Y_u=X_u$ and $y\neq u$, and so $X_u\backslash\{u\}$ is open. We obtain $(ii)$ because $X_u$ is not open, since $X$ is connected as a geodesic space.

Let $u,v\in L$ be distinct, $x\in X_u$ and $y\in X_v$. A geodesic from $x$ to $y$ has to hit $L\backslash\{u\}$ by definition of $Y_u$. Thus, we can define $s=\inf g^{-1}(L)$ and we have $g(s)\in L$ by continuity of $g$. Since $g$ induces a path from $x$ to $g(s)$ that stays outside $L\backslash\{g(s)\}$, we have $x\in Y_{g(s)}$ and so $g(s)=u$ thanks to $(i)$. In particular, we get $d(x,y)=d(x,u)+d(u,y)$. By symmetry, we also have $d(u,y)=d(u,v)+d(v,y)$, which yields $(v)$. We just proved that any geodesic from some point of $X_u$ to some point outside $X_u$ must hit $u$. Therefore, a geodesic between two points of $X_u$ stays in $X_u$, because it would have to hit $u$ at least twice otherwise. Similarly, a geodesic between two points of $L$ cannot leave $L$ without hitting some $u\in L$ twice. Hence, the compact subspace $(X_u,d)$ and $(L,d)$ are geodesic for all $u\in L$. Then, they are vernation trees by Definition~\ref{vernation_topo_def} because a loop on one of them is also a loop on $X$. This concludes the proof of $(iii)$ and $(iv)$.
\end{proof}

\noindent
We point out that this decomposition can also be observed on a vernation tree coded by an excursion, in the sense of Definition~\ref{excursion}, via (\ref{dv_easy}). Indeed, recall $\Delta_s,x_s^t,\delta_s$ from (\ref{delta_x}) and (\ref{circle_dist_s}), and let $f$ be an excursion and $s\in[0,1]$ such that $\Delta_s>0$. If $X=\vern[f]$ then one can use (\ref{loop_on_coded}) and Proposition~\ref{branching_property} to construct a loop $\gamma$ on $X$ such that $\Ima\gamma$ is the loop associated with $s$ and such that $X_{\gamma(r)}$ is the projection of $\{t\in[0,1]\, :\, \delta_s(x_s^t,r\Delta_s)=0\}$ for any $r\in[0,1)$.

When $(X,d)$ is a compact metric space, we write $\diam(X)=\sup_{x,y\in X}d(x,y)$. According to the next result, a compact vernation tree has only a countable number of subsets homeomorphic to the circle.

\begin{proposition}
\label{number_loops}
If $(X,d)$ is a compact vernation tree, then $\{\Ima\gamma : \gamma\text{ loop on }X,\diam(\Ima\gamma)\geq\varepsilon\}$ is finite for all $\varepsilon\! >\! 0$.
\end{proposition}

\begin{proof}
We assume by contradiction that the set of interest is infinite. Then, the compactness of $X$ yields that it admits two distinct elements $L,L'$ such that there are $x,y\! \in\! L$  and $x',y'\!\in\! L'$ with $d(x,y)\geq\varepsilon$, $d(x',y')\geq\varepsilon$, and $d(x,x')+d(y,y')<\varepsilon$. By Definition~\ref{vernation_topo_def} of vernation trees, there is $v\in L'$ such that $L$ and $L'\backslash\{v\}$ are disjoint. Then, there is $u\in L$ such that $L'\backslash\{v\}$ encounters $X_u$, where $\{X_u:u\in L\}$ is the partition as in Proposition~\ref{branching_at_loop}. We observe that $L'\backslash\{v\}$ is connected, since it is homeomorphic to $(0,1)$, so it is in fact a subset of $X_u$ because it would need to meet $\partial X_u=\{u\}\subset L$ otherwise. As $X_u$ is closed, it even holds that $L'\subset X_u$. By (\ref{central_loop}), we have $d(a,a')\geq d(a',u)$ for all $a\in L$ and $a'\in L'$. Thus, $\varepsilon\leq d(x',y')\leq d(x',u)+d(y',u)\leq d(x',x)+d(y',y)<\varepsilon$, which is a contradiction.
\end{proof}

Now, recall from Section~\ref{def_GH} that $\mathbb{K}$ is the space of compact metric spaces up to an isometry, endowed with the Gromov--Hausdorff distance $\mathtt{d}_{\mathrm{GH}}$, and that it is separable and complete. Let $\mathbb{L}$ be the space of isometry classes of compact geodesic spaces and let $\mathbb{T}$ be the space of isometry classes of compact real trees. They are both closed subsets of $\mathbb{K}$, see for example Evans~\cite[Theorem 4.19 and Lemma 4.22]{evans}, so they are also separable and complete. We denote by $\mathbb{V}$ the space of isometry classes of compact vernation trees. We have the chain of inclusions $\mathbb{T}\subset\mathbb{V}\subset\mathbb{L}\subset\mathbb{K}$. The next theorem asserts that $\mathbb{V}$ is closed in $\mathbb{K}$, and so that it is separable and complete.

\begin{theorem}
\label{closed_vern}
Let $(X_n,d_n)$ be a compact vernation tree for all $n\geq 0$ and let $(X,d)$ be a compact metric space. If $X_n\longrightarrow X$ for the Gromov--Hausdorff topology then $X$ is a compact vernation tree.
\end{theorem}

\begin{proof}
First, the subspace $\mathbb{L}$ is closed so $X$ is a compact geodesic space. Let $x,y\in X$ with $x\neq y$, we only need to find $u,v\in X\backslash\{x,y\}$ such that $x$ and $y$ are in different connected components of $X\backslash\{u,v\}$. There exists a sequence $\left(\mathcal{R}_n\right)$ of correspondences respectively between $X_n$ and $X$ such that $\mathrm{dis}\left(\mathcal{R}_n\right)\longrightarrow0$. For any $n\geq 0$, there are $x_n,y_n\in X_n$ such that $(x_n,x)\in\mathcal{R}_n$ and $(y_n,y)\in\mathcal{R}_n$. In particular, $\lim d_n(x_n,y_n)=d(x,y)>0$ so we can assume that $x_n\neq y_n$ for all $n\geq 0$ by only considering large enough $n$. Hence, there are $u_n,v_n\in X_n$ as in the point $(iii)$ of Proposition~\ref{equiv_vern}. Then, we have $u_n',v_n'\in X$ such that $(u_n,u_n')\in\mathcal{R}_n$ and $(v_n,v_n')\in\mathcal{R}_n$ for all $n\geq 0$. Since $X$ is compact, we can assume that there exist $u,v\in X$ such that $u_n'\longrightarrow u$ and $v_n'\longrightarrow v$, by restricting ourselves on a subsequence. It follows that $2d(x,u)=\lim 2d_n(x_n,u_n)\geq\lim d_n(x_n,y_n)=d(x,y)$, so $x\neq u$. In the same way, we show that $u,v\in X\backslash\{x,y\}$. We are going to show that $x$ and $y$ are in different connected components of $X\backslash\{u,v\}$.

Let $\gamma:[0,1]\longrightarrow X$ be a path on $X$ from $x$ to $y$. For all $n\geq 1$ and $1\leq i\leq n-1$, there exists $w_{n,i}\in X_n$ such that $(w_{n,i},\gamma(i/n))\in \mathcal{R}_n$. We also write $w_{n,0}=x_n$ and $w_{n,n}=y_n$. We consider a path $\gamma_n:[0,1]\longrightarrow X_n$ such that the map $s\in[0,1]\mapsto \gamma_n((i+s)/n)$ is a geodesic from $w_{n,i}$ to $w_{n,i+1}$ for all $0\leq i\leq n-1$. Since $\gamma_n$ is a path on $X_n$ from $x_n$ to $y_n$, we have $\{u_n,v_n\}\cap\Ima \gamma_n\neq\emptyset$, so by restricting ourselves on another subsequence, we can assume that $u_n\in\Ima\gamma_n$ for all $n\geq 1$ without loss of generality. We choose $t_n\in[0,1)$ such that $\gamma_n(t_n)=u_n$, and we can assume that there exists $t\in[0,1]$ such that $t_n\longrightarrow t$. Let us prove that $\gamma(t)=u$. Observe that $d_n\left(w_{n,\lfloor nt_n\rfloor},u_n\right)\leq d_n\left(w_{n,\lfloor nt_n\rfloor },w_{n,\lfloor nt_n\rfloor +1}\right)$ by definition of $\gamma_n$. By definition of the $w_{n,i}$, we find \[d_n\left(w_{n,\lfloor nt_n\rfloor},u_n\right)\leq2\mathrm{dis}\left(\mathcal{R}_n\right)+\sup_{\substack{s,r\in[0,1]\\ |s-r|\leq 1/n}}|\gamma(s)-\gamma(r)|,\]
so $\lim d_n(w_{n,\lfloor nt_n\rfloor},u_n)=0$ thanks to the continuity of $\gamma$. It follows that $\lim d\left(\gamma(\lfloor nt_n\rfloor/n),u_n'\right)=0$ too. Moreover, we have $t_n\longrightarrow t$, so $\gamma(\lfloor n t_n\rfloor/n)\longrightarrow \gamma(t)$. Recall that $u_n'\longrightarrow u$, thus $d(\gamma(t),u)=0$. 
\end{proof} 

\subsection{Contraction of small loops on a vernation tree}

This section is devoted to the construction of an operation that rescales every small loop of a given topological vernation tree. This then yields a transformation that contracts every loop whose diameter is smaller than a threshold $\varepsilon>0$ into a single point. Through our coding (\ref{dv_easy}) by excursions, this is analogous to the map $f\mapsto f-Jf+J^{2\varepsilon}\! f$, where $Jf$ and $J^{2\varepsilon}\! f$ are as in (\ref{Jeps_def}), which removes all the small jumps of an excursion. Here, if $S$ is a subset of $(0,1)$ then we denote by $\mathrm{Int}\, S$ its interior. We also denote by $\mathrm{Leb}$ the Lebesgue measure on $(0,1)$. We fix a compact vernation tree $(X,d)$ and we set
\begin{equation}
\label{set_loops_metric}
\ell(X)=\{\Ima\gamma\ :\ \gamma\text{ loop on }X\}=\{L\subset X\ :\ \text{homeomorphic to the circle}\}.
\end{equation}

\begin{proposition}
Let $x,y\!\in\! X$, $\varepsilon\!>\!0$, $\eta\!\in\![0,1]$. There is $d_{\varepsilon,\eta}(x,y)\in\mathbb{R}$ such that for any geodesic $g$ on $X$ from $x$ to $y$,
\begin{equation}
\label{vern_contracte}
d_{\varepsilon,\eta}(x,y)=d(x,y)-(1-\eta)d(x,y)\mathrm{Leb}\left(\{t\in(0,1)\ :\ \exists L\in\ell(X),\ t\in\mathrm{Int}\, g^{-1}(L)\text{ and }\diam(L)<\varepsilon\}\right).
\end{equation}
\end{proposition}
\begin{proof}
First, the right-hand side of (\ref{vern_contracte}) is well-defined because the subset measured by $\mathrm{Leb}$ is open. We only need to show that it does not depend on the choice of the geodesic $g$. Let $g_1,g_2$ be two geodesics on $X$ from $x$ to $y$. Let $t\in(0,1)$ and let $I$ be an open interval containing $t$ such that there is $L\in\ell(X)$ with $g_1(I)\subset L$ and $\diam(L)<\varepsilon$. Let $s\in I$. If $g_2(s)\in\Ima g_1$ then $g_2(s)=g_1(s)$ because $g_1,g_2$ are geodesics, so $g_2(s)\in L$. If $g_2(s)\notin\Ima g_1$ then Lemma~\ref{two_arcs_loop} asserts that there is a loop $\gamma$ on $X$ based at $g_2(s)$ such that $\Ima\gamma\backslash\Ima g_2$ is an open subset of $\Ima g_1$. Moreover, there exists $z\in\Ima\gamma\cap\Ima g_1$ such that $d(x,z)=d(x,g_2(s))$, so $z\in\Ima g_1$ and $z=g_1(s)$ because $g_1,g_2$ are geodesics. Both $g_1^{-1}(\Ima\gamma)$ and $g_1^{-1}(L)$ contain a neighborhood of $s$, thus $\Ima\gamma=L$ by Definition~\ref{vernation_topo_def}. Whatever the case, we have proved that $g_2(s)\in L$, which gives $g_2(I)\subset L$. By symmetry, it follows that $d_{\varepsilon,\eta}$ is indeed well-defined.
\end{proof}

Clearly, we have $\eta d(x,y)\leq d_{\varepsilon,\eta}(x,y)\leq d(x,y)$ for any $x,y\in X$, so we have $d_{\varepsilon,\eta}(x,x)=0$. Then, $d_{\varepsilon,\eta}$ is symmetric because a geodesic $g_1$ from $x$ to $y$ gives a geodesic from $y$ to $x$ after reversing the time. Therefore, we only need to show that $d_{\varepsilon,\eta}$ enjoys the triangular inequality to obtain the following.

\begin{proposition}
For any $\varepsilon>0$ and $\eta\in(0,1]$, the function $d_{\varepsilon,\eta}$ is a distance on $X$, and $d_{\varepsilon,0}$ is a pseudo-distance on $X$.
\end{proposition}

\begin{proof}
Let $L\subset X$ be homeomorphic to the circle and let $u,v\in L$ be distinct. Thanks to Proposition~\ref{branching_at_loop} $(iv)$, we know that there is a geodesic $g$ on $X$ from $u$ to $v$ such that $\Ima g\subset L$. Now, if $L'\in\ell(X)$ verifies that $g^{-1}(L)$ has non-empty interior in $(0,1)$, then $L\cap L'$ is infinite and so $L=L'$ by Definition~\ref{vernation_topo_def}. Hence, (\ref{vern_contracte}) yields that if $u,v\in L$ then 
\begin{equation}
\label{distance_contracte1}
d_{\varepsilon,\eta}(u,v)=(1-(1-\eta)\I{\diam(L)<\varepsilon})d(u,v).
\end{equation}
Let $x_0,x_1,x_2\in X$ and let $g_1$ be a geodesic on $X$ from $x_0$ to $x_2$. If $x_1\in \Ima g_1$ then $d(x_0,x_2)=d(x_0,x_1)+d(x_1,x_2)$ and $g$ induces a geodesic from $x_0$ to $x_1$ and a geodesic from $x_1$ to $x_2$. Hence, we find that if $x_1\in\Ima g_1$ then
\begin{equation}
\label{distance_contracte2}
d_{\varepsilon,\eta}(x_0,x_2)=d_{\varepsilon,\eta}(x_0,x_1)+d_{\varepsilon,\eta}(x_1,x_2).
\end{equation}
Now we assume that $x_1\notin\Ima g_1$, and we fix a geodesic $g_0$ from $x_2$ to $x_1$ and a geodesic $g_2$ from $x_1$ to $x_0$. If $\Ima g_0\cap\Ima g_1\cap \Ima g_2=\emptyset$, then we have a loop $\gamma$ on $X$ and three points $x_0',x_1',x_2'\in \Ima\gamma$ as given by Lemma~\ref{three_arcs_loop}. If $\Ima g_0\cap\Ima g_1\cap \Ima g_2$ is non-empty then we choose $x_0'=x_1'=x_2'$ in this set. Thanks to (\ref{distance_contracte1}), there is $\kappa\in\{\eta,1\}$ such that $d_{\varepsilon,\eta}(x_i',x_{i+1}')=\kappa d(x_i',x_{i+1}')$ for all $i\in\mathbb{Z}/3\mathbb{Z}$, whatever the case we are in. Furthermore, (\ref{distance_contracte2}) and Lemma~\ref{three_arcs_loop} yield that
\begin{align*}
d_{\varepsilon,\eta}(x_0,x_1)&=d_{\varepsilon,\eta}(x_0,x_0')+d_{\varepsilon,\eta}(x_0',x_1')+d_{\varepsilon,\eta}(x_1',x_1),\\
d_{\varepsilon,\eta}(x_0,x_2)&=d_{\varepsilon,\eta}(x_0,x_0')+d_{\varepsilon,\eta}(x_0',x_2')+d_{\varepsilon,\eta}(x_2',x_2),\\
d_{\varepsilon,\eta}(x_1,x_2)&=d_{\varepsilon,\eta}(x_1,x_1')+d_{\varepsilon,\eta}(x_1',x_2')+d_{\varepsilon,\eta}(x_2',x_2).
\end{align*}
The triangular inequality $d_{\varepsilon,\eta}(x_0,x_2)\leq d_{\varepsilon,\eta}(x_0,x_1)+d_{\varepsilon,\eta}(x_1,x_2)$ follows from that of $d$.
\end{proof}

Let $\varepsilon>0$ and $\eta\in(0,1]$. Recall that $d$ and $d_{\varepsilon,\eta}$ are bilipschitz, so $d_{\varepsilon,\eta}$ is continuous on $X\times X$, and $(X,d)$ and $(X,d_{\varepsilon,\eta})$ admit the same topology. Thus, $(X,d_{\varepsilon,\eta})$ is compact and satisfies $(b)$ in Definition~\ref{vernation_topo_def}. Let $x,y\in X$ be distinct. If $g$ is a geodesic on $(X,d)$ from $x$ to $y$, then the function $t\in[0,1]\longmapsto d_{\varepsilon,\eta}(x,g(t))$ is continuous and increasing, so there exists a continuous and increasing bijection $\lambda:[0,1]\to[0,1]$ such that $d_{\varepsilon,\eta}(x,g\circ \lambda(t))=td_{\varepsilon,\eta}(x,y)$ for all $t\in[0,1]$. Thanks to (\ref{distance_contracte2}), it follows that $g\circ\lambda$ is a geodesic on $(X,d_{\varepsilon,\eta})$ from $x$ to $y$. Thus, $(X,d_{\varepsilon,\eta})$ is a compact vernation tree when $\eta>0$. This still holds for $\eta=0$ in the following sense because $d_{\varepsilon,0}$ is a Lipschitz pseudo-distance on $X$.

For all $\varepsilon>0$, let us denote by $\rG_\varepsilon(X)$ the quotient metric space of $X$ induced by the pseudo-distance $d_{\varepsilon,0}$, as defined in Section~\ref{space}. Recall from Section~\ref{def_GH} the Gromov--Hausdorff topology and from (\ref{set_loops_metric}) the notation $\ell$.

\begin{theorem}
\label{vern_contracte_thm}
Let $(X,d)$ be a compact vernation tree. Then, for all $\varepsilon>0$, $(\rG_\varepsilon(X),d_{\varepsilon,0})$ is a compact vernation tree such that $\diam(L)\geq\varepsilon$ for any $L\in\ell(\rG_\varepsilon(X))$. Moreover, $\rG_\varepsilon(X)\longrightarrow X$ for the Gromov--Hausdorff topology as $\varepsilon\to 0^+$.
\end{theorem}

\begin{proof}
First note that if $g$ is a geodesic on $X$ between two distinct points and if $s\in(0,1)$, then Definition~\ref{vernation_topo_def} yields that there is at most one $L\in\ell(X)$ such that $g^{-1}(L)$ is a neighborhood of $s$. By Fubini's theorem, (\ref{vern_contracte}) entails that $(d_{\varepsilon,\eta})_{0<\eta\leq 1}$ non-increasingly converges pointwise to $d_{\varepsilon,0}$ when $\eta\to 0^+$, and that $(d_{\varepsilon,0})_{\varepsilon>0}$ non-decreasingly converges pointwise to $d$ when $\varepsilon\to 0^+$. These functions are continuous on the compact space $X\times X$, so Dini's theorem implies that these convergences happen uniformly. Thus, $(X,d_{\varepsilon,\eta})\longrightarrow\rG_\varepsilon(X)$ as $\eta\to 0^+$ and $\rG_\varepsilon(X)\longrightarrow X$ as $\varepsilon\to 0^+$ for the GH topology. In particular, $\rG_\varepsilon(X)$ is a compact vernation tree by Theorem~\ref{closed_vern}.

We now focus on the diameters of the loops of $\rG_\varepsilon(X)$. We denote by $\mathsf{p}:X\to\rG_\varepsilon(X)$ the canonical projection. Let $L\in\ell(\rG_\varepsilon(X))$. Proposition~\ref{loop_iso} and Proposition~\ref{branching_at_loop} $(iv)$ yields that $(L,d_{\varepsilon,0})$ is isometric to $4\Delta\cdot\mathcal{C}$, where $2\Delta:=\diam(L)>0$. It follows that there exist $x_0,x_1,x_2,x_3\in X$ whose projections on $\rG_\varepsilon(X)$ are in $L$ and such that $d_{\varepsilon,0}(x_i,x_{i+1})=\Delta$ and $d_{\varepsilon,0}(x_i,x_{i+2})=2\Delta$ for all $i\in\mathbb{Z}/4\mathbb{Z}$. We choose a geodesic $g_1$ on $(X,d)$ from $x_0$ to $x_2$. Since $d_{\varepsilon,0}$ is continuous on $X^2$, the intermediate value theorem yields that there is $s\in(0,1)$ such that $d_{\varepsilon,0}(x_0,g_1(s))=\Delta$. We can assume that $d_{\varepsilon,0}(g_1(s),x_1)\geq \Delta$ without loss of generality because $2\Delta\leq d_{\varepsilon,0}(x_1,g_1(s))+d_{\varepsilon,0}(g_1(s),x_3)$. We then fix a geodesic $g_0$ on $(X,d)$ from $x_2$ to $x_1$ and a geodesic $g_2$ on $(X,d)$ from $x_1$ to $x_0$. 

Let us first assume that there exist $x_0'\in\Ima g_2\cap\Ima g_1$, $x_1'\in\Ima g_0\cap\Ima g_2$, and $x_2'\in\Ima g_1\cap \Ima g_0$ such that $\mathsf{p}(x_0')=\mathsf{p}(x_1')=\mathsf{p}(x_2')$. Then, the choice of $x_0,x_1,x_2,x_3$ and (\ref{distance_contracte2}) allow us to rewrite $\Delta+\Delta=2\Delta$ as
\[d_{\varepsilon,0}(x_0,x_1')+2d_{\varepsilon,0}(x_1',x_1)+d_{\varepsilon,0}(x_1',x_2)= d_{\varepsilon,0}(x_0,x_1)+d_{\varepsilon,0}(x_1,x_2)=d_{\varepsilon,0}(x_0,x_2)=d_{\varepsilon,0}(x_0,x_1')+d_{\varepsilon,0}(x_1',x_2).\]
Thus, we get $d_{\varepsilon,0}(x_1',x_1)=0$ and then $\Delta\leq d_{\varepsilon,0}(g_1(s),x_1)=d_{\varepsilon,0}(g_1(s),x_0')$ by triangular inequality. Then, using (\ref{distance_contracte2}) again gives that if $g_1^{-1}(x_0')\leq s$ (resp.~$g_1^{-1}(x_0')\geq s$) then $d_{\varepsilon,0}(x_0,x_0')=0$ (resp.~$d_{\varepsilon,0}(x_0',x_2)=0$). This contradicts $d_{\varepsilon,0}(x_0,x_1)=d_{\varepsilon,0}(x_1,x_2)=\Delta$. Therefore, we have shown that the projections of $\Ima g_2\cap\Ima g_1$, $\Ima g_0\cap\Ima g_2$, and $\Ima g_1\cap \Ima g_0$ do not intersect. In particular, we have a loop $\gamma$ on $X$ and three points $x_0',x_1',x_2'\in \Ima\gamma$ as given by Lemma~\ref{three_arcs_loop}. Moreover, $\mathsf{p}(x_0'),\mathsf{p}(x_1'),\mathsf{p}(x_2')$ are distinct. Let us write $L'=\Ima\gamma$ to lighten the notation.

We know from (\ref{distance_contracte1}) that $d_{\varepsilon,0}(x_0',x_1')=\I{\diam(L')\geq\varepsilon}d(x_0',x_1')$, which implies that $\diam(L')\geq\varepsilon$. As a result, (\ref{distance_contracte1}) entails that $(L',d)$ and $(\mathsf{p}(L'),d_{\varepsilon,0})$ are isometric. Now, recall that $(\rG_\varepsilon(X),d_{\varepsilon,0})$ is a compact vernation tree. By Proposition~\ref{branching_at_loop}, for all $i\in\{0,1,2\}$, we choose $x_i''\in X$ so that $\mathsf{p}(x_i'')$ is the unique point $u\in L$ such that $d_{\varepsilon,0}(\mathsf{p}(x_i'),u)=d_{\varepsilon,0}(\mathsf{p}(x_i'),L)$. We then successively apply (\ref{distance_contracte2}), (\ref{central_loop}) on $\rG_\varepsilon(X)$, and the triangular inequality of $d_{\varepsilon,0}$ to write
\[d_{\varepsilon,0}(x_0,x_2)=d_{\varepsilon,0}(x_0,x_0')+d_{\varepsilon,0}(x_0',x_2)=d_{\varepsilon,0}(x_0,x_0'')+2d_{\varepsilon,0}(x_0'',x_0')+d_{\varepsilon,0}(x_0'',x_2)\geq d_{\varepsilon,0}(x_0,x_2)+2d_{\varepsilon,0}(x_0'',x_0').\]
So we get $\mathsf{p}(x_0')=\mathsf{p}(x_0'')$. We similarly prove that $\mathsf{p}(x_1')=\mathsf{p}(x_1'')$ and $\mathsf{p}(x_2')=\mathsf{p}(x_2'')$. Thus, $L\cap \mathsf{p}(L')$ contains at least $3$ points, which ensures that $L=\mathsf{p}(L')$ and then $\diam(L)=\diam(L')\geq\varepsilon$.
\end{proof}

\subsection{Proof of Theorem~\ref{topo_conclusion}}

Here, we prove Theorem~\ref{topo_conclusion}. Recall from Sections~\ref{fst_properties_sec} and \ref{def_GH} the notions of pointed metric spaces and the definition of the pointed GH topology. We remind Definition~\ref{excursion} of excursions and $\Delta_s,x_s^t,\delta_s,\tilde{\delta}_s$ from (\ref{delta_x}), (\ref{circle_dist_s}), and (\ref{circle_shuffled_dist_s}). Finally, recall from Sections~\ref{intro_desc} and \ref{intro_shuffling} the tree $(\tree[f],\dtree[f])$, the looptrees $(\lop[f],\dl[f])$ and $(\lopt[f],\dlt[f])$, and the vernation trees $(\vern[f],\dv[f])$ and $(\vernt[f],\dvt[f])$ coded by an excursion $f$. We begin by restating in terms of the present work the classic analogous result about the coding of real trees by real-valued functions. Recall Definition~\ref{type_exc} of continuous excursions.

\begin{proposition}
\label{real_tree_topo_coded}
If $f$ is a continuous excursion, then $\tree[f]$ is a compact real tree. Conversely, if $(X,d)$ is a compact real tree, then for	all $a\in X$, there is a continuous excursion $f$ such that $(X,d,a)$ and $\tree[f]$ are pointed-isometric.
\end{proposition}

\begin{proof}
If $f$ is a continuous excursion then $\dtree$ is expressed by (\ref{classical_tree_distance}), and it is well-known that the quotient metric space $\tree[f]$ it induces is a real tree (see Le Gall~\cite{legall_trees} for example). Conversely, Duquesne~\cite[Lemma 4.2]{duquesne_coding} shows that any pointed compact real tree is coded by a continuous function $F:[0,M]\to[0,\infty)$ via the classic pseudo-distance given by (\ref{classical_tree_distance}), with some $M>0$. Although $F(0)=0$, $F$ is not an excursion in the sense of Definition~\ref{excursion} because it is not defined on $[-1,1]$ and because its last value does not have to be $0$. Nevertheless, the function $f:[-1,1]\to [0,\infty)$ defined by $f(t)=0$ when $t\leq 0$, $f(t)=F(2tM)$ when $0\leq t\leq 1/2$, and $f(t)=2F(M)(1-t)$ when $t\geq 1/2$ is a genuine continuous excursion. It is then straightforward to check that $\tree[f]$ is pointed-isometric to the pointed tree coded by $F$ via (\ref{classical_tree_distance}).
\end{proof}

Let us begin the proof of Theorem~\ref{topo_conclusion}. The first step is showing the following result.

\begin{theorem}
\label{vern_code/topo}
If $f$ is an excursion, then the spaces $\vern[f]$ and $\vernt[f]$ are compact vernation trees in the sense of Definition~\ref{vernation_topo_def}.
\end{theorem}

\noindent
Observe that the spaces $\tree[f]$, $\lop[f]$, and $\lopt[f]$ are also vernation trees, thanks to Theorems~\ref{intro_uni} and \ref{d_J}. To demonstrate Theorem~\ref{vern_code/topo}, we will use the following lemma that involves the gluing of pointed metric spaces presented in Definition~\ref{gluing}.

\begin{lemma}
\label{glu_vern}
Let $(X_0,d_0,a_0)$ and $(X_1,d_1,a_1)$ be two pointed compact vernation trees, and let $a\in X_0$. The gluing of $X_1$ on $X_0$ at $a$ is also a pointed compact vernation tree.
\end{lemma}

\begin{proof}
The space $X:=X_0\vee_a X_1$ is compact as a union of two compact spaces. Let $x,y\in X_i$ with $i\in\{0,1\}$. There is a geodesic on $X_i$ from $x$ to $y$, so seen as an $X$-valued function, it is also a geodesic on $X$ from $x$ to $y$. Let $x\in X_i$ and $y\in X_{1-i}$ with $i\in\{0,1\}$. We check that the concatenation of a geodesic on $X_i$ from $x$ to $a$ with a geodesic on $X_{1-i}$ from $a$ to $y$ induces a geodesic on $X$ from $x$ to $y$. Hence, the metric space $X$ is geodesic. Remark that $\partial X_0=\partial X_1=\{a\}$, so any path $\gamma$ on $X$ from a point of $X_0$ to a point of $X_1$ must hit $a$ by connectedness. Thus, if $\gamma$ is a loop on $X$ then $\Ima\gamma\subset X_0$ or $\Ima\gamma\subset X_1$. Indeed, if $\gamma(s_0)\in X_0\backslash\{a\}$ and $\gamma(s_1)\in X_1\backslash\{a\}$ with $s_0<s_1$ for example, then there exists $t\in(s_0,s_1)$ such that $\gamma(t)=a$ but there also exists $t'\in[0,1)\backslash(s_0,s_1)$ such that $\gamma(t')=a$. This contradicts the injectivity of $\gamma$ on $[0,1)$. Let $\gamma,\gamma'$ be two loops on $X$ based at the same point $x$. If $\Ima\gamma\subset X_i$ and $\Ima\gamma'\subset X_i$ with the same $i\in\{0,1\}$, then they can be seen as loops on the vernation tree $X_i$ based at the same point. Otherwise, it holds that $\{x\}\subset\Ima\gamma\cap\Ima\gamma'\subset X_0\cap X_1=\{a\}$. Eventually, the metric space $X_1\vee_a X_2$ is a vernation tree.
\end{proof}

\begin{proof}[Proof of Theorem~\ref{vern_code/topo}]
We already know from Propositions~\ref{d_continu} and \ref{dt_continu} that $\vern[f]$ and $\vernt[f]$ are compact. If $f$ is a continuous excursion, then $\vern[f]=\vernt[f]=\tree[f]$ by Theorems~\ref{intro_uni} and \ref{d_J}. Proposition~\ref{real_tree_topo_coded} thus asserts that they are real trees, and a fortiori vernation trees. Now, we suppose that the only jump of $f$ is at $0$, namely $\Delta_0=f(0)>0$ and $\Delta_t=0$ for all $t\in(0,1]$. The set $A(f)=\left\{x_0^t(f)\ :\ t\in[0,1]\text{ such that }x_0^t(f)<f(t)\right\}$ is countable because $f$ is c\`adl\`ag. Informally, $A(f)$ is the set of positions where non-trivial subspaces of $\vern[f]$ are attached to the loop associated with $0$. If $A(f)=\emptyset$ then $f$ is non-increasing on $[0,1]$, and we obtain from (\ref{dv_easy}) and (\ref{dlt_easy}) that $\dv(s,t)=2\delta_0\left(f(s),f(t)\right)$ and $\dvt(s,t)=2\tilde{\delta}_0\left(f(s),f(t)\right)$ for all $s,t\in[0,1]$. It follows that both $\vern[f]$ and $\vernt[f]$ are metric circles of perimeter $2\Delta_0$, so they are vernation trees by Proposition~\ref{equiv_vern}. If $x\in A(f)$, a judicious application of the branching property, as stated by Proposition~\ref{branching_property}, shows there are a continuous excursion $h$ and an excursion $g$ with its only jump at $0$ such that $A(g)=A(f)\backslash\{x\}$, such that $\vern[f]$ (resp.~$\vernt[f]$) is isometric to a gluing of $\vern[h]$ (resp.~$\vernt[h]$) on $\vern[g]$ (resp.~$\vernt[g]$). Thus, Lemma~\ref{glu_vern} allows us to prove by induction on $\#A(f)$ that if $f$ has its only jump at $0$ and if $A(f)$ is finite, then $\vern[f]$ and $\vernt[f]$ are vernation trees. If $A(f)$ is infinite, set
\[f_n(t)=\begin{cases}
			f(t)&\text{ if there exists } s\in[0,1]\text{ such that }x_0^s(f)=x_0^t(f)\leq f(s)-1/n,\\
			x_0^t(f)&\text{ otherwise, }
		 \end{cases}\]
for all $t\in[0,1]$ and $n\geq 1$. We easily check that $f_n$ is an excursion with a unique jump at $0$, that $|f_n(t)-f(t)|\leq 1/n$, and that $x_0^t(f_n)=x_0^t(f)$ for all $t\in[0,1]$. Then, (\ref{dl_easy}) and (\ref{dlt_easy}) become $\dl[f_n]=\dl[f]$ and $\dlt[f_n]=\dlt[f]$, and we get that $\dtree[f_n]\longrightarrow \dtree[f]$ uniformly on $[0,1]^2$ by applying Corollary~\ref{continu_particular}. Hence, Lemma~\ref{GHP_lemma} yields that $\vern[f_n]\longrightarrow\vern[f]$ and $\vernt[f_n]\longrightarrow\vernt[f]$ for the GH topology. Moreover, we observe that $A(f_n)=\{x_0^t(f)\, :\, t\in[0,1]\text{ such that }x_0^t(f)\leq f(t)-1/n\}$, so $A(f_n)$ is finite because $f$ is c\`adl\`ag, and so $\vern[f_n]$ and $\vernt[f_n]$ are vernation trees. Theorem~\ref{closed_vern} then implies that so are $\vern[f]$ and $\vernt[f]$.

To sum up, we have proved that if $f$ has at most one jump at $0$, then $\vern[f]$ and $\vernt[f]$ are vernation trees. Induction over the number of jumps of $f$ shows this still holds if $f$ has finitely many jumps. Indeed, using the branching property at the last jump of $f$ allows writing $\vern[f]$ (resp.~$\vernt[f]$) as a gluing of $\vern[h]$ (resp.~$\vernt[h]$) on $\vern[g]$ (resp.~$\vernt[g]$), where $g$ is an excursion with one less jump than $f$ and where $h$ is an excursion with a unique jump at $0$. Finally, we consider a general excursion $f$. We set $g_n=f-Jf+J^{1/n}f$ for all $n\geq1$, where $Jf$ and $J^{1/n}f$ are given by (\ref{Jeps_def}). By (\ref{x(Jf)}), all the jumps of $g_n$ are not lower than $1/n$ so there is only a finite number of them, which implies that $\vern[g_n]$ and $\vernt[g_n]$ are vernation trees. Furthermore, we know $J g_n=J^{1/n}f$ from Theorem~\ref{J_prop} $(iii)$, so Theorem~\ref{d_J} yields $\dtree[g_n]=\dtree[f]$, $\dl[g_n]=\dl[J^{1/n}f]$, and $\dlt[g_n]=\dlt[J^{1/n}f]$. Corollary~\ref{propo} and Proposition~\ref{cv_J} $(iii)$ give $\vern[g_n]\longrightarrow\vern[f]$ and $\vernt[g_n]\longrightarrow\vernt[f]$ for the GH topology. Theorem~\ref{closed_vern} concludes the proof.
\end{proof}

We already know from Theorem~\ref{closed_vern} that the space of compact vernation trees is closed for the GH topology. Hence, Theorem~\ref{vern_code/topo} entails the reverse implications claimed by the desired Theorem~\ref{topo_conclusion}. Now, let us denote by $\mathbb{V}'$ the set of isometry classes of compact metric spaces $(X,d)$ such that for any $a\in X$, there are two sequences $(f_n)$ and $(g_n)$ of excursions such that $(X,d,a)$ is the limit of $\vern[f_n]$ and of $\vernt[g_n]$ for the pointed GH topology. In the following, we show that any compact vernation tree $(X,d)$ is an element of $\mathbb{V}'$, which readily completes the proof of Theorem~\ref{topo_conclusion}. Let us already observe that $\mathbb{V}'$ is closed for the GH topology. We denote the diameter of $(X,d)$ by $\diam(X)=\sup_{x,y\in X}d(x,y)$. Recall from (\ref{set_loops_metric}) that $\ell(X)$ stands for the set of the images of loops on $X$.

\begin{proof}[End of the proof of Theorem~\ref{topo_conclusion}]
We first consider the case where $\ell(X)$ is finite. If $\ell(X)=\emptyset$ then $X$ is a real tree according to Proposition~\ref{equiv_tree} $(iii)$. Thus, if $\ell(X)=\emptyset$ then Proposition~\ref{real_tree_topo_coded} yields that $X\in\mathbb{V}'$. By induction, we assume that $\#\ell(X)\geq 1$ and that $Y\in\mathbb{V}'$ whenever $Y$ is a compact vernation tree with $\#\ell(Y)\leq\#\ell(X)-1$. Let $\gamma$ be a loop on $X$, so that $L=\Ima\gamma\in\ell(X)$, and let us use the notation of Proposition~\ref{branching_at_loop}. For any $u\in L$, the subspace $(X_u,d)$ is a compact vernation tree for all $u\in \Ima\gamma$. Plus, $L$ is not included in $X_u$, so $\#\ell(X_u)\leq \#\ell(X)-1$ and $X_u\in \mathbb{V}'$ by induction hypothesis. Moreover, $(\Ima \gamma,d)$ is a geodesic space that is homeomorphic to the metric circle with unit length $\mathcal{C}$. Thus, Proposition~\ref{loop_iso} asserts that $L$ is isometric to $2\Delta\cdot\mathcal{C}$ with some $\Delta>0$. In particular, $L\in \mathbb{V}'$ because the (shuffled) vernation tree coded by the excursion $f:t\in[-1,1]\longmapsto \I{t\geq 0}\Delta(1-t)$ is pointed-isometric to $2\Delta\cdot\mathcal{C}$, for any choice of the root. Then, (\ref{central_loop}) and the compactness of $X$ ensure that there is only a finite number of $u\in L$ such that $\diam(X_u)\geq \varepsilon$ for any $\varepsilon>0$. It follows that there is a sequence of points $u_k\in L$ such that $\diam(X_{u_k})\longrightarrow 0$ and $X=\bigcup_{k\geq 1} (X_{u_k}\cup L)$. In particular, if we set $X_n=\bigcup_{k=1}^n(X_{u_k}\cup L)$ for all $n\geq 1$, then (\ref{central_loop}) yields that $X_n\longrightarrow X$ for the GH topology. Furthermore, this same identity entails that $X_n$ can be constructed with a finite number of successive gluings only involving elements of $\mathbb{V}'$. With the help of Proposition~\ref{branching_property} and of the bounds (\ref{maj_dl}), (\ref{maj_dtree}), and (\ref{maj_dlt}), it is not hard to show that if $g,h$ are excursions, then the gluing of $\vern[h]$ on $\vern[g]$ at any point of $\vern[g]$ is pointed-isometric to another vernation tree coded by some excursion $f$. The same result also holds with shuffled vernation trees. Moreover, we observe that if $(Y_n)$ and $(Z_n)$ respectively converge to $Y$ and $Z$ for the pointed GH topology, then for all $a\in Y$, there is a sequence of points $a_n\in Y_n$ such that $(Y_n\vee_{a_n} Z_n)$ converges to $Y\vee_a Z$ for the pointed GH topology. Indeed, one can simply choose $a_n$ such that $(a_n,a)$ is in a correspondence between $Y_n$ and $Y$ with small distortion. Hence, $\mathbb{V}'$ is stable by gluing, so it follows that $X_n\in \mathbb{V}'$ for all $n\geq1$, and so $X\in \mathbb{V}'$ because $\mathbb{V}'$ is closed for the GH topology.

Now, we consider the general case where $\ell(X)$ may be infinite. By Theorem~\ref{vern_contracte_thm}, $X$ is the limit for the GH topology of a sequence of compact vernation trees $\rG_{1/n}(X),n\geq 1$ such that $\diam(L_{1/n})\geq 1/n$ for all $L_{1/n}\in\ell(\rG_{1/n}(X))$. In particular, Proposition~\ref{number_loops} entails that $\ell(\rG_{1/n}(X))$ is finite, which implies that $\rG_{1/n}(X)\in\mathbb{V}'$ thanks to the above case. Since $\mathbb{V}'$ is closed for the GH topology, we eventually get $X\in\mathbb{V}'$ in general.
\end{proof}

The reader might wonder if any compact vernation tree can be coded by an excursion. Let us informally explain why the answer is no for $f\mapsto \vern[f]$. For all $n\geq 1$ and for all $u\in\{+1,-1\}^n$, we give ourselves $\mathcal{C}_u$ a pointed metric circle of length $2/n$, seen as $[-1/n,1/n]$ with the identification $1/n=-1/n$ and whose root is $0$. The metric space $X$ is constructed by gluing each $\mathcal{C}_{(\epsilon_1,\ldots,\epsilon_{n+1})}$ on $\mathcal{C}_{(\epsilon_1,\ldots,\epsilon_{n})}$ at the point $2\epsilon_{n+1}/(n+1)^2$, and then by taking the closure. One can prove $X$ is compact because $\diam(X)\leq 2+4\sum 1/n^2<\infty$, and that it is a vernation tree thanks to Lemma~\ref{glu_vern} and Theorem~\ref{closed_vern}. However, if there was an excursion $f$ such that $\vern[f]$ was isometric to $X$, there would be a point $t\in[0,1]$ whose projection would be on $\mathcal{C}_{(1)}$ or on $\mathcal{C}_{(-1)}$. The reader should now be able to convince oneself that there would be $t\in[0,1]$ and distinct $r_1,\ldots,r_n\in[0,1]$ such that $x_{r_i}^t=1/i-1/(i+1)^2$ for all $1\leq i\leq n$, for all $n\geq 1$. Recalling the expression (\ref{J_def}) and Theorem~\ref{intro_uni}, this leads to $\infty=\sup Jf\leq\sup f$ and contradicts that $f$ is c\`adl\`ag. The issue here is the same as what prevents $\dv$ from enjoying a functional continuity, namely that high variations of excursions may code short distances on the vernation tree. This problem disappears for $\dvt$ thanks to (\ref{11}), so we believe that any compact vernation tree is isometric to a shuffled vernation tree $\vernt[f]$ coded by an excursion $f$. Nevertheless, we were not able to demonstrate it. We have recently realized that Blanc-Renaudie~\cite{blanc-renaudie} managed to construct a contour path continuously exploring an inhomogeneous continuum random looptree in finite time. Adapting his construction and the coding of any real tree by a continuous excursion of Duquesne~\cite{duquesne_coding} may result in a proof of our conjecture.

\section{Probabilistic applications}
\label{application_proba}

\paragraph{Words and plane trees.}
\label{word_trees}
Before presenting our three applications, we recall the formalism of plane trees: see Le Gall~\cite{legall_trees} for example. Let $\mathbb{N}^*=\{1,2,3,\ldots\}$ be the set of positive integers and let\[\mathbb{U}=\bigcup_{n=0}^\infty (\mathbb{N}^*)^n,\quad \text{with the convention }\quad (\mathbb{N}^*)^0=\{\varnothing\}.\]The set $\mathbb{U}$ is totally ordered by the lexicographic order, denoted by $\leq$. An element of $\mathbb{U}\backslash\{\varnothing\}$ is a finite sequence of positive integers $u=(u_1,\cdots, u_m)$; we set $|u|=m$ the generation or height of $u$, and $\overleftarrow{u}=(u_1,\cdots,u_{m-1})$ the parent of $u$. We also set $|\varnothing|=0$. If $u,v\in\mathbb{U}$, we write $u*v\in\mathbb{U}$ for the concatenation of $u$ and $v$, we say $u$ is an ancestor of $u*v$, and we write $u\preceq u*v$. If $j\in\mathbb{N}^*$, we say that $u*(j)$ is a child of $u$. A \emph{plane tree} $\tau$ is a subset of $\mathbb{U}$ such that:
\begin{enumerate}
\item[$(a)$] $\varnothing\in\tau$ and $\tau$ is finite;
\item[$(b)$] if $v\in\tau$ and $v\neq\varnothing$, then $\overleftarrow{v}\in\tau$;
\item[$(c)$] for all $u\in\tau$, there is a nonnegative integer $k_u(\tau)$ (called the number of children of $u$ in $\tau$) such that for every $j\in\mathbb{N}^*$, $u*(j)\in\tau$ if and only if $1\leq j\leq k_u(\tau)$.
\end{enumerate}
We denote the total progeny of $\tau$, which is the total number of vertices of $\tau$, by $\#\tau$ and the height of $\tau$, which is the maximal generation, by $|\tau|=\max_{u\in\tau}|u|$. For $u\in\tau$, its subtree stemming from $u$ is $\theta_u\tau=\{v\in\mathbb{U}\, :\, u*v\in\tau\}$, which is also a plane tree. Furthermore, we denote by $\varnothing=u(0)<u(1)<\cdots<u(\#\tau-1)$ the vertices of $\tau$ listed in lexicographic order. We call $\left(u(i)\right)_{0\leq i\leq\#\tau-1}$ the \emph{depth-first exploration} of $\tau$. The \emph{exploration by contour} $c=\left(c(i)\right)_{0\leq i\leq 2(\#\tau-1)}$ gives another way to explore the tree. Informally, $c$ starts at the root and continuously visits the whole tree from the left to the right. Maybe more precisely, $c(0)=\varnothing$, and $c(i+1)$ is the $\leq$-smallest child of $c(i)$ that has not been visited yet if one exists or $c(i+1)$ is the parent of $c(i)$ otherwise. This process crosses each edge two times, one time upwards and one time downwards, so the needed time to complete the exploration is indeed $2(\#\tau-1)$ and $c(2\#\tau-2)=\varnothing$. In order to link those two explorations, we define $\xi(i)$ for all $0\leq i\leq 2\#\tau-2$ as the largest integer $l$ such that $u(l)$ has been visited by $\left(c(j)\right)_{0\leq j\leq i}$. Moreover, the exploration processes induce finite sequences of integers that characterize the plane tree. Namely, the \emph{height process} $H(\tau)=\left(H_i(\tau)\right)_{0\leq i\leq\#\tau-1}$ is defined by $H_i(\tau)=|u(i)|$, and the \emph{contour process} $C(\tau)=\left(C_i(\tau)\right)_{0\leq i\leq2\#\tau-2}$ is defined by $C_i(\tau)=|c(i)|$. General arguments (see e.g.~Le Gall~\cite[Section 1.6]{legall_trees}) draw relations between the height process and the contour process via the two bounds
\begin{align}
\label{depth-first_to_contour}
\max_{0\leq i\leq 2\#\tau-2}\left|\xi(i)-\frac{i}{2}\right|&\leq \frac{|\tau|}{2}+1,\\
\label{height_vs_contour}
\max_{0\leq i\leq 2\#\tau-2}\left|C_i(\tau)-H_{\xi(i)}(\tau)\right|&\leq 1+\max_{0\leq j\leq \#\tau-1}\left|H_{i+1}(\tau)-H_i(\tau)\right|,
\end{align}
with the convention $H_{\#\tau}(\tau)=0$. The process $C(\tau)$ is extended to the real interval $[0,2\#\tau-2]$ by specifying $C(\tau)$ is affine on each interval $[i,i+1]$ for $0\leq i\leq 2\#\tau-3$. Then, $\left(C_{2t(\#\tau-1)}(\tau)\right)_{t\in[0,1]}$ is a continuous excursion in the sense of Definition~\ref{type_exc} and it codes via (\ref{classical_tree_distance}) or (\ref{dtree_easy}) the pointed weighted metric space spanned by $\tau$. More precisely, the tree $T$ coded by this excursion is obtained by replacing each edge $(\overleftarrow{u},u)$ of $\tau$ with a line segment of length $1$ and by endowing it with the normalized sum of the Lebesgue measures of those segments. With a slight abuse, $\tau$ can be seen as a finite subset of $T$, so that $\varnothing$ is the distinguished point of $T$ and that the graph distance of $\tau$ is inherited from the metric on $T$.

\subsection{Metric asymptotics for uniform random mappings}
\label{random_mapping_sec}

We write $[n]=\{1,\ldots,n\}$ for $n\geq 1$. We are interested in the asymptotic behavior of the patterns of random uniform mappings on $[n]$. We follow the presentation of Aldous, Miermont \& Pitman~\cite{p-mapping}. A mapping $m:[n]\longrightarrow[n]$ may be interpreted as a digraph with set of vertices $[n]$ and with oriented edges $i\rightarrow m(i)$, thus allowing edges of the form $i\rightarrow i$. It naturally induces a simple graph (that may have edges of the form $(i,i)$), that we will denote by $\mathcal{G}(m)$. We define $m^0(i)=i$ and $m^{k+1}(i)=m(m^k(i))$ the $(k+1)$-fold iteration of $m$ on $i\in[n]$, for any $k\geq0$. We say that $i$ is a cyclic point of $m$ if $m^k(i)=i$ for some $k\geq1$ and we denote by $\Gamma(m)$ the set of cyclic points of $m$. If $\gamma\in\Gamma(m)$, we say that the finite set $\{m^k(\gamma)\ :\ k\geq1\}\subset\Gamma(m)$ is a cycle of $m$ and we write \[\mathcal{T}_\gamma(m)=\{\gamma\}\cup\left\{i\in[n]\backslash\Gamma(m)\ :\ \exists k\geq 0,\ m^k(i)=\gamma\right\}\] for the tree component of the mapping graph with root $\gamma$. By independently putting each set of children of the vertices of $\mathcal{T}_\gamma(m)$ into uniform random order, the graph $\mathcal{T}_\gamma(m)$ naturally corresponds to a plane tree $T_\gamma(m)$. The tree components are bundled by the disjoint cycles $\Gamma_j(m)$, for $1\leq j\leq k(m)$, to form the basins of attraction $\mathcal{B}_j(m)$ of $m$, that are the connected components of the graph $\mathcal{G}(m)$, namely \[\mathcal{B}_j(m)=\bigsqcup_{\gamma\in\Gamma_j(m)}\mathcal{T}_\gamma(m).\] While the $\Gamma_j(m)$ partition $\Gamma(m)$, the $\mathcal{B}_j(m)$ partition $[n]$. Let us explain how to index the cycles $\Gamma_j(m)$ while ordering $\Gamma(m)$. We consider a random sample $(X_k)_{k\geq 1}$ of independent uniform random points of $[n]$, independent from the random orders on the tree components. We index the basins of attraction $\mathcal{B}_j(m)$ for $1\leq j\leq k(m)$ by the order of their first appearances in the sample. For example, $\mathcal{B}_1(m)$ is the basin that contains $X_1$ and $\mathcal{B}_2(m)$ is the basin which contains the first $X_k\notin \mathcal{B}_1(m)$ if it exists, and so on. When $\gamma_i\in \Gamma_{i}(m)$ and $\gamma_j\in \Gamma_j(m)$ with $1\leq i<j\leq k(m)$, we write $\gamma_i\leq_m \gamma_j$. To extend $\leq_m$ into a total order on $\Gamma(m)$, we only need to order each cycle $\Gamma_j(m)$. Let $\gamma_j^\bullet\in \Gamma_j(m)$ be the root of the tree component that contains the first $X_k\in \mathcal{B}_j(m)$. We then specify \[m(\gamma_j^\bullet)\leq_m m^2(\gamma_j^\bullet)\leq_m \ldots\leq_m m^{\#\Gamma_j(m)-1}(\gamma_j^\bullet)\leq_m \gamma_j^\bullet.\]
We point out that $\leq_m$ may even be extended to $[n]$ using the lexicographic orders on the tree components. See Figure~\ref{mapping_fig} for a complete example. When the mapping $M$ is random, that whole ordering procedure is done independently from $M$. Let us write $\gamma(1)<_m \gamma(2)<_m \ldots<_m\gamma(\#\Gamma(m))$ for the elements of $\Gamma(m)$ listed in the $\leq_m$-increasing order. Finally, let us define the metric object associated with the mapping we are interested in. For all $1\leq j\leq k(m)$, we denote by $B_j(m)$ the pointed weighted metric space $\mathcal{B}_j(m)$ endowed with its graph distance, with its uniform probability measure, and with $\gamma_j^\bullet$ as its distinguished point. For all $j>k(m)$, we set $B_j(m)=\partial$ where $\partial$ is the pointed weighted metric space with a unique point. We define $G(m)=\left(\I{j\leq k(m)}\#\mathcal{B}_j(m)/n,B_j(m)\right)_{j\geq 1}$. One can observe that $G(m)$ determines the mapping $m$ up to the labels of the vertices and the orientations of the cycles. Hence, it is a good metric interpretation of the mapping pattern.

Given the random order on the graph, one can construct several processes associated with the mapping. We define the height process $H(m)=\left(H_i(m)\right)_{0\leq i\leq n}$ of the mapping as the concatenation of the height processes of the tree components $\left(H_i(T_{\gamma(j)}(m))\ ;\ 0\leq i\leq \# T_{\gamma(j)}(m)-1\right)$ for $1\leq j\leq\#\Gamma(m)$, in that order, followed by a last zero term. Similarly, the contour process $C(m)=\left(C_i(m)\right)_{0\leq i\leq 2n}$ of the mapping is the concatenation of the sequences $\left(C_i(T_{\gamma(j)}(m))\ ;\ 0\leq i\leq 2\# T_{\gamma(j)}(m)-1\right)$ for $1\leq j\leq\#\Gamma(m)$, in that order and with the convention $C_{2\#T_{\gamma(j)}(m)-1}(T_{\gamma(j)}(m))=0$, followed by a last zero term. We stress that we have inserted a zero term between the contour processes of successive tree components. Alternatively, $H(m)$ and $C(m)$ can be expressed as variations of the height and contour processes of a single plane tree. Indeed, let $T(m)$ be the unique plane tree such that $k_\varnothing(T(m))=\#\Gamma(m)$ and $\theta_{(i)}T(m)=T_{\gamma(i)}(m)$ for all $1\leq i\leq\#\Gamma(m)$, so that $\#T(m)=n+1$. We let the reader observe that $H_n(m)=C_{2n}(m)=0$, \[H_i(m)=H_{i+1}(T(m))-1\quad\text{ and }\quad C_{i'}(m)=\max\left(0,C_{i'+1}(T(m))-1\right)\]
for all $0\leq i\leq n-1$ and for all $0\leq i'\leq 2n-1$. We keep track of the number of cyclic points via two processes $\ell(m)=\left(\ell_i(m)\right)_{0\leq i\leq n}$ and $\ell'(m)=\left(\ell_i'(m)\right)_{0\leq i\leq 2n}$ defined by
\[\ell_i(m)=\#\left\{1\leq j\leq i\ :\ H_i(m)=0\right\}\quad\text{ and }\quad\ell_{i'}'(m)=\frac{1}{2}\#\left\{1\leq j\leq i'\ :\ C_i(m)=C_{i-1}(m)=0\right\}\]
for all $0\leq i\leq n$ and for all $0\leq i\leq 2n$. We extend the processes $C(m)$ and $\ell'(m)$ to the real interval $[0,2n]$ by specifying they are affine on each interval $[i,i+1]$ for $0\leq i\leq 2n-1$. Moreover, we set $Z_0(m)=0$ and $Z_j(m)=Z_{j-1}(m)+\I{j\leq k(m)}\#\mathcal{B}_j(m)$ for all $j\geq 1$. These marks delimit the intervals corresponding to the basins of attraction. Observe $C_{2Z_j(m)}(m)=0$ and $\ell_{2Z_j(m)}'(m)-\ell_{2Z_{j-1}(m)}'(m)=\#\Gamma_j(m)$ for all $1\leq j\leq k(m)$. See Figure~\ref{mapping_fig} for an example.

\begin{figure}
\begin{center}
\begin{tikzpicture}[baseline={(current bounding box.center)}, bulle/.style={shape=circle,draw=black,minimum size=21pt}, Caption/.style={minimum size=21pt},edge from parent/.style={draw,latex-}, every path/.style={draw,-latex},
	level distance=15mm,level 1/.style={sibling distance=12mm},
    level 2/.style={sibling distance=10mm}]
\node[bulle,draw=orange] (A) at (0,0) {$4$} [grow=up]
child {node[bulle]{$12$}};
\node[bulle,draw=orange] (B) at (2,0) {$14$};
\node[bulle,draw=mystery] (C) at (4,0) {$7$} [grow=up]
child[draw=green] {node[bulle,draw=green]{$13$}}
child[draw=green] {node[bulle,draw=green]{$5$}
	child[draw=green] {node[bulle,draw=green]{$8$}}
	child[draw=green] {node[bulle,draw=green]{$11$}}
	};
\node[bulle] (D) at (8,0) {$3$} [grow=up]
child {node[bulle]{$2$}}
child {node[bulle]{$6$}}
child {node[bulle]{$9$}};
\node[bulle] (E) at (12,0) {$10$};
\node[bulle] (F) at (14,0) {$1$};

\path [->,color=orange] (A) edge (B);
\path [->,color=orange] (B) edge (C);
\path [->,color=orange] (C) edge [bend left=30] (A);
\path [->] (D) edge[loop below, in=-150, out=-30,looseness=5] (D);
\path [->] (E) edge (F);
\path [->] (F) edge [bend left=50] (E);

\draw[draw=black,dashed] (-0.7,-0.9) rectangle ++(6,4.5);

\node at (4.85,0) {$=\gamma_1^\bullet$};
\node at (8.85,0) {$=\gamma_2^\bullet$};
\node at (14.85,0) {$=\gamma_3^\bullet$};
\node at (4.8,3) {$=X_1$};
\node at (5.9,1.5) {$X_2=$};
\node at (0.9,1.5) {$=X_3$};
\node at (14.85,0.5) {$=X_4$};
\end{tikzpicture}

\begin{tikzpicture}[line cap=round,line join=round,>=triangle 45,x=1cm,y=2cm,scale=0.55]
\begin{axis}[x=1cm,y=2cm,
axis lines=middle,
xmin=0,
xmax=28.2,
ymin=-3.2,
ymax=2.2,
xtick={0,1,...,29},
ytick={-3,-2.5,-2,-1.5,-1,-0.5,0,1,2},
xticklabel=\empty,clip=false,]
\clip(-2,-3.5) rectangle (28.5,3);
\draw [line width=1.2pt,color=qqqqff] (0,0)-- (1,1);
\draw [line width=1.2pt,color=qqqqff] (1,1)-- (2,0);
\draw [line width=1.2pt,color=qqqqff] (2,0)-- (3,0);
\draw [line width=1.2pt,color=qqqqff] (3,0)-- (4,0);
\draw [line width=1.2pt,color=qqqqff] (4,0)-- (5,0);
\draw [line width=1.2pt,color=qqqqff] (5,0)-- (6,0);
\draw [line width=1.2pt,color=qqqqff] (6,0)-- (7,1);
\draw [line width=1.2pt,color=qqqqff] (7,1)-- (8,2);
\draw [line width=1.2pt,color=qqqqff] (8,2)-- (9,1);
\draw [line width=1.2pt,color=qqqqff] (9,1)-- (10,2);
\draw [line width=1.2pt,color=qqqqff] (10,2)-- (11,1);
\draw [line width=1.2pt,color=qqqqff] (11,1)-- (12,0);
\draw [line width=1.2pt,color=qqqqff] (12,0)-- (13,1);
\draw [line width=1.2pt,color=qqqqff] (13,1)-- (14,0);
\draw [line width=1.2pt,color=qqqqff] (14,0)-- (15,0);
\draw [line width=1.2pt,color=qqqqff] (15,0)-- (16,0);
\draw [line width=1.2pt,color=qqqqff] (16,0)-- (17,1);
\draw [line width=1.2pt,color=qqqqff] (17,1)-- (18,0);
\draw [line width=1.2pt,color=qqqqff] (18,0)-- (19,1);
\draw [line width=1.2pt,color=qqqqff] (19,1)-- (20,0);
\draw [line width=1.2pt,color=qqqqff] (20,0)-- (21,1);
\draw [line width=1.2pt,color=qqqqff] (21,1)-- (22,0);
\draw [line width=1.2pt,color=qqqqff] (22,0)-- (23,0);
\draw [line width=1.2pt,color=qqqqff] (23,0)-- (24,0);
\draw [line width=1.2pt,color=qqqqff] (24,0)-- (25,0);
\draw [line width=1.2pt,color=qqqqff] (25,0)-- (26,0);
\draw [line width=1.2pt,color=qqqqff] (26,0)-- (27,0);
\draw [line width=1.2pt,color=qqqqff] (27,0)-- (28,0);

\draw [line width=1.2pt,color=ffqqqq] (0,0)-- (1,0);
\draw [line width=1.2pt,color=ffqqqq] (0,0)-- (1,0);
\draw [line width=1.2pt,color=ffqqqq] (1,0)-- (2,0);
\draw [line width=1.2pt,color=ffqqqq] (2,0)-- (3,-0.25);
\draw [line width=1.2pt,color=ffqqqq] (3,-0.25)-- (4,-0.5);
\draw [line width=1.2pt,color=ffqqqq] (4,-0.5)-- (5,-0.75);
\draw [line width=1.2pt,color=ffqqqq] (5,-0.75)-- (6,-1);
\draw [line width=1.2pt,color=ffqqqq] (6,-1)-- (7,-1);
\draw [line width=1.2pt,color=ffqqqq] (7,-1)-- (8,-1);
\draw [line width=1.2pt,color=ffqqqq] (8,-1)-- (9,-1);
\draw [line width=1.2pt,color=ffqqqq] (9,-1)-- (10,-1);
\draw [line width=1.2pt,color=ffqqqq] (10,-1)-- (11,-1);
\draw [line width=1.2pt,color=ffqqqq] (11,-1)-- (12,-1);
\draw [line width=1.2pt,color=ffqqqq] (12,-1)-- (13,-1);
\draw [line width=1.2pt,color=ffqqqq] (13,-1)-- (14,-1);
\draw [line width=1.2pt,color=ffqqqq] (14,-1)-- (15,-1.25);
\draw [line width=1.2pt,color=ffqqqq] (15,-1.25)-- (16,-1.5);
\draw [line width=1.2pt,color=ffqqqq] (16,-1.5)-- (17,-1.5);
\draw [line width=1.2pt,color=ffqqqq] (17,-1.5)-- (18,-1.5);
\draw [line width=1.2pt,color=ffqqqq] (18,-1.5)-- (19,-1.5);
\draw [line width=1.2pt,color=ffqqqq] (19,-1.5)-- (20,-1.5);
\draw [line width=1.2pt,color=ffqqqq] (20,-1.5)-- (21,-1.5);
\draw [line width=1.2pt,color=ffqqqq] (21,-1.5)-- (22,-1.5);
\draw [line width=1.2pt,color=ffqqqq] (22,-1.5)-- (23,-1.75);
\draw [line width=1.2pt,color=ffqqqq] (23,-1.75)-- (24,-2);
\draw [line width=1.2pt,color=ffqqqq] (24,-2)-- (25,-2.25);
\draw [line width=1.2pt,color=ffqqqq] (25,-2.25)-- (26,-2.5);
\draw [line width=1.2pt,color=ffqqqq] (26,-2.5)-- (27,-2.75);
\draw [line width=1.2pt,color=ffqqqq] (27,-2.75)-- (28,-3);

\begin{scriptsize}
\draw [fill=ffqqff] (0,0) circle [radius=3pt];
\draw [fill=ffqqff] (2,0) circle [radius=3pt];

\draw [fill=ududff] (1,1) circle [radius=3pt];
\draw [fill=ududff] (3,0) circle [radius=3pt];
\draw [fill=ududff] (4,0) circle [radius=3pt];
\draw [fill=ududff] (5,0) circle [radius=3pt];
\draw [fill=ududff] (6,0) circle [radius=3pt];
\draw [fill=ududff] (7,1) circle [radius=3pt];
\draw [fill=ududff] (8,2) circle [radius=3pt];
\draw [fill=ududff] (9,1) circle [radius=3pt];
\draw [fill=ududff] (10,2) circle [radius=3pt];
\draw [fill=ududff] (11,1) circle [radius=3pt];
\draw [fill=ududff] (12,0) circle [radius=3pt];
\draw [fill=ududff] (13,1) circle [radius=3pt];
\draw [fill=ududff] (14,0) circle [radius=3pt];
\draw [fill=ududff] (15,0) circle [radius=3pt];
\draw [fill=ududff] (16,0) circle [radius=3pt];
\draw [fill=ududff] (17,1) circle [radius=3pt];
\draw [fill=ududff] (18,0) circle [radius=3pt];
\draw [fill=ududff] (19,1) circle [radius=3pt];
\draw [fill=ududff] (20,0) circle [radius=3pt];
\draw [fill=ududff] (21,1) circle [radius=3pt];
\draw [fill=ududff] (22,0) circle [radius=3pt];
\draw [fill=ududff] (23,0) circle [radius=3pt];
\draw [fill=ududff] (24,0) circle [radius=3pt];
\draw [fill=ududff] (25,0) circle [radius=3pt];
\draw [fill=ududff] (26,0) circle [radius=3pt];
\draw [fill=ududff] (27,0) circle [radius=3pt];
\draw [fill=ududff] (28,0) circle [radius=3pt];

\draw [fill=ffqqqq] (1,0) circle [radius=3pt];
\draw [fill=ffudud] (3,-0.25) circle [radius=3pt];
\draw [fill=ffudud] (4,-0.5) circle [radius=3pt];
\draw [fill=ffudud] (5,-0.75) circle [radius=3pt];
\draw [fill=ffudud] (6,-1) circle [radius=3pt];
\draw [fill=ffudud] (7,-1) circle [radius=3pt];
\draw [fill=ffudud] (8,-1) circle [radius=3pt];
\draw [fill=ffudud] (9,-1) circle [radius=3pt];
\draw [fill=ffudud] (10,-1) circle [radius=3pt];
\draw [fill=ffudud] (11,-1) circle [radius=3pt];
\draw [fill=ffudud] (12,-1) circle [radius=3pt];
\draw [fill=ffudud] (13,-1) circle [radius=3pt];
\draw [fill=ffudud] (14,-1) circle [radius=3pt];
\draw [fill=ffudud] (15,-1.25) circle [radius=3pt];
\draw [fill=ffudud] (16,-1.5) circle [radius=3pt];
\draw [fill=ffudud] (17,-1.5) circle [radius=3pt];
\draw [fill=ffudud] (18,-1.5) circle [radius=3pt];
\draw [fill=ffudud] (19,-1.5) circle [radius=3pt];
\draw [fill=ffudud] (20,-1.5) circle [radius=3pt];
\draw [fill=ffudud] (21,-1.5) circle [radius=3pt];
\draw [fill=ffudud] (22,-1.5) circle [radius=3pt];
\draw [fill=ffudud] (23,-1.75) circle [radius=3pt];
\draw [fill=ffudud] (24,-2) circle [radius=3pt];
\draw [fill=ffudud] (25,-2.25) circle [radius=3pt];
\draw [fill=ffudud] (26,-2.5) circle [radius=3pt];
\draw [fill=ffudud] (27,-2.75) circle [radius=3pt];
\draw [fill=ffudud] (28,-3) circle [radius=3pt];

\draw [color=black] (-0.75,0) node {\huge $2 Z_0$};
\draw [color=black] (16,-0.35) node {\huge $2 Z_1$};
\draw [color=black] (24,-0.35) node {\huge $2 Z_2$};
\draw [color=black] (28,-0.35) node {\huge $2 Z_3$};

\draw [color=ududff] (-0.9,2.3) node {\huge $C(m)$};
\draw [color=ffudud] (-1.2,-3.3) node {\huge $-\tfrac{1}{2}\ell'(m)$};
\end{scriptsize}
\end{axis}
\end{tikzpicture}
\end{center}
\caption{\textit{Top :} A mapping $m$ on $[14]$ represented as a plane forest so that its random order coincides with the depth-first order. The positions of $X_1,X_2,X_3,X_4$ are shown. The tree $\mathcal{T}_7(m)$ is in green and the cycle $\Gamma_1(m)$ is in orange. The basin $\mathcal{B}_1(m)$ is contained in the dashed box. \textit{Bottom :} The processes $C(m)$, $\frac{1}{2}\ell'(M)$, $2Z(m)$ associated with the mapping $m$.}
\label{mapping_fig}
\end{figure}

Now, let $B^{|\mathsf{br}|}$ be the standard reflected Brownian bridge on $[0,1]$ and let $L$ be half its local time at $0$, which is normalized to be the density of the occupation measure at $0$ of the reflected Brownian bridge, so that for any $t\in[0,1]$, it verifies the convergence in probability\[\frac{1}{2\varepsilon}\int_0^t\I{\{B_s^{|\mathsf{br}|}\leq\varepsilon\}}\dd s\overset{\P}{\underset{\varepsilon\rightarrow 0^+}{\longrightarrow}}L_t.\]Let $(U_j)_{j\geq 1}$ be a sequence of independent uniform random variables on $[0,1]$. We define a sequence of random points $(D_j)_{j\geq 0}$ by setting $D_0=0$ and $D_j=\inf\{t\geq D_{j-1}+U_j(1-D_{j-1})\ :\ B_t^{|\mathsf{br}|}=0\}$ for all $j\geq 1$. For all $n\geq 1$, let $M_n$ be a uniform random mapping on $[n]$ and consider the associated processes $H(M_n)$, $C(M_n)$, $\ell(M_n)$, $\ell'(M_n)$, and $\left(Z_j(M_n)\right)_{j\geq 1}$. Aldous, Miermont \& Pitman~\cite[Theorem 1]{p-mapping} proved that $2B^{|\mathsf{br}|}$, $L$, and $(D_j)_{j\geq 1}$ are the respective scaling limits of $H(M_n)$, $\ell(M_n)$, and $\left(Z_j(M_n)\right)_{j\geq 1}$ for the uniform topology. Since $B^{|\mathsf{br}|}$ and $L$ are almost surely continuous, the estimates (\ref{depth-first_to_contour}) and (\ref{height_vs_contour}) yield the following corollary.

\begin{corollary}
\label{mapping_processes}
The marks $\left(Z_j(M_n)/n\right)_{j\geq 1}$ converge in distribution to the sequence $(D_j)_{j\geq 1}$. Jointly with that convergence, the following convergence holds in distribution for the uniform topology on $[0,1]^2$:
\[\left(\frac{1}{\sqrt{n}}C_{2nt}(M_n),\frac{1}{\sqrt{n}}\ell_{2nt}'(M_n)\right)_{t\in[0,1]}\convd \left(2B_t^{|\mathsf{br}|},L_t\right)_{t\in[0,1]}.\]
\end{corollary}

Recall from Section~\ref{def_GH} the definition of the space $\mathbb{K}_{\mathrm{w}}^\bullet$ of GHP-isometry classes of pointed weighted compact metric spaces endowed with the pointed Gromov--Hausdorff--Prokhorov distance $\mathtt{d}_{\mathrm{GHP}}^\bullet$. When $(X,d)$ is a compact metric space, we write $\diam(X)=\sup_{x,y\in X}d(x,y)$. The space \[\mathbb{S}=\left\{\left(\alpha_j,X_j\right)_{j\geq 1}\ :\ \alpha_j\geq 0,\sum_{j\geq 1}\alpha_j=1,X_j\in\mathbb{K}_{\mathrm{w}}^\bullet,\diam(X_j)\longrightarrow0\right\}\]is made separable and complete by the uniform distance \[\mathtt{d}_{\mathbb{S}}\left((\alpha_j,X_j)_{j\geq 1},(\beta_j,Y_j)_{j\geq 1}\right)=\sup_{j\geq 1}\max\left(|\alpha_j-\beta_j|,\mathtt{d}_{\mathrm{GHP}}^\bullet(X_j,Y_j)\right).\]We write $\beta\cdot \left(\alpha_j,X_j\right)_{j\geq 1}=\left(\alpha_j,\beta\cdot X_j\right)_{j\geq 1}$ when $\beta>0$. Up to the small abuse of identifying a pointed weighted compact metric space with its class in $\mathbb{K}_{\mathrm{w}}^\bullet$, it holds that $G(m)\in\mathbb{S}$ for any mapping $m$. Let us write $Z_j^n=Z_j(M_n)$ to lighten the notations. For $n,j\geq 1$, $t\in[-1,0)$ and $s\in[0,1]$, we set $f_j^n(t)=f_j(t)=0$ and
\begin{align*}
f_j^n(s)&=C_{2Z_{j-1}^n+2s(Z_j^n-Z_{j-1}^n)}(M_n)+\frac{1}{2}\left(\ell_{2Z_j^n}'(M_n)-\ell_{2Z_{j-1}^n+2s(Z_j^n-Z_{j-1}^n)}'(M_n)\right),\\
f_j(s)&=2B_{D_{j-1}+s(D_j-D_{j-1})}^{|\mathsf{br}|}+\frac{1}{2}\left(L_{D_j}-L_{D_{j-1}+s(D_j-D_{j-1})}\right).
\end{align*}
They are excursions with a unique jump at $0$ and it holds that $\Delta_0(f_j^n)=\#\Gamma_j(M_n)/2$. Recall that the vernation tree coded by the contour process of a plane tree $\tau$ is the pointed weighted metric space spanned by $\tau$, that we denote by $\overline{\tau}$ here. Let us remark that $\ell'(M_n)$ only decreases (at speed $1/2$) on the intervals where $C(M_n)$ is constant. Thanks to the branching property of $\dv$, we observe that $\vern[f_j^n]$ is obtained by regularly gluing the $\overline{T_\gamma(M_n)}$, for $\gamma\in\Gamma(j)$ and in the order prescribed by $\leq_{M_n}$, on a metric circle of length $\#\Gamma_j(M^n)$ endowed with its Lebesgue measure, and by normalizing the sum of the measures of each component. Thus, we find $\mathtt{d}_{\mathrm{GHP}}^\bullet(n^{-1/2}\cdot B_j(M_n),n^{-1/2}\cdot\vern[f_j^n])\leq n^{-1/2}$ for all $j\geq 1$. The root of $\vern[f_j^n]$ corresponds to $M_n(\gamma_j^\bullet)$ instead of $\gamma_j^\bullet$, but they are at distance at most $1$ anyway. Using Skorokhod's representation theorem, we assume the convergences in distribution of Corollary~\ref{mapping_processes} happen almost surely. It follows that $n^{-1/2}f_j^n\longrightarrow f_j$ uniformly on $[0,1]$ for all $j\geq 1$ almost surely, by continuity of $B^{|\mathsf{br}|}$ and $L$. Since $\dv$ is homogeneous, Theorem~\ref{gh-continu_particular} under $(a)$ directly ensures that $n^{-1/2}\cdot\vern[f_j^n]\longrightarrow \vern[f_j]$ for the pointed GHP topology for all $j\geq 1$ almost surely. Next, we point out that clearly $D_j\longrightarrow 1$ almost surely, so for all $\varepsilon>0$, there is $N\geq 1$ such that $||f_j||_\infty\leq \varepsilon$ and $||f_j^n||_\infty\leq \varepsilon\sqrt{n}$ for all $n\geq N$ and $j\geq N$. In that case, $\diam(\vern[f_j])\leq 6\varepsilon$ and $\diam(n^{-1/2}\cdot\vern[f_j^n])\leq 6\varepsilon$ according to the inequalities (\ref{maj_dl}) and (\ref{maj_dtree}). Eventually, we deduce the desired scaling limit in distribution for $G(M_n)$. 

\begin{theorem}
The following convergence holds in distribution on the space $(\mathbb{S},\mathtt{d}_{\mathbb{S}})$:
\[\frac{1}{\sqrt{n}}\cdot G(M_n)\convd \left(D_j-D_{j-1},\vern[f_j]\right)_{j\geq 1}.\]
\end{theorem}

\subsection{Continuity in distribution of random stable looptrees}
\label{stable}

We briefly present the stable Lévy processes as Curien \& Kortchemski~\cite{curien2014}. Let us fix $\alpha\in(1,2)$. Let $X^{(\alpha)}$ be the \emph{$\alpha$-stable Lévy process}, which is defined as the stable spectrally positive Lévy process such that $\E\left[\exp\left(-\lambda X_t^{(\alpha)}\right)\right]=\exp(t\lambda^\alpha)$ for all $t\geq0$ and $\lambda>0$. It is a random c\`adl\`ag process from $[0,\infty)$ to $\mathbb{R}$ whose all jumps are positive, and it satisfies the important scaling property that $\left(c^{-1/\alpha}X_{ct}^{(\alpha)}\right)_{t\geq0}$ has the same distribution as $X^{(\alpha)}$ for all $c>0$. As a Lévy process, it is also characterized by its Lévy measure: \[\Pi_\alpha(\dd r)=\frac{\alpha(\alpha-1)}{\Gamma(2-\alpha)}r^{-\alpha-1}\un_{(0,\infty)}(r)\dd r.\]
Furthermore, $X^{(\alpha)}$ is a recurring process because $\limsup_{t\rightarrow+\infty}X_t^{(\alpha)}=+\infty$ and $\liminf_{t\rightarrow+\infty}X_t^{(\alpha)}=-\infty$ almost surely. We also know the local minima of $X^{(\alpha)}$ are distinct and reached almost surely: namely for all $t\geq 0$, if $X_{t-}^{(\alpha)}=\inf_{[0,t]}X^{(\alpha)}$ then $X_t^{(\alpha)}=X_{t-}^{(\alpha)}$. Plus, it holds 
\begin{equation}
\label{>0pp-ps}
\P\left(X_t^{(\alpha)}=\inf_{[0,t]}X^{(\alpha)}\right)=0
\end{equation}for all $t>0$. For more information about stable processes or the Lévy processes in general, and for the proof of the facts mentioned here, we refer to Bertoin~\cite{taylor_1998}.

Now, we define the normalized excursion of $X^{(\alpha)}$ above its infimum straddling the time $1$ by following Chaumont~\cite{Chaumont1997ExcursionNM}. Its left and right boundaries are
\[\underline{g}=\underline{g}(X^{(\alpha)})=\sup\left\{s\leq 1\ :\ X_s^{(\alpha)}=\inf_{[0,s]}X^{(\alpha)}\right\}\quad\text{ and }\quad
\underline{d}=\underline{d}(X^{(\alpha)})=\inf\left\{s\geq 1\ :\ X_s^{(\alpha)}=\inf_{[0,s]}X^{(\alpha)}\right\},\]
and its length is $\underline{\zeta}=\underline{\zeta}(X^{(\alpha)})=\underline{d}-\underline{g}$. Almost surely, it holds $0<\underline{g}<1<\underline{d}$ and 
\begin{equation}
\label{g,d_running_inf}
X_{\underline{g}}^{(\alpha)}=\inf_{[0,\underline{g}]}X^{(\alpha)}=\inf_{[0,1]}X^{(\alpha)}=\inf_{[0,\underline{d}]}X^{(\alpha)}=X_{\underline{d}}^{(\alpha)}.
\end{equation}
Moreover, the random variable $\underline{\zeta}$ admits a positive density on $(0,\infty)$. Then, we eventually set
\[X_t^{\mathsf{exc},(\alpha)}=\un_{[0,1]}(t)\underline{\zeta}^{-1/\alpha}\left(X_{\underline{g}+\underline{\zeta}t}^{(\alpha)}-X_{\underline{g}-}^{(\alpha)}\right)\]
for all $t\in[-1,1]$. This process can be interpreted as an excursion of $X^{(\alpha)}$ above its infimum conditioned to have length $1$. It is strictly positive on $(0,1)$ because the minima of $X^{(\alpha)}$ are distinct, and it is an excursion in the sense of Definition~\ref{excursion}. Furthermore, it is independent from $\underline{\zeta}$. To lighten the notations, we write $\lop[\alpha]=\lop[X^{\mathsf{exc},(\alpha)}]$, $\lopt[\alpha]=\lopt[X^{\mathsf{exc},(\alpha)}]$, $\vern[\alpha]=\vern[X^{\mathsf{exc},(\alpha)}]$, and $\vernt[\alpha]=\vernt[X^{\mathsf{exc},(\alpha)}]$. We point out that $\lop[\alpha]$ is exactly the random $\alpha$-stable looptree constructed by Curien \& Kortchemski~\cite{curien2014}. By analogy, one could call $\vern[\alpha]$ the random $\alpha$-stable vernation tree. However, this name would not be so useful because Theorem~\ref{intro_uni} and the next result, coming from \cite[Corollary 3.4]{curien2014}, imply that $\vern[\alpha]=2\cdot\lop[\alpha]$ a.s. 

\begin{proposition}
For all $\alpha\in(1,2)$, the excursion $X^{\mathsf{exc},(\alpha)}$ is PJG in the sense of Definition~\ref{type_exc} almost surely.
\end{proposition}

Before proving the proposition, we extend the notation $f(s-)$, $\Delta_s(f)$, and $x_s^t(f)$ given by (\ref{delta_x}) for any càdlàg function $f:[0,\infty)\longrightarrow\mathbb{R}$. We also set $f(0-)=0$ and $\Delta_0(f)=f(0)$, and we define $\underline{f}(t)=\inf_{[0,t]}f=x_0^t(f)$ for all $t\geq 0$.

\begin{proof}
The proof can be essentially found in \cite{curien2014}, so we only detail some tedious steps using tools we have developed earlier. We fix $\alpha\in(1,2)$ and forget it in the notations. For $n\geq 0$ and $s\in[-1,1]$, let $T_n=\inf\{t\geq0\ :\ X_t=-n\}$ and $f_n(s)=\un_{[0,1]}(s)(X_{sT_n}+n)$. By recurrence of $X$, $T_n$ is finite almost surely, and $f_n$ is an excursion. We see that $x_s^t(f_n)=x_{sT_n}^{tT_n}(X)$ for all $s,t\in(0,1]$, and $x_0^t(f_n)=\underline{X}_{tT_n}+n$ for all $t\in[0,1]$. We first prove that almost surely, $f_n$ is PJG for all $n\geq 0$, which is equivalent to
\begin{equation}
\label{X_sauts_purs}
X_t-\inf_{[0,t]}X=\sum_{s\in [0,t]}x_s^t(X)
\end{equation}
for all $t\geq 0$ since $T_n\longrightarrow+\infty$ almost surely. The proof that (\ref{X_sauts_purs}) holds almost surely for every fixed $t\geq 0$ is given by Curien \& Kortchemski~\cite[Lemma 3.3]{curien2014}. The left-hand side of (\ref{X_sauts_purs}) is obviously càdlàg, and it is clear that if $t\in[0,T_n]$ then the right-hand side is equal to $J f_n(t/T_n)-x_0^{t/T_n}(f_n)$ and thus it is càdlàg by Theorem~\ref{intro_uni}. Hence, (\ref{X_sauts_purs}) holds for all $t\geq 0$ almost surely. Now, recall with (\ref{calcul_Jf}) that an excursion $f$ is PJG if and only if $\dtree[f](t,1)=0$ for any $t\in[0,1]$. Almost surely, there is $N\geq 1$ such that $\underline{d}<T_N$. The excursion $f_N$ is PJG so $\dtree[f_N](t,1)=0$ for any $t\in[0,1]$, and it follows that $\dtree[X^{\mathsf{exc}}](t,1)=0$ for any $t\in[0,1]$ by applying Proposition~\ref{time-change} and Proposition~\ref{branching_property} on the interval $(\underline{g}/T_N,\underline{d}/T_N)$.
\end{proof}

We set $X_t^{\mathsf{exc},(1)}=\un_{[0,1]}(t)(1-t)$ for all $t\in[-1,1]$, and we write $\vern[1]=\vern[X^{\mathsf{exc},(1)}]$ and $\vernt[1]=\vernt[X^{\mathsf{exc},(1)}]$. It is immediate that $X^{\mathsf{exc},(1)}$ is a PJG excursion and that $\vern[1]=\vernt[1]=2\cdot\lop[X^{\mathsf{exc},(1)}]=2\cdot\lopt[X^{\mathsf{exc},(1)}]=2\cdot\mathcal{C}$, where we recall from (\ref{circle_dist_cano}) that $\mathcal{C}$ is the metric circle of length $1$. Let $\mathbf{e}$ be the standard Brownian excursion on $[0,1]$. We set $X_t^{\mathsf{exc},(2)}=\un_{[0,1]}(t)\sqrt{2}\cdot\mathbf{e}$ for any $t\in[-1,1]$, and we write $\vern[2]=\vern[X^{\mathsf{exc},(2)}]$ and $\vernt[2]=\vernt[X^{\mathsf{exc},(2)}]$. It is immediate that $X^{\mathsf{exc},(2)}$ is a continuous excursion and $\vern[2]=\vernt[2]=\tree[X^{\mathsf{exc},(2)}]=\sqrt{2}\cdot\tree[\mathbf{e}]$, where $\tree[\mathbf{e}]$ is the Brownian Continuum Random Tree introduced by Aldous~\cite{aldousI,aldous1993}.

\begin{theorem}
\label{stable_main_result}
For all $\beta\in(1,2)$, the following three convergences hold in distribution for the pointed GHP topology:  \[\lop[\alpha]\xrightarrow[\alpha\rightarrow 1+]{d}\mathcal{C},\quad\quad \lop[\alpha]\xrightarrow[\alpha\rightarrow \beta]{d}\lop[\beta],\quad\quad \lop[\alpha]\xrightarrow[\alpha\rightarrow 2-]{d}\tfrac{\sqrt{2}}{2}\cdot\tree[\mathbf{e}].\]
In other words, the family $\alpha\in[1,2]\longmapsto \vern[\alpha]$ is continuous in distribution for the pointed GHP topology. 
\end{theorem}

The convergences when $\alpha$ tends toward $1$ or $2$ have been shown by Curien \& Kortchemski~\cite{curien2014}. Here, we retrieve and generalize them by using Theorem~\ref{gh-continu} and the lemma below.

\begin{lemma}
\label{excursion_stable_continu}
The family $\alpha\in[1,2]\longmapsto X^{\mathsf{exc},(\alpha)}$ is continuous in distribution for the Skorokhod topology on $[-1,1]$.
\end{lemma}

\begin{proof}[Proof of the lemma]
The desired convergences when $\alpha$ tends towards $1$ or $2$ have already been proven in \cite[Propositions 3.5 and 3.6]{curien2014}. However, we point out that when $\alpha\rightarrow 1$, the convergence as stated in \cite{curien2014} is false even with the time-reversing, because it would entail $0=X_0^{\mathsf{exc},(\alpha)}\convd 1$. The issue is that the Skorokhod topology on $[0,1]$ is too rigid for the boundary value $f(0)$. The mistake is at the end of the proof of Proposition~3.6 with the application of the Vervaat transform because the time of the infimum of the bridge is close, but not equal, to the unique big jump of the bridge. Our use of the Skorokhod topology on $[-1,1]$ fixes the problem as it allows $f_n\to f$ even when $f_n(0)=0$ and $f(0)=1$.

If $\beta\in(1,2)$, we obtain the convergence in finite-dimensional distribution thanks to the Lévy property and to the convergences $\E\left[\exp(-\lambda X_t^{(\alpha)})\right]=\exp(t\lambda^\alpha)\underset{\alpha\rightarrow\beta}{\longrightarrow}\exp(t\lambda^\beta)=\E\left[\exp(-\lambda X_t^{(\beta)})\right]$ for any $t,\lambda\geq0$. To show the tightness, we use the scaling property and bounds on fractional moments of stable distributions (see e.g.~\cite{stablemoments1,stablemoments2}) to prove there is $C(\beta)\in(0,\infty)$ that only depends on $\beta$ such that 
\[\E\left[|X_t^{(\alpha)}-X_s^{(\alpha)}|^\gamma |X_s^{(\alpha)}-X_r^{(\alpha)}|^\gamma \right]=\E\left[(X_1^{(\alpha)})^\gamma\right]^2|t-s|^{\gamma/\alpha} |s-r|^{\gamma/\alpha}\leq C(\beta)|t-r|^{1+1/(2\beta+1)}\]
with $2\gamma=\beta+1$, for all $\beta+2<2\alpha<2\beta+1$ and $r\leq s\leq t\leq r+1$. This implies the tightness by the Kolmogorov criterion for c\`adl\`ag processes: see Billingsley~\cite[Theorem 13.5]{billingsley2013convergence}. Hence, $X^{(\alpha)}$ converges in distribution to $X^{(\beta)}$ for the Skorokhod topology on $[0,T]$ when $\alpha\rightarrow\beta$, for all $T\geq 0$. Then, we can suppose that convergence happens a.s.~on $[0,T]$ for all integers $T$ by Skorokhod's representation theorem. Since local minima of $X^{(\beta)}$ are distinct and reached a.s., we get $\underline{g}(X^{(\alpha)})\longrightarrow \underline{g}(X^{(\beta)})$ a.s. Moreover, either by applying the strong Markov property at $\underline{d}(X^{(\beta)})$ or by recalling (\ref{g,d_running_inf}) and that the local local minima of $X^{(\beta)}$ are distinct, we check that $\underline{d}(X^{(\beta)})$ is not a local minima of $X^{(\beta)}$ a.s., so $\underline{d}(X^{(\alpha)})\longrightarrow \underline{d}(X^{(\beta)})$ a.s. See Bertoin~\cite{taylor_1998} for more details. Since $X^{(\beta)}$ is a.s.~continuous at $\underline{g}(X^{(\beta)})$ and $\underline{d}(X^{(\beta)})$, it follows that $X^{\mathsf{exc},(\alpha)}$ converges in law to $X^{\mathsf{exc},(\beta)}$ for the Skorokhod topology on $[0,1]$ when $\alpha\rightarrow\beta$.
\end{proof}

Because we know that $X^{\mathsf{exc},(\alpha)}$ is almost surely PJG for $\alpha\in[1,2)$, Theorem~\ref{gh-continu_particular} automatically shows the continuity in distribution of the family $\alpha\longmapsto \lop[\alpha]$ on $[1,2)$. However, it is not enough to prove the convergence for $\alpha\rightarrow 2$ because $\mathbf{e}$ is continuous. Hence, we present a probabilistic method that will work for all $\alpha\in[1,2]$.

Let us quickly present Itô's excursion theory applied to the stable Lévy process. We refer to Bertoin~\cite[Chapter IV]{taylor_1998} for proof and details. Recall that $\underline{X}_t=\inf\{X_s\ :\ s\in[0,t]\}$ is the running infimum process of $X$. Since the process $X-\underline{X}$ is strong Markov and $0$ is regular for itself, one can use Itô's excursion theory and take $-\underline{X}$ as the local time of $X-\underline{X}$ at level $0$. Let $(g_j,d_j)$, $j\in\mathcal{J}$, be the excursion intervals of $X-\underline{X}$ away from $0$. Almost surely, we have $X_{g_j-}=X_{g_j}=\underline{X}_{g_j}$ for all $j\in\mathcal{J}$ and for all $t\geq 0$, if $t$ is a jump time of $X$ then there exists $j\in\mathcal{J}$ such that $t\in(g_j,d_j)$. For all $j\in\mathcal{J}$, we set $\omega_j(s)=X_{\min(g_j+s,d_j)}-X_{g_j}$ with $s\geq0$, the excursion away from $0$ indexed by $j$. The $\omega_j$ are elements of the space $\mathcal{E}=\{\omega:[0,\infty)\longrightarrow[0,\infty)\ :\ \text{c\`adl\`ag such that }\exists M\geq0,\ \forall t\geq M,\ \omega(t)=0\}$. From Itô's excursion theory, the point measure \[\mathcal{N}_X=\sum_{j\in\mathcal{J}}\delta_{(-X_{g_j},\omega_j)}\]is a Poisson point measure on $[0,\infty)\times\mathcal{E}$ with intensity $\dd t\, \underline{n}(\dd\omega)$, where $\underline{n}$ is a $\sigma$-finite measure on $\mathcal{E}$. Moreover, $X$ is measurable with respect to the measure $\mathcal{N}_X$ because $X_t>\underline{X}_t$ almost everywhere almost surely thanks to (\ref{>0pp-ps}) and because the local minima $X_{g_j}$ are almost surely distinct. 

Let us choose the shuffle we will work with. When $\Delta\in[1/2,1)$, we define $\phi_{1-\Delta}$ as in (\ref{remark_shuffle}), and when $\Delta\geq1/2$, we set $\phi_\Delta=\phi_{1/2}$. It is straightforward to verify that $\Phi:\Delta>0\longmapsto \phi_\Delta$ is a continuous shuffle as in Definition~\ref{shuffle}. Recall the sets $B_{X^{\text{exc},(\alpha)}}$ and $\mathrm{B}(\Phi)$ from Theorem~\ref{gh-continu}. Notice that all the $\phi_\Delta$ have only a countable number of discontinuity points. From the strong Markov property and the law of $\mathcal{N}_X$, we see that $B_{X^{\text{exc},(\alpha)}}\cap\mathrm{B}(\Phi)=\emptyset$ almost surely for all $\alpha\in(1,2)$. We also have $B_{X^{\text{exc},(1)}}=\{1\}$ and $B_{X^{\text{exc},(2)}}=\emptyset$. Thus, Lemma~\ref{excursion_stable_continu} and Theorem~\ref{gh-continu} ensure that the family $\alpha\in[1,2]\longmapsto \vernt[\alpha]=2\cdot\lopt[\alpha]$ is continuous in law for the pointed GHP topology. The end of the proof of Theorem~\ref{stable_main_result} is based on the observation that the Lebesgue measure on $[0,1]$ is invariant under all the $x\in[0,1]\longmapsto\phi_\Delta(x)\in[0,1]$. Informally, the shuffle under $\Phi$ of joint points between loops does not change the law of the stable looptree. This is formally asserted by the next theorem.

\begin{theorem}
\label{stable_tilte_thm}
For all $\alpha\in(1,2)$, the random pointed weighted compact metric spaces $\lop[\alpha]$ and $\lopt[\alpha]$ have the same law.
\end{theorem}

Let $\mathbb{N}^*$ be the set of positive integers. We fix $\alpha\!\in\!(1,2)$ and forget it in the notations. Let
$\varepsilon\!>\!0$. We define a sequence $\left(\tau_n(X)\right)_{n\geq 0}$ of $\sigma(X)$-measurable stopping times by setting $\tau_0(X)=0$ and $\tau_n(X)=\inf\left\{t>\tau_{n-1}(X)\ :\ \Delta_t(X)\geq\varepsilon\right\}$ for $n\geq 1$. The $\tau_{n+1}(X)-\tau_n(X)$ are independent and have the same law as $\tau_1(X)$. Plus, $\tau_1(X)>0$ almost surely because $X$ is c\`adl\`ag with $X_0=0$. We postpone the proof of the following lemma.

\begin{lemma}
\label{stable_tilte}
There almost surely exist an $\alpha$-stable Lévy process $Y$, a bijection $\psi:\mathbb{N}^*\longrightarrow\mathbb{N}^*$, and a function $\lambda:[0,\infty)\longrightarrow [0,\infty)$ which avoids at most countably many points and preserves the Lebesgue measure on $[0,\infty)$ such that for all $t\geq 0$ and $k\in\mathbb{N}^*$, it holds $\underline{Y}_{\lambda(t)}=\underline{Y}_t=\underline{X}_t$, $\underline{Y}_{\tau_{\psi(k)}(Y)}=\underline{X}_{\tau_k(X)}$, $\Delta_{\tau_{\psi(k)}}(Y)=\Delta_{\tau_k}(X)$, and
\[\frac{x_{\tau_{\psi(k)}}^{\lambda(t)}}{\Delta_{\tau_{\psi(k)}}}(Y)=\phi_{\Delta_{\tau_k}(X)}\left(\frac{x_{\tau_k}^{t}}{\Delta_{\tau_k}}(X)\right)\quad\text{ whenever }\quad\tau_k(X)\leq\underline{d}(X).\]
\end{lemma}

The local minima of a stable Lévy process are a.s.~distinct, so (\ref{g,d_running_inf}) entails that $\underline{g}(X)=\inf\left\{t\leq 1\ :\ \underline{X}_t=\underline{X}_1\right\}$ and $\underline{d}(X)=\sup\left\{t\geq 1\ :\ \underline{X}_t=\underline{X}_1\right\}$. The identity $\underline{Y}=\underline{X}$ thus yields $\underline{g}(X)=\underline{g}(Y)$, $\underline{d}(X)=\underline{d}(Y)$, and $\underline{\zeta}(X)=\underline{\zeta}(Y)$. Furthermore, as $t\in\left[\underline{g},\underline{d}\right]$ if and only if $\underline{X}_t=\underline{X}_1$, we deduce from the properties provided by Lemma~\ref{stable_tilte} that
\begin{align*}
\tau_k(X)\in\left[\underline{g},\underline{d}\right]&\Longleftrightarrow \tau_{\psi(k)}(Y)\in\left[\underline{g},\underline{d}\right],\\
t\in\left[\underline{g},\underline{d}\right]&\Longleftrightarrow \lambda(t)\in\left[\underline{g},\underline{d}\right],
\end{align*}
for all $t\geq 0$ and $k\in\mathbb{N}^*$. Hence, $\lambda$ induces a function $\lambda^{\mathsf{exc}}:[0,1]\to[0,1]$ such that $\lambda(\underline{g}+\underline{\zeta}t)=\underline{g}+\underline{\zeta}\lambda^{\mathsf{exc}}(t)$ for all $t\in[0,1]$. This function avoids at most countably many points and preserves the Lebesgue measure on $[0,1]$. Now, recall from (\ref{Jeps_def}) the operator $J^\varepsilon$. For all $s,t\in[0,1]$, Lemma~\ref{stable_tilte} and the identity (\ref{dl_Jeps}) allow us to compute
\begin{align*}
\dlt[\underline{\zeta}^{1/\alpha} J^\varepsilon  X^{\mathsf{exc}}](s,t)&=\sum_{\substack{k\in\mathbb{N}^*\\ \tau_k(X)\in\left[\underline{g},\underline{d}\right]}}\tilde{\delta}_{\tau_k(X)}^X\left(x_{\tau_k}^{\underline{g}+\underline{\zeta}s}(X),x_{\tau_k}^{\underline{g}+\underline{\zeta}t}(X)\right)\\
&=\sum_{\substack{k\in\mathbb{N}^*\\ \tau_{k}(Y)\in\left[\underline{g},\underline{d}\right]}}\delta_{\tau_{k}(Y)}^Y\left(x_{\tau_{k}}^{\lambda(\underline{g}+\underline{\zeta}s)}(Y),x_{\tau_{k}}^{\lambda(\underline{g}+\underline{\zeta}t)}(Y)\right)=\dl[\underline{\zeta}^{1/\alpha} J^\varepsilon  Y^{\mathsf{exc}}]\left(\lambda^{\mathsf{exc}}(s),\lambda^{\mathsf{exc}}(t)\right).
\end{align*}
The same method gives that $0=\dl[\underline{\zeta}^{1/\alpha} J^\varepsilon  Y^{\mathsf{exc}}]\left(\lambda^{\mathsf{exc}}(1),1\right)$. It follows that the map $\lambda^{\mathsf{exc}}$ naturally induces an isometry $\lambda^0$ from $\lopt[\underline{\zeta}^{1/\alpha} J^\varepsilon  X^{\mathsf{exc}}]$ to $\lop[\underline{\zeta}^{1/\alpha} J^\varepsilon  Y^{\mathsf{exc}}]$ which avoids at most countably many points and maps the respective roots onto one another. In fact, Proposition~\ref{dt_continu} implies that $\lambda^0$ is genuinely surjective. Let us denote by $\tilde{\mu}$ and $\mu$ the respective measures of $\lopt[\underline{\zeta}^{1/\alpha} J^\varepsilon  X^{\mathsf{exc}}]$ and $\lop[\underline{\zeta}^{1/\alpha} J^\varepsilon  Y^{\mathsf{exc}}]$, which are the images of the Lebesgue measure on $[0,1]$ by the respective canonical projections. The invariance of the Lebesgue measure on $[0,1]$ by $\lambda^{\mathsf{exc}}$ readily yields $\mu=\lambda_*^0\tilde{\mu}$. Hence, the pointed weighted metric spaces $\lopt[\underline{\zeta}^{1/\alpha}J^\varepsilon X^{\mathsf{exc}}]$ and $\lop[\underline{\zeta}^{1/\alpha}J^\varepsilon  Y^{\mathsf{exc}}]$ are GHP-isometric. Since $X$ and $Y$ have the same distribution, we get that
\[\left(\underline{\zeta}(X),\lopt[\underline{\zeta}^{1/\alpha}(X)J^\varepsilon X^{\mathsf{exc}}]\right)\ed\left(\underline{\zeta}(X),\lop[\underline{\zeta}^{1/\alpha}(X)J^\varepsilon X^{\mathsf{exc}}]\right).\]Proposition~\ref{cv_J} $(iii)$, Corollary~\ref{propo}, and (\ref{GHP_lips}) ensure that the couples $\left(\underline{\zeta},\lopt[\underline{\zeta}^{1/\alpha} X^{\mathsf{exc}}]\right)$ and $\left(\underline{\zeta},\lop[\underline{\zeta}^{1/\alpha} X^{\mathsf{exc}}]\right)$ have the same law, by making $\varepsilon\rightarrow 0^+$. Recall that $\dl$ is homogeneous but not $\dlt$. Nevertheless, recall that $\underline{\zeta}$ admits a positive density on $(0,\infty)$ and that $\underline{\zeta}$ and $X^{\mathsf{exc}}$ are independent. For any bounded continuous function $G$ and any $\varepsilon>0$, we can then write
\[\E\left[G\left(\lop[X^{\mathsf{exc}}]\right)\right]=\frac{1}{\P\left(|\underline{\zeta}-1|<\varepsilon\right)}\E\left[\un_{\left\{|\underline{\zeta}-1|<\varepsilon\right\}}G\left(\underline{\zeta}^{-1/\alpha}\cdot\lopt[\underline{\zeta}^{1/\alpha}X^{\mathsf{exc}}]\right)\right].\]
With the help of the dominated convergence theorem and Theorem~\ref{gh-continu}, we end the proof by making $\varepsilon\rightarrow0^+$.

\begin{proof}[Proof of Lemma~\ref{stable_tilte}]
We set $X^0=X$, $\psi_0=\mathsf{id}_{\mathbb{N}^*}$, and $\lambda_0=\mathsf{id}_{[0,\infty)}$. Let $n\geq 0$. By induction, let us assume we have constructed an $\alpha$-stable Lévy process $X^n$, a bijection $\psi_n:\mathbb{N}^*\longrightarrow\mathbb{N}^*$, and a function $\lambda_n:[0,\infty)\longrightarrow[0,\infty)$ which avoids at most countably many points and preserves the Lebesgue measure on $[0,\infty)$ such that for all $t\geq 0$ and $k\in\mathbb{N}^*$, it holds $\underline{X^n}_{\lambda_n(t)}=\underline{X^n}_t=\underline{X}_t$, $\underline{X^n}_{\tau_{\psi_n(k)}(X^n)}=\underline{X}_{\tau_k(X)}$, $\Delta_{\tau_{\psi_n(k)}}(X^n)=\Delta_{\tau_k}(X)$, and
\[\frac{x_{\tau_{\psi_n(k)}}^{\lambda_n(t)}}{\Delta_{\tau_{\psi_n(k)}}}(X^n)=\begin{cases}
	\phi_{\Delta_{\tau_k}(X)}\left(\frac{x_{\tau_k}^{t}}{\Delta_{\tau_k}}(X)\right)&\text{ if }\psi_n(k)\leq n,\\
	\frac{x_{\tau_k}^t}{\Delta_{\tau_k}}(X)&\text{ if }\psi_n(k)>n.
	\end{cases}\]

We write $\boldsymbol{\Delta}=\Delta_{\tau_{n+1}(X^n)}(X^n)$ and $\ftau=\tau_{n+1}(X^n)$ to lighten the notations. For any $t\geq0$, we set $\mathbf{X}_t=X_{\ftau+t}^n-X_{\ftau}^n$. Thanks to the strong Markov property of an $\alpha$-stable Lévy process, we know that $\mathbf{X}$ is an $\alpha$-stable Lévy process independent from $X_{\min(\cdot,\ftau)}^n$. Let us denote by $(g_i,d_i)$, $i\in\mathcal{I}$, the excursion intervals of $\mathbf{X}-\underline{\mathbf{X}}$ away from $0$ and by $\omega_i$ its excursion away from $0$ on $(g_i,d_i)$, so that the point measure $\mathcal{N}_{\mathbf{X}}=\sum_{i\in\mathcal{I}}\delta_{(-\mathbf{X}_{g_i},\omega_i)}$ is a Poisson point measure on $[0,\infty)\times\mathcal{E}$ with intensity $\dd t \underline{n}(\dd \omega)$. It is independent from $X_{\min(\cdot,\ftau)}^n$, and so from $\boldsymbol{\Delta}$. We define two maps as follows:
\[\varphi:x\in[0,\infty) \longmapsto \begin{cases}
                          \boldsymbol{\Delta}-\boldsymbol{\Delta}\phi_{\boldsymbol{\Delta}}\left(1-x/\boldsymbol{\Delta}\right)&\text{ if }x<\boldsymbol{\Delta},\\
                          x&\text{ if }x\geq\boldsymbol{\Delta},
                        \end{cases}\quad\text{ and }\quad\Upsilon:(x,\omega)\in[0,\infty)\times\mathcal{E}\longmapsto(\varphi(x),\omega)\in[0,\infty)\times\mathcal{E}.\]

Recall from (\ref{remark_shuffle}) the expression of $\phi_\Delta$. We check that the Lebesgue measure on $[0,1]$ is invariant by $\phi_\Delta$ for all $\Delta>0$. Indeed, $\phi_\Delta:[0,1]\longrightarrow[0,1)$ is surjective and we can partition $(0,1]$ into countably many intervals on whose $\phi_\Delta$ is affine with slope $\pm 1$. Whatever the value of $\boldsymbol{\Delta}$ is, it follows that the product measure $\dd t\underline{n}(\dd\omega)$ is invariant by $\Upsilon$, so conditionally given $X_{\min(\cdot,\ftau)}^n$, the point measure $\Upsilon_*\mathcal{N}_{\mathbf{X}}$ is also a Poisson point measure with intensity $\dd t\underline{n}(\dd \omega)$. Therefore, there exists an $\alpha$-stable Lévy process $X^{n+1}$ that satisfies $\tau_{n+1}(X^{n+1})=\ftau$, $X_{\min(\cdot,\ftau)}^{n+1}=X_{\min(\cdot,\ftau)}^n$, and $\mathcal{N}_{\hat{\mathbf{X}}}=\Upsilon_*\mathcal{N}_{\mathbf{X}}$ almost surely, where $\hat{\mathbf{X}}_t=X_{\ftau+t}^{n+1}-X_{\ftau}^{n+1}$ for all $t\geq 0$.
 Analogously, we denote by $(\hat{g}_j,\hat{d}_j)$, $j\in\mathcal{J}$, the excursion intervals of $\hat{\mathbf{X}}-\underline{\hat{\mathbf{X}}}$ away from $0$ and by $\hat{\omega}_j$ its excursion away from $0$ on $(\hat{g}_j,\hat{d}_j)$, so that \[\mathcal{N}_{\hat{\mathbf{X}}}=\sum_{j\in\mathcal{J}}\delta_{(-\hat{\mathbf{X}}_{\hat{g}_j},\hat{\omega}_j)}=\sum_{i\in\mathcal{I}}\delta_{\Upsilon(-\mathbf{X}_{g_i},\omega_i)}.\]
For all $x\geq 0$, we define $\sigma_x(\mathbf{X})=\inf\{t\geq 0\, :\, \underline{\mathbf{X}}_t\leq -x\}$ and $\sigma_x(\hat{\mathbf{X}})$ similarly. We set $\zeta:\omega\in\mathcal{E}\mapsto\sup\{t\geq 0\, :\, \omega(t)>0\}$, so that $\zeta(\omega_i)=d_i-g_i$ is the lifetime of $\omega_i$ for any $i\in\mathcal{I}$. It follows from (\ref{>0pp-ps}) that for all $x\geq 0$, it almost surely holds \[\sigma_x(\mathbf{X})=\sum_{i\in\mathcal{I}}\I{-\mathbf{X}_{g_i}<x}\zeta(\omega_i)=\mathcal{N}_{\mathbf{X}}\left[\un_{[0,x)}\zeta\right].\] But if $x\geq\boldsymbol{\Delta}$, we notice that $(\un_{[0,x)}\zeta)\circ\Upsilon=\un_{[0,x)}\zeta$, which gives $\sigma_x(\mathbf{X})=\sigma_x(\hat{\mathbf{X}})$ in that case. The processes $\mathbf{X}$ and $\hat{\mathbf{X}}$ are continuous, so $\min\left(-\boldsymbol{\Delta},\underline{\mathbf{X}}\right)=\min\left(-\boldsymbol{\Delta},\underline{\hat{\mathbf{X}}}\right)$ and $\sigma_{\boldsymbol{\Delta}}(\mathbf{X})=\sigma_{\boldsymbol{\Delta}}(\hat{\mathbf{X}})$ almost surely. By definition of $\boldsymbol{\Delta}$, it is now clear that $\underline{X^{n+1}}=\underline{X^n}=\underline{X}$ almost surely. Next, recall that the local minima of an $\alpha$-stable Lévy process are distinct almost surely. In particular, the map $\Upsilon$ naturally induces a bijection $\upsilon:\mathcal{I}\to\mathcal{J}$, so that $(-\hat{\mathbf{X}}_{\hat{g}_{\upsilon(i)}},\hat{\omega}_{\upsilon(i)})=\Upsilon(-\mathbf{X}_{g_i},\omega_i)$ for all $i\in\mathcal{I}$. When $\sigma_{\boldsymbol{\Delta}}(\mathbf{X})<g_i$ for some $i\in\mathcal{I}$, we can describe $[\hat{g}_{\upsilon(i)},\hat{d}_{\upsilon(i)}]$ as the interval on which $\underline{\hat{\mathbf{X}}}=\underline{\mathbf{X}}_{g_i}<-\boldsymbol{\Delta}$, and this justifies $\hat{\mathbf{X}}=\mathbf{X}$ on the interval $[\hat{g}_{\upsilon(i)},\hat{d}_{\upsilon(i)}]=[g_i,d_i]$. Eventually, we get $X_t^{n+1}=X_t^n$ for all $t\in[0,\ftau]\cup[\ftau+\sigma_{\boldsymbol{\Delta}}(\mathbf{X}),\infty)$.

Next, we build a function $\lambda':[0,\infty)\to[0,\infty)$. If $t\in[0,\ftau)$, then we set $\lambda'(t)=t$. If $t\geq\ftau$ and $t-\ftau\in[g_i,d_i]$ for $i\in\mathcal{I}$, then we set $\lambda'(t)=\hat{g}_{\upsilon(i)}+t-g_i$. If $t\geq\ftau$ but $t-\ftau$ is not in any of the $[g_i,d_i]$, then we set $\lambda'(t)=\ftau+\sigma_{\varphi(-\underline{\mathbf{X}}_{t-\ftau})}(\hat{\mathbf{X}})$. Observe that $\lambda'$ induces bijections resp.~from $[\ftau+g_i,\ftau+d_i]$ and $[0,\ftau)$ to $[\ftau+\hat{g}_{\upsilon(i)},\ftau+\hat{d}_{\upsilon(i)}]$ and $[0,\ftau)$ which preserve their Lebesgue measures, for any $i\in\mathcal{I}$. A.s., each of these two countable families of intervals partitions a conull subset of $[0,\infty)$, according to (\ref{>0pp-ps}), so the Lebesgue measure on $[0,\infty)$ is invariant by $\lambda'$. It is easy to check that if $s\neq\ftau$ then
\[\underline{X^{n+1}}_{\lambda'(t)}=\underline{X^n}_t,\ \quad x_{\ftau}^{\lambda'(t)}(X^{n+1})=\boldsymbol{\Delta}\phi_{\boldsymbol{\Delta}}\left(\frac{x_{\ftau}^t(X^n)}{\boldsymbol{\Delta}}\right),\quad\text{ and }\quad x_{\lambda'(s)}^{\lambda'(t)}(X^{n+1})=x_s^t(X^n),\]
for all $s,t\geq 0$. If $t'-\ftau>0$ is not in any of the $[\hat{g}_{\upsilon(i)},\hat{d}_{\upsilon(i)}]$ then $t'-\ftau$ is the unique point where $\underline{\hat{\mathbf{X}}}$ reaches $\underline{\hat{\mathbf{X}}}_{t'-\ftau}<0$, which means $t'=\ftau+\sigma_{-\underline{\hat{\mathbf{X}}}_{t'-\ftau}}(\hat{\mathbf{X}})$. The function $\varphi:[0,\infty)\to(0,\infty)$ is surjective, so there is $t\geq\ftau$ such that $\lambda'(t)=t'$. Hence, the only point avoided by $\lambda'$ is $\ftau$. Notice a jump has to be inside an excursion interval away from the infimum and that $\lambda'(\ftau)=\ftau+\sigma_{\boldsymbol{\Delta}/2}(\hat{\mathbf{X}})$ is not a jump of $X^{n+1}$. Eventually, our reasoning justifies the existence of a unique bijection $\psi_{n+1}:\mathbb{N}^*\longrightarrow\mathbb{N}^*$ such that $\psi_{n+1}(k)=n+1$ if $\psi_n(k)=n+1$ and $\tau_{\psi_{n+1}(k)}(X^{n+1})=\lambda'(\tau_{\psi_n(k)}(X^n))$ otherwise, for all $k\in\mathbb{N}^*$. Furthermore, we set $\lambda_{n+1}=\lambda'\circ\lambda_n$. It is then straightforward to check that $(X^{n+1},\psi_{n+1},\lambda_{n+1})$ satisfies the properties desired for the induction.

Now, observe that $t\leq \underline{d}(X)$ if and only if $\underline{X}_t\geq \underline{X}_1$, for any $t\geq 0$. Hence, for any $n\geq 0$ and $k\in\mathbb{N}^*$, we find $\underline{d}(X^n)=\underline{d}(X)$, and that $\tau_k(X)\leq \underline{d}(X)$ if and only if $\tau_{\psi_n(k)}(X^n)\leq \underline{d}(X)$. This yields that $X^n$ has the same number $N$ of jumps of height at least $\varepsilon$ in $[0,\underline{d}(X)]$ as $X$, which is finite because $X$ is c\`adl\`ag. Thus, $\tau_k(X)\leq \underline{d}(X)$ if and only if $\psi_n(k)\leq N$. Hence, $Y=X^N$, $\psi=\psi_N$, and $\lambda=\lambda_N$ verify all the desired properties. In particular, $X^N$ is an $\alpha$-stable Lévy process because $N$ and $\underline{d}$ are the same for all the $X^n$, so we can write for any nonnegative measurable function $G$:
\[\E\left[G\left(X^N\right)\right]=\sum_{n\geq 0}\E\left[\un_{\{\tau_n(X^n)\leq \underline{d}(X^n)<\tau_{n+1}(X^n)\}}G(X^n)\right]=\sum_{n\geq 0}\E\left[\un_{\{N=n\}}G(X)\right]=G(X).\]
\end{proof}

\begin{remark}
Even if it was needed to replace the looptrees with vernation trees, we believe Theorem~\ref{stable_tilte_thm} should stay true for a larger class of random processes. The main argument was that the point measure $\Upsilon_*\mathcal{N}_{\mathbf{X}}$ is distributed as $\mathcal{N}_{\mathbf{X}}$ conditionally given $X_{\min(\cdot,\ftau)}^n$, which should be obtained with some kind of exchangeability. We think in particular of the first passage bridges of processes with exchangeable increments occurring in the work \cite{marzouk_similaire} of Marzouk.\cq
\end{remark}

\subsection{Invariance principle for random discrete looptrees}
\label{subsection_discrete_loop}

Recall the formalism of words presented at the start of Section~\ref{application_proba}. Let $\tau$ be a plane tree. The discrete looptree $\mathsf{Loop}(\tau)$ associated with $\tau$ is defined as the graph on the set of vertices of $\tau$ where two vertices $u\leq v$ are joined by as many edges as verified conditions among the following:
\begin{itemize}
\item $u$ and $v$ are consecutive siblings of the same parent, which means $u=\overleftarrow{u}\! *\! (j)$ and $v=\overleftarrow{u}\! *\! (j\! +\! 1)$ with some $j\! \in\! \mathbb{N}^*$.
\item $v$ is the first child of $u$ in $\tau$, which means $v=u*(1)$,
\item $v$ is the last child of $u$ in $\tau$, which means $v=u*(k_u(\tau))$,
\item $u=v$ and $u$ does not have any children, which means $k_u(\tau)=0$.
\end{itemize}
In particular, if $v$ is the unique child of $u$ in $\tau$, then they are joined by exactly two edges. We endow this graph with the graph distance (every edge has unit length) and with the uniform probability measure on its vertices, and we distinguish $\varnothing$ as its root so that $\mathsf{Loop}(\tau)$ can be seen as a pointed weighted compact metric space. We can also inductively construct $\mathsf{Loop}(\tau)$ by first arranging $\varnothing,(1),(2),\ldots,(k_\varnothing(\tau))$ into a cycle of length $k_\varnothing(\tau)+1$, then by gluing the root of $\mathsf{Loop}\left(\theta_{(j)}\tau\right)$ respectively on $(j)$ for each $1\leq j\leq k_\varnothing(\tau)$. See Figure~\ref{discrete_looptree_associated} for example. There are other ways to associate a discrete looptree with a tree. For example, Curien \& Kortchemski~\cite{curien2014} make $\varnothing$ correspond to a cycle of length $k_\varnothing(\tau)$. For another version of the looptree, they erase the edges linking a leaf to itself. Marzouk~\cite{marzouk_similaire} contracts the last edge of each cycle, so that more than two cycles can share the same point.

Let $\left(u(i)\right)_{0\leq i\leq \#\tau-1}$ be the depth-first exploration of $\tau$ and let $\left(c(i)\right)_{0\leq i\leq 2\#\tau-2}$ be its exploration by contour. It is classic that the plane tree $\tau$ is coded by its Lukasiewicz walk $L(\tau)=\left(L_i(\tau)\right)_{0\leq i\leq \#\tau}$, which is defined by $L_0(\tau)=0$ and $L_{i+1}(\tau)=L_i(\tau)+k_{u(i)}(\tau)-1$ for all $0\leq i\leq \#\tau-1$. We introduce a variant of this process that is adapted with the exploration by contour instead of the depth-first exploration. Let $W(\tau)=\left(W_t(\tau)\right)_{t\in[0,2\#\tau-1]}$ be the process which is affine with slope $-1$ on each of the intervals $[i,i+1)$ for $0\leq i\leq 2\#\tau-2$, and such that $W_0(\tau)=k_\varnothing(\tau)+1$, $W_{2\#\tau-1}(\tau)=0$, and
\[W_{i}(\tau)=\begin{cases}
				W_{i-1}(\tau)+k_{c(i)}(\tau)&\text{ if }|c(i)|=|c(i-1)|+1,\\
				W_{i-1}(\tau)-1&\text{ if }|c(i)|=|c(i-1)|-1,
			  \end{cases}\]
for all $0\leq i\leq 2\#\tau-2$. Observe $W(\tau)$ can be expressed in terms of the Lukasiewicz walk and of the contour process as
\begin{equation}
\label{modif_luka}
W_i(\tau)=L_{\xi(i)+1}(\tau)+C_i(\tau)+2
\end{equation}
for all $0\leq i\leq 2\#\tau-2$, where we remind that $\xi(i)$ is the largest integer $l$ such that $u(l)$ has been visited by $\left(c(j)\right)_{0\leq j\leq i}$. The Lukasiewicz walk is nonnegative before its last value which is $-1$ (see e.g.~Le Gall~\cite[Proposition 1.1]{legall_trees}), so the process $W(\tau)$ is positive on $[0,2\#\tau-1)$ and is continuous at $2\#\tau-1$. Hence, the process \[w(\tau):t\in[-1,1]\longmapsto w_t(\tau)=\un_{[0,1]}(t)W_{(2\#\tau-1)t}(\tau)\] is an excursion whose jumps are at the times $i/(2\#\tau-1)$ with $i=0$ or $|c(i)|=|c(i-1)|+1$. Plus, $w(\tau)$ is PJG by Lemma~\ref{discrete_PJG}. Let $1\leq i\leq 2\#\tau-2$ such that $|c(i)|=|c(i-1)|+1$. We remark that $W_t\left(\theta_{c(i)}\tau\right)=W_{i+t}(\tau)-W_{i-}(\tau)$ for any $t\in[0,2\#(\theta_{c(i)}\tau)-1]$. This kind of branching property can be used with the help of Lemma~\ref{tri_utile} to compute that
\begin{equation}
\label{x_v}
x_{i_1/(2\#\tau-1)}^{i_2/(2\#\tau-1)}\left(w(\tau)\right)=\begin{cases}
	k_{c(i_1)}(\tau)+1&\text{ if }c(i_1)=c(i_2),\\
	k_{c(i_1)}(\tau)+1-j&\text{ if }c(i_1)*(j)\preceq c(i_2),\\
	0&\text{ otherwise},
	\end{cases}
\end{equation}
when $0\leq i_1,i_2\leq 2\#\tau-2$ are respectively either equal to $0$ or such that $|c(i)|=|c(i-1)|+1$. Furthermore, applying Proposition~\ref{branching_property} while recalling that $W(\tau)$ is affine of slope $-1$ by parts allows us to prove that $\lop[w(\tau)]$ can be constructed by gluing the pointed weighted metric spaces $\lop[w(\theta_{(j)}\tau)]$ for $1\leq j\leq k_\varnothing(\tau)$ respectively at the position $j$ on a metric circle of length $k_\varnothing(\tau)+1$ which is rooted at position $0$ and endowed with its Lebesgue measure, and then by normalizing the sum of the measures of the components. Hence, an easy induction on $\#\tau$ yields that $\lop[w(\tau)]$ is the pointed weighted metric space spanned with the graph $\mathsf{Loop}(\tau)$. Namely, it is obtained by replacing each edge with a line segment of length $1$ and by endowing this space with the normalized sum of the Lebesgue measures of those segments. With a slight abuse of notations, we can also see $\mathsf{Loop}(\tau)$ as a finite subset of $\lop[w(\tau)]$, so that they share the same root and that the graph distance of $\mathsf{Loop}(\tau)$ is inherited from the metric on $\lop[w(\tau)]$. See Figure~\ref{discrete_looptree_associated} for example. It is now clear that for any $a>0$, it holds
\begin{equation}
\label{GHloop_tau_v}
\mathtt{d}_{\mathrm{GHP}}^\bullet\left(\lop[\frac{1}{a}w(\tau)],\frac{1}{a}\cdot\mathsf{Loop}(\tau)\right)\leq \frac{1}{a}+\frac{1}{\#\tau},
\end{equation}
where recall from Section~\ref{def_GH} that $\mathtt{d}_{\mathrm{GHP}}^\bullet$ is the pointed GHP distance. Thus, the excursion $w(\tau)$ is crucial for understanding the scaling limits of large discrete looptrees. The following theorem confirms this intuition in the PJG case.

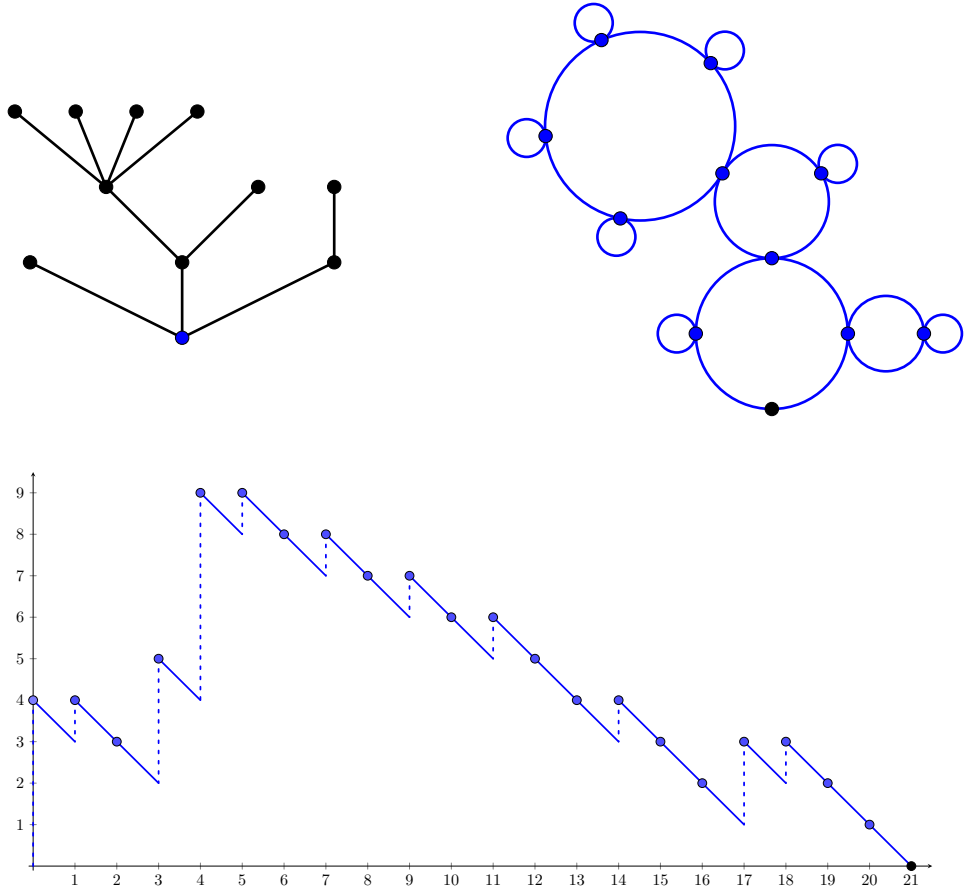
\begin{figure}
\begin{center}
\begin{subfigure}{0.4\textwidth}
\begin{tikzpicture}[line cap=round,line join=round,>=triangle 45,x=1cm,y=1cm,scale=1]
\clip(-2.3,-0.3253850531052378) rectangle (2.3,3.2);
\draw [line width=1pt] (0,0)-- (0,1);
\draw [line width=1pt] (0,0)-- (-2,1);
\draw [line width=1pt] (0,0)-- (2,1);
\draw [line width=1pt] (2,1)-- (2,2);
\draw [line width=1pt] (1,2)-- (0,1);
\draw [line width=1pt] (0,1)-- (-1,2);
\draw [line width=1pt] (-1,2)-- (0.2,3);
\draw [line width=1pt] (-1,2)-- (-0.6,3);
\draw [line width=1pt] (-1,2)-- (-1.4,3);
\draw [line width=1pt] (-1,2)-- (-2.2,3);
\begin{scriptsize}
\draw [fill=blue] (0,0) circle (2.5pt);
\draw [fill=black] (0,1) circle (2.5pt);
\draw [fill=black] (-2,1) circle (2.5pt);
\draw [fill=black] (2,1) circle (2.5pt);
\draw [fill=black] (2,2) circle (2.5pt);
\draw [fill=black] (1,2) circle (2.5pt);
\draw [fill=black] (-1,2) circle (2.5pt);
\draw [fill=black] (-0.6,3) circle (2.5pt);
\draw [fill=black] (0.2,3) circle (2.5pt);
\draw [fill=black] (-1.4,3) circle (2.5pt);
\draw [fill=black] (-2.2,3) circle (2.5pt);
\end{scriptsize}
\end{tikzpicture}
\end{subfigure}
\begin{subfigure}{0.4\textwidth}
\begin{tikzpicture}[line cap=round,line join=round,>=triangle 45,x=1cm,y=1cm,scale=0.5]
\clip(-7.1,-3.59719470642) rectangle (5.2,9.13408476118452);
\draw [line width=1pt,color=qqqqff] (0,0) circle (2cm);
\draw [line width=1pt,color=qqqqff] (-2.5,0) circle (0.5cm);
\draw [line width=1pt,color=qqqqff] (3,0) circle (1cm);
\draw [line width=1pt,color=qqqqff] (0,3.5) circle (1.5cm);
\draw [line width=1pt,color=qqqqff] (4.5,0) circle (0.5cm);
\draw [line width=1pt,color=qqqqff] (1.735047998310382,4.50173042888142) circle (0.5034608577628408cm);
\draw [line width=1pt,color=qqqqff] (-3.4644997335770853,5.500229853788119) circle (2.5004597075762347cm);
\draw [line width=1pt,color=qqqqff] (-1.2347236278383877,7.50792927727393) circle (0.5cm);
\draw [line width=1pt,color=qqqqff] (-4.684896642720836,8.241286190484029) circle (0.5cm);
\draw [line width=1pt,color=qqqqff] (-6.448522608931939,5.186596411458663) circle (0.5cm);
\draw [line width=1pt,color=qqqqff] (-4.088330384609823,2.5653373897239655) circle (0.5cm);
\begin{scriptsize}
\draw [fill=qqqqff] (0,2) circle (5pt);
\draw [fill=black] (0,-2) circle (5pt);
\draw [fill=qqqqff] (2,0) circle (5pt);
\draw [fill=qqqqff] (-2,0) circle (5pt);
\draw [fill=qqqqff] (4,0) circle (5pt);
\draw [fill=qqqqff] (1.299038105676658,4.25) circle (5pt);
\draw [fill=qqqqff] (-1.2990381056766578,4.25) circle (5pt);
\draw [fill=qqqqff] (-1.6062960405770874,7.173363974094498) circle (5pt);
\draw [fill=qqqqff] (-4.481528321182939,7.784513461662735) circle (5pt);
\draw [fill=qqqqff] (-5.9512616612478,5.238860643092491) circle (5pt);
\draw [fill=qqqqff] (-3.9843745392009424,3.0544111900908724) circle (5pt);
\end{scriptsize}
\end{tikzpicture}
\end{subfigure}
\hfill
\begin{subfigure}{0.8\textwidth}
\begin{tikzpicture}[line cap=round,line join=round,>=triangle 45,x=1cm,y=1cm,scale=0.55]
\begin{axis}[x=1cm,y=1cm,
axis lines=middle,
xmin=-0.1,
xmax=21.5,
ymin=-0.1,
ymax=9.5,
xtick={0,1,...,21},
ytick={0,1,...,12},]
\clip(-0.4924160820364542,-1.3673533859272964) rectangle (21.84483488601739,10);
\draw [line width=1.2pt,color=qqqqff] (0,4)-- (1,3);
\draw [line width=1.2pt,color=qqqqff] (1,4)-- (3,2);
\draw [line width=1.2pt,color=qqqqff] (3,5)-- (4,4);
\draw [line width=1.2pt,color=qqqqff] (4,9)-- (5,8);
\draw [line width=1.2pt,color=qqqqff] (5,9)-- (7,7);
\draw [line width=1.2pt,color=qqqqff] (7,8)-- (9,6);
\draw [line width=1.2pt,color=qqqqff] (9,7)-- (11,5);
\draw [line width=1.2pt,color=qqqqff] (11,6)-- (14,3);
\draw [line width=1.2pt,color=qqqqff] (14,4)-- (17,1);
\draw [line width=1.2pt,color=qqqqff] (17,3)-- (18,2);
\draw [line width=1.2pt,color=qqqqff] (18,3)-- (21,0);
\draw [line width=1.2pt,dash pattern=on 2pt off 5pt,color=qqqqff] (4,9)-- (4,4);
\draw [line width=1.2pt,dash pattern=on 2pt off 5pt,color=qqqqff] (3,5)-- (3,2);
\draw [line width=1.2pt,dash pattern=on 2pt off 5pt,color=qqqqff] (1,4)-- (1,3);
\draw [line width=1.2pt,dash pattern=on 2pt off 5pt,color=qqqqff] (0,4)-- (0,0);
\draw [line width=1.2pt,dash pattern=on 2pt off 5pt,color=qqqqff] (5,9)-- (5,8);
\draw [line width=1.2pt,dash pattern=on 2pt off 5pt,color=qqqqff] (7,8)-- (7,7);
\draw [line width=1.2pt,dash pattern=on 2pt off 5pt,color=qqqqff] (9,7)-- (9,6);
\draw [line width=1.2pt,dash pattern=on 2pt off 5pt,color=qqqqff] (11,6)-- (11,5);
\draw [line width=1.2pt,dash pattern=on 2pt off 5pt,color=qqqqff] (14,4)-- (14,3);
\draw [line width=1.2pt,dash pattern=on 2pt off 5pt,color=qqqqff] (17,3)-- (17,1);
\draw [line width=1.2pt,dash pattern=on 2pt off 5pt,color=qqqqff] (18,3)-- (18,2);
\begin{scriptsize}
\draw [fill=xdxdff] (0,4) circle [radius=3pt];
\draw [fill=ududff] (1,4) circle [radius=3pt];
\draw [fill=ududff] (2,3) circle [radius=3pt];
\draw [fill=ududff] (3,5) circle [radius=3pt];
\draw [fill=ududff] (4,9) circle [radius=3pt];
\draw [fill=ududff] (5,9) circle [radius=3pt];
\draw [fill=ududff] (6,8) circle [radius=3pt];
\draw [fill=ududff] (7,8) circle [radius=3pt];
\draw [fill=ududff] (8,7) circle [radius=3pt];
\draw [fill=ududff] (9,7) circle [radius=3pt];
\draw [fill=ududff] (10,6) circle [radius=3pt];
\draw [fill=ududff] (11,6) circle [radius=3pt];
\draw [fill=ududff] (12,5) circle [radius=3pt];
\draw [fill=ududff] (13,4) circle [radius=3pt];
\draw [fill=ududff] (14,4) circle [radius=3pt];
\draw [fill=ududff] (15,3) circle [radius=3pt];
\draw [fill=ududff] (16,2) circle [radius=3pt];
\draw [fill=ududff] (17,3) circle [radius=3pt];
\draw [fill=ududff] (18,3) circle [radius=3pt];
\draw [fill=ududff] (19,2) circle [radius=3pt];
\draw [fill=ududff] (20,1) circle [radius=3pt];
\draw [fill=black] (21,0) circle [radius=3pt];
\end{scriptsize}
\end{axis}
\end{tikzpicture}
\end{subfigure}
\end{center}
\caption{\textit{Top Left:} A plane tree $\tau$. \textit{Top right:} Its looptree $\mathsf{Loop}(\tau)$. \textit{Bottom:} Its process $W(\tau)$.}
\label{discrete_looptree_associated}
\end{figure}

\begin{theorem}
\label{inv_principle_PJG}
Let $(\tau_n)$ be a sequence of plane trees, let $f$ be an excursion, and let $(a_n)$ be a sequence of positive real numbers which tends to $\infty$ such that $|\tau_n|/a_n\longrightarrow0$ and $\left(\un_{[0,1]}(t)a_n^{-1}L_{\lfloor t\#\tau_n\rfloor}(\tau_n)\right)_{t\in[-1,1]}\longrightarrow f$ for the Skorokhod topology on $[-1,1]$. If $f$ is a PJG excursion, then $a_n^{-1}\cdot\mathsf{Loop}(\tau_n)\longrightarrow\lop[f]$ for the pointed GHP topology.
\end{theorem}

\begin{proof}
If $\#\tau_n/a_n\longrightarrow 0$ then $f=0$ because $\sup L(\tau_n)\leq \#\tau_n$, and the desired result holds because the diameter of $\mathsf{Loop}(\tau_n)$ is bounded by $\#\tau_n$. We can now assume $\#\tau_n/a_n\geq \varepsilon$ for some constant $\varepsilon>0$, so that $\#\tau_n\longrightarrow \infty$. Plus, $|\tau_n|/\#\tau_n\longrightarrow 0$ since $|\tau_n|/a_n\longrightarrow 0$. The identity (\ref{modif_luka}) and the inequality (\ref{depth-first_to_contour}) ensure that $a_n^{-1}w(\tau_n)\longrightarrow f$ for the Skorokhod topology on $[-1,1]$. Finally, we apply Theorem~\ref{gh-continu_particular} and the inequality (\ref{GHloop_tau_v}).
\end{proof}

Recall the definitions of $X^{\mathsf{exc},(\alpha)}$ and $\lop[\alpha]$ presented in Section~\ref{stable}, as well as that $X^{\mathsf{exc},(\alpha)}$ is a PJG excursion. We thus retrieve the invariance principle of Curien \& Kortchemski~\cite[Theorem 4.1]{curien2014} stating that if $|\tau_n|/a_n\convd 0$ and $\displaystyle{\left(a_n^{-1}L_{\lfloor t\#\tau_n\rfloor}(\tau_n)\right)_{t\in[0,1]}\convd X^{\mathsf{exc},(\alpha)}}$ for the Skorokhod topology, with some $\alpha\in(1,2)$, then $a_n^{-1}\cdot\mathsf{Loop}(\tau_n)\convd\lop[\alpha]$ for the Gromov--Hausdorff topology. We stress again that our definition of $\mathsf{Loop}(\tau)$ slightly differs from Curien \& Kortchemski's, but the two objects are at Gromov--Hausdorff distance at most $10$, so the result stays true anyway. Moreover, this result stills holds with $\alpha=1$, when $X^{\mathsf{exc},(1)}=1-t$ and $\lop[1]=\mathcal{C}$. In the light of the asymptotic behavior when $\alpha\rightarrow 2$, previously studied in Section~\ref{stable}, it is tempting to guess that if 
$|\tau_n|/a_n\convd 0$ and $\displaystyle{\left(a_n^{-1}L_{\lfloor t\#\tau_n\rfloor}(\tau_n)\right)_{t\in[0,1]}\convd \mathbf{e}}$ for the Skorokhod topology then $2a_n^{-1}\cdot\text{Loop}(\tau_n)\convd\tree[\mathbf{e}]$ for the Gromov--Hausdorff topology, where $\mathbf{e}$ is the standard Brownian excursion. Sadly, the reader could convince oneself this is false by approximating $\mathbf{e}$ by PJG excursions, in a similar fashion as the counterexample given at the beginning of Section~\ref{intro_shuffling}. Here, the problem comes from a lack of exchangeability in the order of siblings inside the plane trees. Nevertheless, by adding a hypothesis that expresses such exchangeability, we will be able to state a general result for convergences of random discrete looptrees.

Let $\Psi=\left(\psi_k\right)_{k\geq 1}$ be a family of permutations $\psi_k$ respectively of $\{1,\ldots,k\}$. The family $\Psi$ induces a bijection, also denoted by $\Psi$ with a slight abuse of notations, from any plane tree $\tau$ to another plane tree $\Psi(\tau)=\left\{\Psi(u)\ :\ u\in\tau\right\}$ by inductively setting $\Psi(\varnothing)=\varnothing$ and $\Psi\left(u*(j)\right)=\Psi\left(u\right)*\left(\psi_{k_u(\tau)}(j)\right)$ for $u\in\tau$ and $1\leq j\leq k_u(\tau)$. Observe that if $u,v\in\tau$ then $|\Psi(u)|=|u|$, $k_{\Psi(u)}(\Psi(\tau))=k_u(\tau)$, and that $u\preceq v$ if and only if $\Psi(u)\preceq \Psi(v)$. Basically, $\Psi(\tau)$ is the plane tree recursively constructed by shuffling the order of siblings of the same parents, so $\mathsf{Loop}(\Psi(\tau))$ can be retrieved by shuffling the joint points between the loops of $\mathsf{Loop}(\tau)$. It is natural to try to compare the looptree shuffled by a family of permutations $\Psi$ with the shuffled looptree $\lopt[w(\tau)]$ induced by a shuffle $\Phi$, as defined via (\ref{dlt_easy}). Recall from (\ref{circle_dist_cano}) that $\delta$ is the pseudo-distance on $[0,1]$ inducing the circle with unit length $(\mathcal{C},\delta)$ by quotient.

\begin{lemma}
\label{lemma_phi_vs_psi}
We assume that the shuffle $\Phi$ is such that the Lebesgue measure on $[0,1]$ is invariant by $\phi_\Delta$, and that $\phi_\Delta(0)=0$ for all $\Delta>0$. Let $\tau$ be a plane tree, let $\Psi$ be a family of permutations, and let $a>0$. Then, it holds that
\[\mathtt{d}_{\mathrm{GHP}}^\bullet\left(\lopt[\frac{1}{a}w(\tau)],\lop[\frac{1}{a}w(\Psi(\tau))]\right)\leq \max_{v\in\tau}\sum_{\substack{u\in\tau,j\geq 1\\u*(j)\preceq v}}\frac{k_{u}(\tau)+1}{a}\delta\left(\phi_{\frac{k_{u}(\tau)+1}{a}}\left(1-\frac{j}{k_{u}(\tau)+1}\right),1-\frac{\psi_{k_{u}(\tau)}(j)}{k_{u}(\tau)+1}\right).\]
\end{lemma}

\begin{proof}
We denote by $c_1$ the exploration by contour of $\tau$ and by $c_2$ the exploration by contour of $\Psi(\tau)$. For all $u\in\tau\backslash\{\varnothing\}$, we define $(2\#\tau-1)r_1(u)$ as the unique integer $i$ such that $c_1(i)=u$ and $|c_1(i)|=|c_1(i-1)|+1$. Similarly, we define $(2\#\tau-1)r_2(u)$ as the unique integer $i$ such that $c_2(i)=\Psi(u)$ and $|c_2(i)|=|c_2(i-1)|+1$. We also set $r_1(\varnothing)=r_2(\varnothing)=0$. The $r_1(u)$ and the $r_2(u)$ for $u\in\tau$ are respectively all the jumps of $w(\tau)$ and $w(\Psi(\tau))$. For all $u\in\tau$ and $y\in[0,1]$,
\begin{align*}
t_1(u,y)&=\inf\left\{t\geq r_1(u)\ :\ w_t(\tau)=w_{r_1(u)-}(\tau)+y\Delta_{r_1(u)}(w(\tau))\right\}\\
t_2(u,y)&=\inf\left\{t\geq r_2(u)\ :\ w_t(\Psi(\tau))=w_{r_2(u)-}(\Psi(\tau))+y\Delta_{r_2(u)}\left(w(\Psi(\tau))\right)\right\}
\end{align*}
are well-defined and the functions $t_1,t_2:\tau\times(0,1]\longrightarrow [0,1)$ are bijective. It is a simple consequence of the fact that $W(\tau)$ is strictly decreasing on the intervals $[i,i+1)$. We use the conventions $t_2(u,\phi_\Delta(x))=t_2(u,1)$ when $\phi_\Delta(x)=0=1\in\mathcal{C}$ and $t_2(u,\phi_\Delta(x))=t_2(u,y)$ when $\phi_\Delta(x)=y\in(0,1)\subset\mathcal{C}$. Let us denote by $[s]_1$ and $[s]_2$ the canonical projections of $s\in[0,1]$ respectively on $\lopt[\frac{1}{a}w(\tau)]$ and $\lop[\frac{1}{a}w(\Psi(\tau))]$. Since $\phi_\Delta:[0,1]\longrightarrow\mathcal{C}$ is always surjective, the set
\[\mathcal{R}=\left\{\left([1]_1,[1]_2\right)\right\}\cup\left\{\left(\left[t_1(u,y)\right]_1,\left[t_2(u,\phi_{\frac{k_u(\tau)+1}{a}}(y))\right]_2\right)\ :\ u\in\tau,y\in[0,1]\right\}\]
is a correspondence between $\lopt[\frac{1}{a}w(\tau)]$ and $\lop[\frac{1}{a}w(\Psi(\tau))]$. We check that its distortion is bounded by $2$ times the right-hand side of the desired inequality thanks to (\ref{x_v}), Lemma~\ref{tri_utile}, basic observations about $\Psi$, and the identities $\phi_\Delta(0)=0$. We define a probability measure $\nu$ on $\lopt[\frac{1}{a}w(\tau)]\times\lop[\frac{1}{a}w(\Psi(\tau))]$ by setting for any bounded measurable function $g$,
\[\int_{\lopt[\frac{1}{a}w(\tau)]\times\lop[\frac{1}{a}w(\Psi(\tau))]}g\dd\nu=\frac{1}{2\#\tau-1}\sum_{u\in\tau}(k_u(\tau)+1)\int_0^1 g\left(\left[t_1(u,y)\right]_1,\left[t_2(u,\phi_{\frac{k_u(\tau)+1}{a}}(y))\right]_2\right)\dd y.\]
Obviously, $\nu(\mathcal{R})=1$. The process $W(\tau)$ takes integer values at integer times and is affine of slope $-1$ on the intervals $[i,i+1)$, so it follows that for all $u\in\tau$, $y\in[0,1)$, and $1\leq j\leq k_u(\tau)+1$, it holds that\[t_1\left(u,\frac{j-y}{k_u(\tau)+1}\right)=t_1\left(u,\frac{j}{k_u(\tau)+1}\right)+\frac{y}{2\#\tau-1}.\]
The function $t_2$ has an analogous behavior. Together with the invariance of the Lebesgue measure on $[0,1]$ by $\phi_\Delta$, this allows us to verify that $\nu$ is a coupling between the probability measures of $\lopt[\frac{1}{a}w(\tau)]$ and $\lop[\frac{1}{a}w(\Psi(\tau))]$.
\end{proof}

We say a random plane tree $\tau$ is called \emph{exchangeable} if for any deterministic family $\Psi$ of permutations, $\Psi(\tau)$ has the same distribution as $\tau$. This will be the main assumption of our invariance principle for random discrete looptrees. For example, it is satisfied by uniform random trees or Galton--Watson trees conditioned on an event or a quantity that does not depend on the plane order: the number of vertices, the number of leaves, the height, the degree sequence, the generation size sequence, the existence of a chain with prescribed degrees, etc...

\begin{theorem}
\label{inv_principle_exch}
Let $(\tau_n)$ be a sequence of random exchangeable plane trees, let $X$ be a random excursion, and let $(a_n)$ be a sequence of positive real numbers with $a_n\rightarrow\infty$. If the convergence $|\tau_n|/a_n\convd 0$ holds in distribution and if the convergence $\left(\un_{[0,1]}(t)a_n^{-1}L_{\lfloor t\#\tau_n\rfloor}(\tau_n)\right)_{t\in[-1,1]}\convd X$ holds in distribution for the Skorokhod topology on $[-1,1]$, then the convergence $a_n^{-1}\cdot\mathsf{Loop}(\tau_n)\convd\frac{1}{2}\cdot\vern[X]$ holds in distribution for the pointed GHP topology.
\end{theorem}

In particular, this theorem gives scaling limits for $\text{Loop}(\tau_n)$ when $\tau_n$ is a large critical Galton--Watson tree conditioned to have $n$ vertices with offspring distribution in the domain of attraction of the Gaussian law but with infinite variance, because the scaling limit of the Lukasiewicz walk and $|\tau_n|$ are known in this case: see Duquesne~\cite{duquesne_contour_stable}. This was conjectured by Curien \& Kortchemski~\cite{curien2014} and later proved by Kortchemski \& Richier~\cite{KR2020} with a spinal decomposition argument. As scaling limits of Lukasiewicz walks of uniform random trees with prescribed degree sequence are given by \cite[Theorem 16.23]{kallenberg2002foundations}, our theorem yields another proof for scaling limits of uniform random looptrees with prescribed cycle lengths obtained by Marzouk~\cite{marzouk_similaire}, but only when the height of the associated tree is negligible against $a_n$.

\begin{proof}
Thanks to Skorokhod's representation theorem, we can assume that all the convergences in the hypothesis happen almost surely. We can adapt the argument in the proof of Theorem~\ref{inv_principle_PJG} to prove that almost surely, $\mathtt{d}_{\mathrm{GHP}}^\bullet\left(a_n^{-1}\cdot\mathsf{Loop}(\tau_n),\lop[a_n^{-1}w(\tau_n)]\right)\longrightarrow 0$ and $a_n^{-1}w(\tau_n)\longrightarrow X$ for the Skorokhod topology on $[-1,1]$. Thus, we only need to find the limit in distribution of $\lop[a_n^{-1}w(\tau_n)]$. The strategy of the proof is to construct a family of permutations $\Psi$ and a shuffle $\Phi$ such that we can show with Lemma~\ref{lemma_phi_vs_psi} that $\lop[a_n^{-1}w(\Psi(\tau_n))]$ and $\lopt[a_n^{-1}w(\tau_n)]$ are close. Then, as Theorem~\ref{intro_uni} ensures that $\vernt[a_n^{-1}w(\tau_n)]=2\cdot\lopt[a_n^{-1}w(\tau_n)]$ because $w(\tau_n)$ is PJG, we would find the desired limit thanks to Theorem~\ref{gh-continu} and the exchangeability of $\tau_n$. However, we have two technical difficulties to overcome. On the one hand, if $(k+1)/a_n=(l+1)/a_m=\Delta$, then $\phi_\Delta$ would have to verify some relations with both $\psi_k$ and $\psi_l$, which could complexify its construction. On the other hand, if the set of the $(k+1)/a_n$ is dense, it may be impossible for $\Phi$ to be continuous at some points, and so to almost surely verify the condition $B_X\cap\mathrm{B}(\Phi)=\emptyset$ required by Theorem~\ref{gh-continu}. 

By induction, we can choose a sequence $(b_n)_{n\geq 1}$ such that $b_1=a_1$ and $|b_{n+1}-a_{n+1}|\leq 1$ but $b_{n+1}\notin\bigcup_{m=1}^n b_m\mathbb{Q}$ for all $n\geq 1$. Therefore, we can assume that $a_n/a_m\in\mathbb{Q}\Longrightarrow n=m$ without loss of generality, which resolves the first difficulty. It is not difficult to express $B_X$ as the set of values of a sequence of $\sigma(X)$-measurable random variables because $X$ is almost surely c\`adl\`ag, so that the set
\[E=\left\{x\in\mathbb{R}\ :\ \P(x\in B_X)+\P(\exists \Delta>0,\ (\Delta,x)\in B_X)>0\right\}\]
is countable. Let us fix $\varepsilon\in(0,\infty)\backslash E$ for the time being, and recall the operator $J^\varepsilon$ from (\ref{Jeps_def}). For all $\eta\in(0,\varepsilon)\backslash E$, we are able to prove the convergence $||J^\eta (a_n^{-1}w(\tau_n))-J^\varepsilon(a_n^{-1}w(\tau_n))||_\infty\longrightarrow ||J^\eta X-J^\varepsilon X||_\infty$ holds almost surely with Proposition~\ref{cv_J} $(i)$ and $(ii)$, since neither $\varepsilon$ nor $\eta$ are the height of some jump of $X$ by definition of $E$. The set $E$ is countable so we can give ourselves a deterministic strictly decreasing sequence $(\eta_p')_{p\geq 1}$ of points of $(0,\varepsilon)\backslash E$ that tends to $0$. Then, there exists a deterministic strictly increasing sequence $(m_p)_{p\geq 1}$ of integers such that \[\P\left(\exists n\geq m_p,\ \left|||J^{\eta_p'} (a_n^{-1}w(\tau_n))-J^\varepsilon(a_n^{-1}w(\tau_n))||_\infty - ||J^{\eta_p'} X-J^\varepsilon X||_\infty\right|\geq\frac{1}{2^p}\right)\leq\frac{1}{2^p}\]for all $p\geq 1$. We construct another deterministic sequence $(\eta_n)$ by setting $\eta_n=\varepsilon$ if $n<m_1$ and $\eta_n=\eta_p'$ if $m_p\leq n< m_{p+1}$ with $p\geq 1$. This sequence tends to $0$ so $||JX-J^{\eta_n}X||_\infty\longrightarrow 0$ almost surely according to Proposition~\ref{cv_J} $(iii)$, and the Borel-Cantelli lemma ensures that almost surely,
\begin{equation}
\label{Jnu_Jeta}
||J^{\eta_n}(a_n^{-1}w(\tau_n))-J^{\varepsilon}(a_n^{-1}w(\tau_n))||_\infty\longrightarrow ||JX-J^{\varepsilon}X||_\infty.
\end{equation}
Plus, the only limit point of the set $D=\{(k+1)/a_n\ :\ k,n\geq 1\text{ such that }k+1<a_n\eta_n\}$ is $0$ because $\eta_n\longrightarrow 0$, and that will resolve the second difficulty.

Let $k,n\geq 1$. We define a deterministic permutation $\psi_k$ of $\{1,\ldots,k\}$ by setting
\[\psi_k(l)=\begin{cases}
			\frac{k+1-l+1}{2}&\text{ if }k+1-l\text{ is odd},\\
			k+1-\frac{k+1-l}{2}&\text{ if }k+1-l\text{ is even},
		\end{cases}\]for all $1\leq l\leq k$. Since $E$ is countable, it is possible to choose a deterministic $\alpha\in(1/2,1)$ such that $\alpha^m(1+\alpha)/2\notin E$ and $\alpha^m\notin E$ for all $m\geq 1$ and such that $\max(1-\alpha,1/\alpha-1)\leq \min\left((k+1)/a_n,1/(k+1)\right)$. For all $m\geq 0$, we write $I_{2m}=\left(\alpha^m(1+\alpha)/2,\alpha^m\right]$ and $I_{2m+1}=\left(\alpha^{m+1},\alpha^m(1+\alpha)/2\right]$. Plus, we also set for all $x\in(0,1]$,
\[\varphi(x)=\begin{cases}
	x-\frac{\alpha^m}{2}&\text{ if }x\in I_{2m}\text{ with an integer }m,\\
	1-x+\frac{\alpha^{m+1}}{2}&\text{ if }x\in I_{2m+1}\text{ with an integer }m.
	\end{cases}\]
Note that there is at most one $1\leq j\leq k$ such that $j/(k+1)\in I_{2m}\cup I_{2m+1}$ because $1-\alpha\leq 1/(k+1)$. We denote by $b$ the parity function defined on the set of integers by $b(2m)=0$ and $b(2m+1)=1$. We define $\phi_{\frac{k+1}{a_n}}:[0,1]\to \mathcal{C}$ by setting
\[\phi_{\frac{k+1}{a_n}}(x)=\begin{cases}
	1-\varphi(x)&\text{ if }\exists j\in\{1,\ldots,k\}\text{ such that }j/(k+1)\in I_{2m+1-b(j)},\\
	\varphi(x)&\text{ otherwise },
	\end{cases}\]
for all $x\in I_{2m}\cup I_{2m+1}$ with an integer $m$, and $\phi_{\frac{k+1}{a_n}}(0)=0$. We emphasize that the notation $\phi_{\frac{k+1}{a_n}}$ does not lead to any inconsistency. Indeed, if $(k+1)/a_n=(l+1)/a_m$ then $a_n/a_m\in\mathbb{Q}$, so $n=m$ and $k=l$. It is clear that $\phi_{\frac{k+1}{a_n}}$ is c\`agl\`ad and surjective. The intervals $(I_m)_{m\geq 0}$ partition $(0,1]$ and $\phi_{\frac{k+1}{a_n}}$ is affine with slope $\pm 1$ on each one of them, so it preserves the Lebesgue measure on $[0,1]$. Finally, we check that for all $k,n\geq 1$,
\begin{equation}
\label{phi_vs_psi}
\sup_{x\in(0,1]}\left|\frac{2}{x}\delta\left(0,\phi_{\frac{k+1}{a_n}}(x)\right)-1\right|\leq \min\left(\frac{k+1}{a_n},1\right)\quad\text{ and }\quad\sup_{1\leq j\leq k}\delta\left(\phi_{\frac{k+1}{a_n}}\left(1-\frac{j}{k+1}\right),1-\frac{\psi_k(j)}{k+1}\right)\leq \frac{2}{k+1}.
\end{equation}

For all $n\geq 1$, we define a deterministic family of permutation $\Psi^n=(\psi_k^n)_{k\geq 1}$ by setting $\psi_k^n=\psi_k$ if $k+1< a_n\eta_n$ and $\psi_k^n=\mathsf{id}_{\{1,\ldots,k\}}$ if $k+1\geq a_n\eta_n$. Recall that the only limit point of the set $D$ is $0$, $D$ is bounded from above by $\varepsilon$, and $E$ is countable. Hence, we can give ourselves a deterministic strictly decreasing sequence $(\beta_p)_{p\geq 1}$ of points of $(0,\infty)\backslash E$ which tends to $0$, such that $\beta_1=\varepsilon$, and such that there is exactly one element of $D$ inside $[\beta_{p+1},\beta_p)$ for all $p\geq 1$. Then, we define a deterministic shuffle $\Phi^\varepsilon:\Delta>0\longmapsto \phi_\Delta^\varepsilon$ by setting $\phi_\Delta^\varepsilon(x)=x$ for all $x\in[0,1]$ if $\Delta\geq\beta_1=\varepsilon$, and by setting $\phi_\Delta^\varepsilon=\phi_{\frac{k+1}{a_n}}$ if $\Delta\in[\beta_{p+1},\beta_p)$ with $p\geq 1$, where $(k+1)/a_n$ is the unique point of $D$ also inside $[\beta_{p+1},\beta_p)$. This indeed verifies the conditions of Definition~\ref{shuffle}. By definition of $E$ and of the $\phi_{\frac{k+1}{a_n}}$, it holds $B_X\cap\mathrm{B}(\Phi^\varepsilon)=\emptyset$ almost surely. Theorems~\ref{intro_uni} and \ref{gh-continu} yield that almost surely, $2\cdot\lopt[a_n^{-1}w(\tau_n)]=\vernt[a_n^{-1}w(\tau_n)]\longrightarrow\vernt[f]$ for the pointed GHP topology, because $w(\tau_n)$ is PJG. Moreover, the Lebesgue measure on $[0,1]$ is invariant by $\phi_\Delta^\varepsilon$ and $\phi_\Delta^\varepsilon(0)=0$ for all $\Delta>0$, and
\begin{equation}
\label{K_for_discrete_looptree}
\sup_{\Delta>0}\sup_{x\in[0,1]}\left|\frac{2}{x}\delta\left(0,\phi_\Delta^\varepsilon(x)\right)-1\right|\leq 1.
\end{equation}
Thus, we can apply Lemma~\ref{lemma_phi_vs_psi}. Recall the constructions of $\Phi^\varepsilon$ and $\Psi^n$, and that $\eta_n\leq\varepsilon$. Let $u*(j)\in\tau_n$ with $j\geq 1$, we want to bound the corresponding term in the right-hand-side of the inequality provided by Lemma~\ref{lemma_phi_vs_psi}. If $k_{u}(\tau_n)+1\geq a_n\varepsilon$, then this term is equal to $0$ because $\phi_\Delta^\varepsilon(x)=x$ when $\Delta\geq\varepsilon$ and $\psi_k^n(j)=j$ when $k+1\geq a_n\eta_n$. If $a_n\eta_n\leq k_{u}(\tau_n)+1< a_n\varepsilon$, then we control the term with the inequality (\ref{K_for_discrete_looptree}). If $k_u(\tau_n)+1<a_n\eta_n$, then we use the inequality (\ref{phi_vs_psi}). Eventually, we obtain that
\begin{align*}
\mathtt{d}_{\mathrm{GHP}}^\bullet\left(\lop[\frac{1}{a_n}w(\Psi^n(\tau_n))],\lopt[\frac{1}{a_n}w(\tau_n)]\right)&\leq \max_{v\in\tau_n}\sum_{\substack{u\in\tau_n,j\geq 1\\ u*(j)\preceq v}}\left(0+\un_{\left\{\eta_n\leq \frac{k_u(\tau_n)+1}{a_n}<\varepsilon\right\}} 2\frac{k_u(\tau_n)+1-j}{a_n}+\frac{2}{a_n}\right)\\
&\leq \frac{2|\tau_n|}{a_n}+2\sup_{t\in[0,1]}\sum_{\substack{s\in[0,1]\\ \eta_n\leq \Delta_s(a_n^{-1}w(\tau_n))<\varepsilon}}x_s^t\left(\frac{1}{a_n}w(\tau_n)\right)\\
&\leq \frac{2|\tau_n|}{a_n}+2||J^{\eta_n}(a_n^{-1}w(\tau_n))-J^\varepsilon(a_n^{-1}w(\tau_n))||_\infty.
\end{align*}
The second inequality comes from the identity (\ref{x_v}), and the third follows from the expression (\ref{Jeps_def}) of $J^\varepsilon$. Then, the almost sure convergences $|\tau_n|/a_n\longrightarrow 0$, $2\cdot\lopt[a_n^{-1}w(\tau_n)]\longrightarrow\vernt[X]$, and (\ref{Jnu_Jeta}) yield that almost surely,
\begin{equation}
\label{GH3}
\limsup_{n\rightarrow\infty}\mathtt{d}_{\mathrm{GHP}}^\bullet\left(\lop[\frac{1}{a_n}w(\Psi^n(\tau_n))],\frac{1}{2}\cdot\vernt[X]\right)\leq2\left|\left|JX-J^{\varepsilon}X\right|\right|_\infty.
\end{equation}

We remind that $\phi_\Delta^\varepsilon(x)=x$ for all $x\in[0,1]$ whenever $\Delta\geq\varepsilon$. The identity (\ref{dl_Jeps}) from Theorem~\ref{d_J} then implies that $\dl[J^\varepsilon X]=\dlt[J^\varepsilon X]$. Let us set $Y=X-JX+J^{\varepsilon}X$. We know that $JY=J^{\varepsilon}Y$ from the point $(iii)$ of Theorem~\ref{J_prop}, thus $\dtree[Y-JY]=\dtree[X-JX]$ and $\dlt[JY]=\dl[JY]=\dl[J^\varepsilon X]$. Therefore, Theorem~\ref{d_J} gives $\dtree[Y]=\dtree[X]$ and $\dv[Y]=\dvt[Y]$. Eventually, we find
\[||\dvt[X]-\dv[Y]||_\infty= 2||\dlt[X]-\dlt[J^{\varepsilon}X]||_\infty\leq 4||JX-J^{\varepsilon}X||_\infty\quad\text{ and }\quad||\dv[X]-\dv[Y]||_\infty=2||\dl[X]-\dl[J^{\varepsilon}X]||_\infty\leq 4||JX-J^{\varepsilon}X||_\infty,\]
thanks to Corollary~\ref{propo}, which is properly used with the constant $K_{\Phi^\varepsilon}=1$ according to (\ref{K_for_discrete_looptree}). Then, we combine (\ref{GH3}) and (\ref{GHP_lips}) to obtain $\limsup\mathtt{d}_{\mathsf{GHP}}^\bullet\left(\lop[\frac{1}{a_n}w(\Psi^n(\tau_n))],\frac{1}{2}\cdot\vern[X]\right)\leq 10\left|\left|JX-J^{\varepsilon}X\right|\right|_\infty$ almost surely. The random plane tree $\tau_n$ is exchangeable so it has the same distribution as $\Psi^n(\tau_n)$. Furthermore, neither $\lop[a_n^{-1}w(\tau_n)]$ nor $\vern[X]$ depend on $\varepsilon$. Plus, $\varepsilon$ may be chosen arbitrarily small because $E$ is countable. Hence, we deduce that $\lop[a_n^{-1}w(\tau_n)]\convd\frac{1}{2}\cdot\vern[X]$ for the pointed GHP topology thanks to Proposition~\ref{cv_J} $(iii)$.
\end{proof}

\begin{remark}
In Theorem~\ref{inv_principle_PJG}, the assumption $|\tau_n|/a_n\longrightarrow 0$ may be removed if one directly knows the limit of $a_n^{-1}w(\tau_n)$. However, doing the same for Theorem~\ref{inv_principle_exch} is not possible because we also used $|\tau_n|/a_n\convd 0$ to obtain the estimate (\ref{GH3}). In fact, Kortchemski \& Richier~\cite[Theorem 2]{KR2020} and Marzouk~\cite[Theorem 7.5]{marzouk_similaire} provide examples that confirm that the latter assumption is not superfluous. As shown by Marzouk, when $|\tau_n|/a_n$ is not negligible, an unbalance between $\dl$ and $\dtree$ may occur so that the scaling limit of the looptrees is the quotient metric space induced by the pseudo-distance $a\dtree+2\dl$, with some $a>1$.\cq
\end{remark}

%
%

\begin{acks}[Acknowledgments]
This work was written under the successive supervision of Nicolas Broutin and Thomas Duquesne, whom the author thanks warmly. The author would like to express his gratitude to an anonymous referee for a detailed report that led to substantial improvements and simplifications of the presentation and proofs.
\end{acks}
%


\bibliographystyle{imsart-number} 
\bibliography{looptree_final}       

\begin{thebibliography}{43}

\bibitem{abraham_surveymaps}
\begin{barticle}[author]
\bauthor{\bsnm{Abraham},~\bfnm{Céline}\binits{C.}},
  \bauthor{\bsnm{Bettinelli},~\bfnm{Jérémie}\binits{J.}},
  \bauthor{\bsnm{Collet},~\bfnm{Gwendal}\binits{G.}} \AND
  \bauthor{\bsnm{Kortchemski},~\bfnm{Igor}\binits{I.}}
(\byear{2015}).
\btitle{Random maps}.
\bjournal{ESAIM: Proceedings and Surveys}
\bvolume{51}
\bpages{133-149}.
\bdoi{10.1051/proc/201551008}
\end{barticle}
\endbibitem

\bibitem{GHP_polish}
\begin{barticle}[author]
\bauthor{\bsnm{Abraham},~\bfnm{Romain}\binits{R.}},
  \bauthor{\bsnm{Delmas},~\bfnm{Jean-François}\binits{J.-F.}} \AND
  \bauthor{\bsnm{Hoscheit},~\bfnm{Patrick}\binits{P.}}
(\byear{2013}).
\btitle{A note on the {G}romov-{H}ausdorff-{P}rokhorov distance between
  (locally) compact metric measure spaces}.
\bjournal{Electronic Journal of Probability}
\bvolume{18}
\bpages{1 -- 21}.
\bdoi{10.1214/EJP.v18-2116}
\end{barticle}
\endbibitem

\bibitem{aldousI}
\begin{barticle}[author]
\bauthor{\bsnm{Aldous},~\bfnm{David}\binits{D.}}
(\byear{1991}).
\btitle{The {C}ontinuum {R}andom {T}ree {I}}.
\bjournal{The Annals of Probability}
\bvolume{19}
\bpages{1--28}.
\bmrnumber{1085326}
\end{barticle}
\endbibitem

\bibitem{aldous1993}
\begin{barticle}[author]
\bauthor{\bsnm{Aldous},~\bfnm{David}\binits{D.}}
(\byear{1993}).
\btitle{The {C}ontinuum {R}andom {T}ree {III}}.
\bjournal{The Annals of Probability}
\bvolume{21}
\bpages{248--289}.
\bdoi{10.1214/aop/1176989404}
\end{barticle}
\endbibitem

\bibitem{p-mapping}
\begin{barticle}[author]
\bauthor{\bsnm{Aldous},~\bfnm{David}\binits{D.}},
  \bauthor{\bsnm{Miermont},~\bfnm{Gregory}\binits{G.}} \AND
  \bauthor{\bsnm{Pitman},~\bfnm{Jim}\binits{J.}}
(\byear{2004}).
\btitle{Brownian {B}ridge {A}symptotics for {R}andom $p$-{M}appings}.
\bjournal{Electronic Journal of Probability}
\bvolume{9}
\bpages{37 -- 56}.
\bdoi{10.1214/EJP.v9-186}
\end{barticle}
\endbibitem

\bibitem{UIPT}
\begin{barticle}[author]
\bauthor{\bsnm{Angel},~\bfnm{Omer}\binits{O.}} \AND
  \bauthor{\bsnm{Schramm},~\bfnm{Oded}\binits{O.}}
(\byear{2003}).
\btitle{Uniform {I}nfinite {P}lanar {T}riangulations}.
\bjournal{Communications in Mathematical Physics}
\bvolume{241}
\bpages{191-213}.
\bdoi{10.1007/s00220-003-0932-3}
\end{barticle}
\endbibitem

\bibitem{barabasi}
\begin{barticle}[author]
\bauthor{\bsnm{Barabasi},~\bfnm{Albert-Laszlo}\binits{A.-L.}} \AND
  \bauthor{\bsnm{Albert},~\bfnm{Rita}\binits{R.}}
(\byear{1999}).
\btitle{{Emergence of Scaling in Random Networks}}.
\bjournal{Science (New York, N.Y.)}
\bvolume{286}
\bpages{509-512}.
\bdoi{10.1126/science.286.5439.509}
\end{barticle}
\endbibitem

\bibitem{taylor_1998}
\begin{bbook}[author]
\bauthor{\bsnm{Bertoin},~\bfnm{Jean}\binits{J.}}
(\byear{1996}).
\btitle{L{\'e}vy Processes}.
\bseries{Volume 121 of Cambridge Tracts in Mathematics}.
\bpublisher{Cambridge University Press}.
\end{bbook}
\endbibitem

\bibitem{billingsley2013convergence}
\begin{bbook}[author]
\bauthor{\bsnm{Billingsley},~\bfnm{Patrick}\binits{P.}}
(\byear{1999}).
\btitle{Convergence of probability measures},
\bedition{second} ed.
\bseries{Wiley Series in Probability and Statistics: Probability and
  Statistics}.
\bpublisher{John Wiley \& Sons Inc.}, \baddress{New York}.
\bnote{A Wiley-Interscience Publication}.
\bmrnumber{MR1700749 (2000e:60008)}
\end{bbook}
\endbibitem

\bibitem{blanc-renaudie}
\begin{barticle}[author]
\bauthor{\bsnm{Blanc-Renaudie},~\bfnm{Arthur}\binits{A.}}
(\byear{2022}).
\btitle{{Looptree, Fennec, and Snake of ICRT}}.
\bjournal{Preprint available on arXiv}.
\bnote{arXiv:2203.10891}.
\end{barticle}
\endbibitem

\bibitem{bollobas}
\begin{barticle}[author]
\bauthor{\bsnm{Bollobás},~\bfnm{Béla}\binits{B.}},
  \bauthor{\bsnm{Riordan},~\bfnm{Oliver}\binits{O.}},
  \bauthor{\bsnm{Spencer},~\bfnm{Joel}\binits{J.}} \AND
  \bauthor{\bsnm{Tusnády},~\bfnm{Gábor}\binits{G.}}
(\byear{2001}).
\btitle{The degree sequence of a scale-free random graph process}.
\bjournal{Random Structures Algorithms}
\bvolume{18}
\bpages{279-290}.
\bdoi{10.1002/rsa.1009}
\end{barticle}
\endbibitem

\bibitem{BDFG}
\begin{barticle}[author]
\bauthor{\bsnm{Bouttier},~\bfnm{J{\'e}r{\'e}mie}\binits{J.}},
  \bauthor{\bsnm{Francesco},~\bfnm{Philippe~Di}\binits{P.~D.}} \AND
  \bauthor{\bsnm{Guitter},~\bfnm{Emmanuel}\binits{E.}}
(\byear{2004}).
\btitle{Planar maps as labeled mobiles}.
\bjournal{Electronic Journal of Combinatorics}
\bvolume{11}
\bpages{Research Paper 69, 27}.
\bmrnumber{2097335}
\end{barticle}
\endbibitem

\bibitem{arcwise}
\begin{barticle}[author]
\bauthor{\bsnm{Börger},~\bfnm{Reinhard}\binits{R.}}
(\byear{1992}).
\btitle{How to make a path injective}.
\bjournal{Recent developments of general topology and its applications :
  International Conference in Memory of Felix Hausdorff (1868-1942)}
\bpages{57-59}.
\end{barticle}
\endbibitem

\bibitem{Chaumont1997ExcursionNM}
\begin{barticle}[author]
\bauthor{\bsnm{Chaumont},~\bfnm{Lo\"ic}\binits{L.}}
(\byear{1997}).
\btitle{Excursion normalis\'ee, m\'eandre et pont pour les processus de
  {L}\'evy stables}.
\bjournal{Bulletin des Sciences Mathématiques, 121}
\bpages{377-403}.
\end{barticle}
\endbibitem

\bibitem{chiswell2001introduction}
\begin{bbook}[author]
\bauthor{\bsnm{Chiswell},~\bfnm{Ian}\binits{I.}}
(\byear{2001}).
\btitle{{Introduction To Lambda-Trees}}.
\bpublisher{World Scientific Publishing Company}.
\end{bbook}
\endbibitem

\bibitem{loopLPAM}
\begin{barticle}[author]
\bauthor{\bsnm{Curien},~\bfnm{Nicolas}\binits{N.}},
  \bauthor{\bsnm{Duquesne},~\bfnm{Thomas}\binits{T.}},
  \bauthor{\bsnm{Kortchemski},~\bfnm{Igor}\binits{I.}} \AND
  \bauthor{\bsnm{Manolescu},~\bfnm{Ioan}\binits{I.}}
(\byear{2015}).
\btitle{Scaling limits and influence of the seed graph in preferential
  attachment trees}.
\bjournal{Journal de l{\textquoteright}\'Ecole polytechnique {\textemdash}
  Math\'ematiques}
\bvolume{2}
\bpages{1--34}.
\bdoi{10.5802/jep.15}
\end{barticle}
\endbibitem

\bibitem{haas}
\begin{barticle}[author]
\bauthor{\bsnm{Curien},~\bfnm{Nicolas}\binits{N.}},
  \bauthor{\bsnm{Haas},~\bfnm{Bénédicte}\binits{B.}} \AND
  \bauthor{\bsnm{Kortchemski},~\bfnm{Igor}\binits{I.}}
(\byear{2015}).
\btitle{{The CRT is the scaling limit of random dissections}}.
\bjournal{Random Structures \& Algorithms}
\bvolume{47}
\bpages{304-327}.
\bdoi{https://doi.org/10.1002/rsa.20554}
\end{barticle}
\endbibitem

\bibitem{curien2014}
\begin{barticle}[author]
\bauthor{\bsnm{Curien},~\bfnm{Nicolas}\binits{N.}} \AND
  \bauthor{\bsnm{Kortchemski},~\bfnm{Igor}\binits{I.}}
(\byear{2014}).
\btitle{Random stable looptrees}.
\bjournal{Electronic Journal of Probability}
\bvolume{19}
\bpages{1-35}.
\bdoi{10.1214/EJP.v19-2732}
\end{barticle}
\endbibitem

\bibitem{perco_curien}
\begin{barticle}[author]
\bauthor{\bsnm{Curien},~\bfnm{Nicolas}\binits{N.}} \AND
  \bauthor{\bsnm{Kortchemski},~\bfnm{Igor}\binits{I.}}
(\byear{2015}).
\btitle{Percolation on random triangulations and stable looptrees}.
\bjournal{Probability Theory and Related Fields}
\bvolume{163}
\bpages{303-3037}.
\bdoi{10.1007/s00440-014-0593-5}
\end{barticle}
\endbibitem

\bibitem{burago}
\begin{bbook}[author]
\bauthor{\bsnm{Dmitri~Burago},~\bfnm{Sergei~Ivanov}\binits{S.~I.}
  \bsuffix{Yuri~Burago}}
(\byear{2001}).
\btitle{A course in metric geometry}.
\bseries{Graduate studies in Mathematics}.
\bpublisher{American Mathematical Society}, \baddress{Providence (R.I.)}.
\end{bbook}
\endbibitem

\bibitem{duquesne_contour_stable}
\begin{barticle}[author]
\bauthor{\bsnm{Duquesne},~\bfnm{Thomas}\binits{T.}}
(\byear{2003}).
\btitle{{A limit theorem for the contour process of conditioned Galton--Watson
  trees}}.
\bjournal{The Annals of Probability}
\bvolume{31}
\bpages{996 -- 1027}.
\bdoi{10.1214/aop/1048516543}
\end{barticle}
\endbibitem

\bibitem{duquesne_coding}
\begin{barticle}[author]
\bauthor{\bsnm{Duquesne},~\bfnm{Thomas}\binits{T.}}
(\byear{2006}).
\btitle{The coding of compact real trees by real valued functions}.
\bjournal{Preprint available on arXiv}.
\bnote{arXiv:math/0604106}.
\end{barticle}
\endbibitem

\bibitem{levytree_DLG}
\begin{bbook}[author]
\bauthor{\bsnm{Duquesne},~\bfnm{Thomas}\binits{T.}} \AND
  \bauthor{\bsnm{Le~Gall},~\bfnm{Jean-Fran\c{c}ois}\binits{J.-F.}}
(\byear{2002}).
\btitle{Random trees, {L\'evy} processes and spatial branching processes}.
\bseries{Ast\'erisque}
\bvolume{281}.
\bpublisher{Soci\'et\'e math\'ematique de France}.
\bmrnumber{1954248}
\end{bbook}
\endbibitem

\bibitem{evans}
\begin{bbook}[author]
\bauthor{\bsnm{Evans},~\bfnm{Steven}\binits{S.}}
(\byear{2007}).
\btitle{Probability and Real Trees: {\'E}cole d'{\'E}t{\'e} de Probabilit{\'e}s
  de Saint-Flour XXXV-2005}.
\bseries{Lecture Notes in Mathematics}.
\bpublisher{Springer Berlin Heidelberg}.
\end{bbook}
\endbibitem

\bibitem{jacod}
\begin{bbook}[author]
\bauthor{\bsnm{Jacod},~\bfnm{Jean}\binits{J.}} \AND
  \bauthor{\bsnm{Shiryaev},~\bfnm{Albert~N.}\binits{A.~N.}}
(\byear{2003}).
\btitle{Limit Theorems for Stochastic Processes},
\bedition{second} ed.
\bseries{Grundlehren der mathematischen Wissenschaften}
\bvolume{288}.
\bpublisher{Springer Berlin Heidelberg}, \baddress{Berlin, Heidelberg}.
\end{bbook}
\endbibitem

\bibitem{stefansson}
\begin{barticle}[author]
\bauthor{\bsnm{Janson},~\bfnm{Svante}\binits{S.}} \AND
  \bauthor{\bsnm{Stef\'{a}nsson},~\bfnm{Sigurdur~\"{O}rn}\binits{S.~O.}}
(\byear{2015}).
\btitle{Scaling limits of random planar maps with a unique large face}.
\bjournal{The Annals of Probability}
\bvolume{43}
\bpages{1045--1081}.
\bdoi{10.1214/13-AOP871}
\bmrnumber{3342658}
\end{barticle}
\endbibitem

\bibitem{kallenberg2002foundations}
\begin{bbook}[author]
\bauthor{\bsnm{Kallenberg},~\bfnm{Olav}\binits{O.}}
(\byear{2002}).
\btitle{Foundations of Modern Probability},
\bedition{second} ed.
\bseries{Probability and its Applications (New York)}.
\bpublisher{Springer-Verlag, New York}.
\bdoi{10.1007/978-1-4757-4015-8}
\bmrnumber{1876169 (2002m:60002)}
\end{bbook}
\endbibitem

\bibitem{GHP_correspondence}
\begin{barticle}[author]
\bauthor{\bsnm{Khezeli},~\bfnm{Ali}\binits{A.}}
(\byear{2020}).
\btitle{{Metrization of the Gromov–Hausdorff (-Prokhorov) topology for
  boundedly-compact metric spaces}}.
\bjournal{Stochastic Processes and their Applications}
\bvolume{130}
\bpages{3842-3864}.
\bdoi{https://doi.org/10.1016/j.spa.2019.11.001}
\end{barticle}
\endbibitem

\bibitem{lamination}
\begin{barticle}[author]
\bauthor{\bsnm{Kortchemski},~\bfnm{Igor}\binits{I.}}
(\byear{2014}).
\btitle{Random stable laminations of the disk}.
\bjournal{The Annals of Probability}
\bvolume{42}
\bpages{725 -- 759}.
\bdoi{10.1214/12-AOP799}
\end{barticle}
\endbibitem

\bibitem{KR2020}
\begin{barticle}[author]
\bauthor{\bsnm{Kortchemski},~\bfnm{Igor}\binits{I.}} \AND
  \bauthor{\bsnm{Richier},~\bfnm{Lo{\"\i}c}\binits{L.}}
(\byear{2020}).
\btitle{The boundary of random planar maps via looptrees}.
\bjournal{Annales de la Facult\'e des sciences de Toulouse : Math\'ematiques}
\bvolume{29}
\bpages{391--430}.
\bdoi{10.5802/afst.1636}
\end{barticle}
\endbibitem

\bibitem{stablemoments1}
\begin{barticle}[author]
\bauthor{\bsnm{Kuruoglu},~\bfnm{Ercan}\binits{E.}}
(\byear{2001}).
\btitle{Density parameter estimation of skewed alpha-stable distribution}.
\bjournal{IEEE Transactions on Signal Processing}
\bvolume{49}
\bpages{2192 - 2201}.
\bdoi{10.1109/78.950775}
\end{barticle}
\endbibitem

\bibitem{legall_trees}
\begin{barticle}[author]
\bauthor{\bsnm{Le~Gall},~\bfnm{Jean-François}\binits{J.-F.}}
(\byear{2005}).
\btitle{{Random trees and applications}}.
\bjournal{Probability Surveys}
\bvolume{2}
\bpages{245 -- 311}.
\bdoi{10.1214/154957805100000140}
\end{barticle}
\endbibitem

\bibitem{legall1998}
\begin{barticle}[author]
\bauthor{\bsnm{Le~Gall},~\bfnm{Jean-Francois}\binits{J.-F.}} \AND
  \bauthor{\bsnm{Le~Jan},~\bfnm{Yves}\binits{Y.}}
(\byear{1998}).
\btitle{Branching processes in {L}\'evy processes: the exploration process}.
\bjournal{The Annals of Probability}
\bvolume{26}
\bpages{213--252}.
\bdoi{10.1214/aop/1022855417}
\end{barticle}
\endbibitem

\bibitem{LGM11}
\begin{barticle}[author]
\bauthor{\bsnm{Le~Gall},~\bfnm{Jean-Fran{\c c}ois}\binits{J.-F.}} \AND
  \bauthor{\bsnm{Miermont},~\bfnm{Gr{\'e}gory}\binits{G.}}
(\byear{2011}).
\btitle{Scaling limits of random planar maps with large faces}.
\bjournal{The Annals of Probability}
\bvolume{39}
\bpages{1 -- 69}.
\bdoi{10.1214/10-AOP549}
\end{barticle}
\endbibitem

\bibitem{marzouk19}
\begin{barticle}[author]
\bauthor{\bsnm{Marzouk},~\bfnm{Cyril}\binits{C.}}
(\byear{2022}).
\btitle{{On scaling limits of random trees and maps with a prescribed degree
  sequence}}.
\bjournal{{Annales Henri Lebesgue}}
\bvolume{5}
\bpages{317-386}.
\bdoi{10.5802/ahl.125}
\end{barticle}
\endbibitem

\bibitem{marzouk_similaire}
\begin{barticle}[author]
\bauthor{\bsnm{Marzouk},~\bfnm{Cyril}\binits{C.}}
(\byear{2022}).
\btitle{Scaling limits of random looptrees and bipartite plane maps with
  prescribed large faces}.
\bjournal{Preprint available on arXiv}.
\bnote{arXiv:2202.08666}.
\end{barticle}
\endbibitem

\bibitem{stablemoments2}
\begin{barticle}[author]
\bauthor{\bsnm{Matsui},~\bfnm{Muneya}\binits{M.}} \AND
  \bauthor{\bsnm{Pawlas},~\bfnm{Zbyněk}\binits{Z.}}
(\byear{2016}).
\btitle{Fractional absolute moments of heavy tailed distributions}.
\bjournal{Brazilian Journal of Probability and Statistics}
\bvolume{30}
\bpages{272--298}.
\end{barticle}
\endbibitem

\bibitem{miermont_surveymaps}
\begin{bunpublished}[author]
\bauthor{\bsnm{Miermont},~\bfnm{Grégory}\binits{G.}}
(\byear{2014}).
\btitle{Aspects of random maps}.
\bnote{Preliminary draft}.
\end{bunpublished}
\endbibitem

\bibitem{PAULIN1989197}
\begin{barticle}[author]
\bauthor{\bsnm{Paulin},~\bfnm{Fr\'ed\'eric}\binits{F.}}
(\byear{1989}).
\btitle{The {G}romov topology on {R}-trees}.
\bjournal{Topology and its Applications}
\bvolume{32}
\bpages{197 - 221}.
\bdoi{https://doi.org/10.1016/0166-8641(89)90029-1}
\end{barticle}
\endbibitem

\bibitem{richier18}
\begin{barticle}[author]
\bauthor{\bsnm{Richier},~\bfnm{Lo\"{\i}c}\binits{L.}}
(\byear{2018}).
\btitle{Limits of the boundary of random planar maps}.
\bjournal{Probability Theory and Related Fields}
\bvolume{172}
\bpages{789--827}.
\bdoi{10.1007/s00440-017-0820-y}
\bmrnumber{3877547}
\end{barticle}
\endbibitem

\bibitem{straf1972}
\begin{binproceedings}[author]
\bauthor{\bsnm{Straf},~\bfnm{Miron~L.}\binits{M.~L.}}
(\byear{1972}).
\btitle{Weak convergence of stochastic processes with several parameters}.
In \bbooktitle{Proceedings of the Sixth Berkeley Symposium on Mathematical
  Statistics and Probability, Volume 2: Probability Theory}
\bpages{187--221}.
\bpublisher{University of California Press}, \baddress{Berkeley, California}.
\end{binproceedings}
\endbibitem

\bibitem{Szymanski1987}
\begin{barticle}[author]
\bauthor{\bsnm{Szymański},~\bfnm{Jerzy}\binits{J.}}
(\byear{1987}).
\btitle{{On a Nonuniform Random Recursive Tree}}.
\bjournal{North-Holland Mathematics Studies}
\bvolume{144}
\bpages{297-306}.
\end{barticle}
\endbibitem

\bibitem{willard2004general}
\begin{bbook}[author]
\bauthor{\bsnm{Willard},~\bfnm{Stephen}\binits{S.}}
(\byear{2004}).
\btitle{General Topology}.
\bseries{Addison-Wesley series in mathematics}.
\bpublisher{Dover Publications}.
\end{bbook}
\endbibitem

\end{thebibliography}


\end{document}